\documentclass[11pt,a4paper,oneside]{amsart}

\usepackage[T1]{fontenc}
\usepackage[utf8]{inputenc}
\usepackage{lmodern}
\usepackage{microtype}
\microtypesetup{expansion=false}
\usepackage{amsmath,amsfonts,amssymb,mathtools}
\usepackage{amsthm}
\usepackage{enumitem}
\usepackage{xcolor}
\usepackage[margin=1.02in]{geometry}
\usepackage[colorlinks=true,linkcolor=blue!55!black,citecolor=blue!55!black,urlcolor=blue!55!black]{hyperref}
\hypersetup{
 pdftitle={Richardson volume models for skew Schur and skew Schur P/Q-functions},
 pdfauthor={Khai-Hoan Nguyen-Dang; Appendix A contributed by Zhenpeng Wang},
 pdfsubject={Types A, C, and D Richardson varieties, realizable volume polynomials, and Lorentzian coefficient arrays},
 pdfkeywords={skew Schur polynomial, Schur P-function, Schur Q-function, Richardson variety, Lagrangian Grassmannian, spinor variety, characteristic two, realizable volume polynomial, Lorentzian polynomial}
}

\allowdisplaybreaks
\numberwithin{equation}{section}
\setlist{itemsep=2pt,topsep=4pt}

\newtheorem{theorem}{Theorem}[section]
\newtheorem{proposition}[theorem]{Proposition}
\newtheorem{corollary}[theorem]{Corollary}
\newtheorem{lemma}[theorem]{Lemma}
\theoremstyle{definition}
\newtheorem{definition}[theorem]{Definition}
\newtheorem{example}[theorem]{Example}
\theoremstyle{remark}
\newtheorem{remark}[theorem]{Remark}

\DeclareMathOperator{\Gr}{Gr}
\DeclareMathOperator{\LG}{LG}
\DeclareMathOperator{\Supp}{Supp}
\DeclareMathOperator{\Newt}{Newt}
\DeclareMathOperator{\conv}{conv}
\DeclareMathOperator{\diag}{diag}
\DeclareMathOperator{\Cov}{Cov}
\DeclareMathOperator{\Var}{Var}
\DeclareMathOperator{\rank}{rank}

\newcommand{\N}{\mathbb N}
\newcommand{\Z}{\mathbb Z}
\newcommand{\Q}{\mathbb Q}
\newcommand{\R}{\mathbb R}
\newcommand{\C}{\mathbb C}
\newcommand{\cN}{\mathcal N}
\newcommand{\cO}{\mathcal O}
\newcommand{\cS}{\mathcal S}
\newcommand{\cQ}{\mathcal Q}
\newcommand{\cE}{\mathcal E}
\newcommand{\cF}{\mathcal F}
\newcommand{\cG}{\mathcal G}
\newcommand{\cK}{\mathcal K}
\newcommand{\cU}{\mathcal U}
\newcommand{\one}{\mathbf 1}
\newcommand{\abs}[1]{\lvert#1\rvert}
\newcommand{\angles}[1]{\left\langle #1\right\rangle}
\newcommand{\psd}{\preceq_{\mathrm{psd}}}

\title[Richardson volume models in types A, C, and D]{Richardson volume models for skew Schur and skew Schur $P/Q$-functions}

\author{Khai-Hoan Nguyen-Dang,\\With an appendix by Zhenpeng Wang}
\address{Morningside Center of Mathematics, Chinese Academy of Sciences, Beijing 100190, China}
\email{khaihoann@gmail.com}

\date{}

\begin{document}

\begin{abstract}
We identify ordinary skew Schur polynomials and skew Schur $P$-functions as top-degree total-Chern intersection polynomials on Richardson varieties in ordinary and Lagrangian Grassmannians.  We then obtain that
\[
 \cN(s_{\lambda/\mu}),\qquad
 \cN(P_{\lambda/\mu}),\qquad
 \cN(Q_{\lambda/\mu})
\]
are realizable volume polynomials.  This settles the skew-Schur and Schur-$P$ Lorentzian conjectures of Huh--Matherne--M\'esz\'aros--St.~Dizier and strengthens the latter to arbitrary skew $P/Q$-functions.

The constructions extend to cycle transforms attached to arbitrary irreducible subvarieties of ordinary and Lagrangian Grassmannians.  Their realizable-volume interpretation yields reverse Khovanskii--Teissier and Lorentzian Hodge--Riemann inequalities for ordinary and shifted tableau multiplicities; exact ordinary skew-Schur support permutahedra and extremal coefficients; the known straight shifted support polytopes with their vertex coefficients; implicit exact permutahedra for arbitrary skew $P/Q$-functions; and weighted-aggregation, covariance, and two-row Littlewood--Richardson consequences.

An appendix by Zhenpeng Wang constructs the dual type~$D$ spinor cycle transform and a direct Richardson realization of $Q_{\lambda/\mu}$.  In characteristic two, compatible very special isogenies induce finite flat radicial morphisms between the ambient Lagrangian and spinor models.  Their restrictions to the corresponding Richardson varieties have degrees $2^{\ell(\lambda)-\ell(\mu)}$ and $2^{|\lambda|-|\mu|-\ell(\lambda)+\ell(\mu)}$, and a regular complete-intersection bridge explains the difference between the two projective-bundle shifts.
\end{abstract}

\maketitle

\section{Introduction}\label{sec:introduction}

Schur polynomials connect symmetric-function theory, polynomial representations of general linear groups, tableaux, and the intersection theory of Grassmannians.  For partitions $\mu\subseteq\lambda$, the skew Schur polynomial
\[
 s_{\lambda/\mu}(x_1,\ldots,x_n)
 =\sum_T x^{\operatorname{cont}(T)}
\]
is the character of the corresponding skew Schur module; see Fulton~\cite[Chapters~1 and~9]{FultonYT} and Akin--Buchsbaum--Weyman~\cite{ABW82}.  The parallel projective-representation theory is governed by marked shifted tableaux and the Schur $P$- and $Q$-functions of strict partitions; see Macdonald~\cite[Chapter~III, Section~8]{Macdonald} and Stembridge~\cite{Stembridge89}.

For a polynomial $F(x)=\sum_\alpha c_\alpha x^\alpha$, write
\[
 \cN(F):=\sum_\alpha c_\alpha\frac{x^\alpha}{\alpha!},
 \qquad
 \alpha!:=\prod_i\alpha_i!.
\]
Lorentzian polynomials, introduced by Br\"and\'en and Huh~\cite{BH20}, provide a common framework for Hodge--Riemann relations, Alexandrov--Fenchel inequalities, complete log-concavity, and discrete convexity.  Huh--Matherne--M\'esz\'aros--St.~Dizier proved that $\cN(s_\lambda)$ is Lorentzian for every straight shape and, in fact, realized it as a projective volume polynomial~\cite[Theorem~3 and pp.~4413--4414]{HMMSD22}; see also Huh~\cite[Example~3.6]{HuhVolume26}.  They asked whether $\cN(s_{\lambda/\mu})$ is Lorentzian for every skew shape and whether $\cN(P_\lambda)$ is Lorentzian for every strict partition~\cite[Conjectures~19 and~20]{HMMSD22}.

A homogeneous polynomial $f$ of degree $d$ is a \emph{realizable volume polynomial over $\C$} if
\[
 f(x)=\frac{q}{d!}\int_X(x_1D_1+\cdots+x_nD_n)^d
\]
for some $q\in\Q_{\geq0}$, an irreducible projective $d$-fold $X/\C$, and semiample Cartier divisors $D_1,\ldots,D_n$ on $X$; see Grund--Huh--Micha\l ek--S\"uss--Wang~\cite[Definition~1.1]{GHMSW25} and Huh~\cite[Definition~3.1]{HuhVolume26}.  This class is strictly smaller than the Lorentzian cone in general~\cite[Section~3.2]{HuhVolume26}.

\subsection{Main results}

Our principal theorem resolves both conjectures in the stronger volume-realization form and extends the type~$C$ statement from straight to arbitrary skew shifted shapes.

\begin{theorem}[Main volume-realization theorem]\label{thm:intro-volume}
Let $\theta$ be an ordinary skew shape, let $\vartheta$ be a strict skew shape, and let $n\geq1$.  Then
\[
 \cN(s_\theta(x_1,\ldots,x_n)),\qquad
 \cN(P_\vartheta(x_1,\ldots,x_n)),\qquad
 \cN(Q_\vartheta(x_1,\ldots,x_n))
\]
are realizable volume polynomials over $\C$.  If one of these polynomials is nonzero, its realization may be chosen on a smooth irreducible projective variety with semiample divisors.
\end{theorem}

The ordinary and shifted assertions are proved in Theorems~\ref{thm:main-A} and~\ref{thm:main-PQ}, respectively; the smooth-realization clause follows from Theorem~\ref{thm:total-chern-volume}.

The theorem is part of a more uniform statement.  Let
\[
 G_A=\Gr(r,r+c),
 \qquad
 0\longrightarrow\cS\longrightarrow\C^{r+c}\otimes\cO_{G_A}
 \longrightarrow\cQ\longrightarrow0
\]
be the ordinary Grassmannian and its universal sequence.  For an irreducible $d$-dimensional subvariety $X\subseteq G_A$, set
\[
 \Theta^A_{X,n}(x)
 :=\sum_{\substack{\nu\subseteq c^r\\|\nu|=d}}
 \left(\int_X s_\nu(\cS^\vee|_X)\right)s_\nu(x_1,\ldots,x_n).
\]
Likewise, let $G_C=\LG(N,2N)$, write $\sigma_\nu$ for its Schubert classes indexed by strict partitions $\nu\subseteq\rho_N=(N,N-1,\ldots,1)$, and set
\[
 \Theta^C_{Z,n}(x)
 :=\sum_{\substack{\nu\subseteq\rho_N\text{ strict}\\|\nu|=d}}
 \left(\int_Z\sigma_\nu|_Z\right)P_\nu(x_1,\ldots,x_n)
\]
for an irreducible $d$-dimensional subvariety $Z\subseteq G_C$.

\begin{theorem}[Richardson cycle transforms]\label{thm:intro-cycle-transforms}
For every irreducible projective $X\subseteq G_A$ and $Z\subseteq G_C$, the polynomials
\[
 \cN(\Theta^A_{X,n})
 \qquad\text{and}\qquad
 \cN(\Theta^C_{Z,n})
\]
are realizable volume polynomials.  Complementary Schubert varieties recover straight Schur and Schur $P$-functions, while the type~$A$ Richardson variety associated with $\lambda/\mu$ recovers $s_{\lambda/\mu}$ and the type~$C$ Richardson variety recovers $P_{\lambda/\mu}$ once $N\geq\lambda_1$.
\end{theorem}

This is Theorems~\ref{thm:typeA-transform} and~\ref{thm:typeC-transform}.  The skew specializations are the Richardson identities in Theorems~\ref{thm:main-A} and~\ref{thm:main-PQ}.

The type~$C$ realization has a dual spinor companion.

\begin{theorem}[Type $D$ spinor realization; Wang]\label{thm:intro-typeD}
Fix $N\geq1$, and let
\[
 G_D=\operatorname{OG}^{+}(N+1,2N+2)
\]
be a spinor component, let $\mathcal Q_D$ be its universal quotient bundle,
and let $\tau_\nu$ denote the Schubert class indexed by the strict partition
$\nu\subseteq\rho_N$.  If $X\subseteq G_D$ is an irreducible projective
$d$-fold, then
\[
 \Theta^D_{X,n}(x)
 :=\sum_{\substack{\nu\subseteq\rho_N\text{ strict}\\|\nu|=d}}
   \left(\int_X\tau_\nu|_X\right)Q_\nu(x_1,\ldots,x_n)
\]
satisfies
\[
 [x^\alpha]\Theta^D_{X,n}
 =\int_X\prod_{i=1}^n c_{\alpha_i}(\mathcal Q_D|_X)
 \qquad(|\alpha|=d),
\]
and $\cN(\Theta^D_{X,n})$ is a realizable volume polynomial.  For strict
$\mu\subseteq\lambda$ and $N\geq\max\{1,\lambda_1\}$, the spinor
Richardson variety $R^D_{\lambda/\mu}$ gives
\[
 Q_{\lambda/\mu}(x)
 =\int_{R^D_{\lambda/\mu}}
   \prod_{i=1}^n c_{x_i}(\mathcal Q_D).
\]
\end{theorem}

\begin{theorem}[Characteristic-two comparison; Wang]
\label{thm:intro-char2-comparison}
Let $\mu\subseteq\lambda\subseteq\rho_N$ be strict partitions, put
\[
 d=|\lambda|-|\mu|,
 \qquad
 r=\ell(\lambda)-\ell(\mu),
\]
and work over an algebraically closed field of characteristic two.  For
$N\geq2$, the pinning-compatible special morphisms restrict to finite,
surjective, radicial maps
\[
 R^D_{\lambda/\mu}\xrightarrow{\ \beta_{\lambda/\mu}\ }
 R^C_{\lambda/\mu},
 \qquad
 R^C_{\lambda/\mu}\xrightarrow{\ \alpha_{\lambda/\mu}\ }
 R^D_{\lambda/\mu};
\]
for $N=1$, the same notation refers to the maps obtained after identifying
both ambient spaces with $\mathbb P^1$, with $\alpha$ an isomorphism and
$\beta$ relative Frobenius.  In either case, the two restricted maps have
degrees $2^r$ and $2^{d-r}$, respectively.  The corresponding normalized
projective-bundle intersection polynomials $V_C$ and $V_D$, defined in
Theorem~\ref{appA:thm:projective-bundle-bridge}, satisfy
\[
 \partial_{x_1}\cdots\partial_{x_n}V_D(x)=2^rV_C(x).
\]
\end{theorem}

These statements are proved in Appendix~\ref{appA:main}, contributed by Zhenpeng Wang; see Theorem~\ref{appA:thm:D-cycle-transform}, Theorem~\ref{appA:thm:D-Richardson}, Corollary~\ref{appA:cor:D-projective-bundle}, Theorem~\ref{appA:thm:Richardson-degrees}, and Theorem~\ref{appA:thm:projective-bundle-bridge}.  The type~$D$ model is not needed for the type~$C$ proof; it supplies a dual integral Schubert normalization and a geometric explanation of the $P/Q$ factor.

\subsection{Geometric method}

The proof combines the Richardson coefficient identifications with two general mechanisms.  Cid-Ruiz proves that a homogenization of the total-Chern polynomial of globally generated bundles under a proper map to projective space is denormalized Lorentzian~\cite[Proposition~6.2]{CidRuiz26}; after choosing a projective embedding and specializing the homogenizing variable to zero, this gives an independent direct Lorentzian proof for the top-degree polynomial considered here.  Grund--Huh--Micha\l ek--S\"uss--Wang prove that realizable covolume polynomials are precisely the differential operators preserving realizable volume polynomials~\cite[Theorem~1.3]{GHMSW25}.  Theorem~\ref{thm:total-chern-volume} first constructs the stronger realizable-volume model and then records the direct Cid-Ruiz route. Here, the Richardson coefficient identities, the cycle transforms, and the type~$C/D$ comparison are the substantive new geometric input.

The skew-specific input is a pair of Richardson coefficient formulas.  In type~$A$, repeated Pieri multiplication and complementary Schubert duality give
\begin{equation}\label{eq:intro-A-coeff}
 [x^\alpha]s_{\lambda/\mu}
 =\int_{R^A_{\lambda/\mu}}
   \prod_i h_{\alpha_i}(\cS^\vee).
\end{equation}
Equivalently,
\begin{equation}\label{eq:intro-A-total}
 s_{\lambda/\mu}(x)
 =\int_{R^A_{\lambda/\mu}}
   \prod_i c_{x_i}(\cQ).
\end{equation}
In type~$C$, the Schur $P/Q$ Cauchy identity and a stable Schubert homomorphism give
\begin{equation}\label{eq:intro-C-coeff}
 [x^\alpha]P_{\lambda/\mu}
 =\int_{R^C_{\lambda/\mu}}
   \prod_i c_{\alpha_i}(\cE),
\end{equation}
and hence
\begin{equation}\label{eq:intro-C-total}
 P_{\lambda/\mu}(x)
 =\int_{R^C_{\lambda/\mu}}
   \prod_i c_{x_i}(\cE).
\end{equation}
The type~$C$ passage is not the false literal substitution $Q_\lambda(\cE)=\sigma_\lambda$.  Proposition~\ref{prop:stable-map} first sends the ordinary Schur $Q$-function $Q_\nu$ to the modified $\widetilde Q_\nu$-representative of the Lagrangian Schubert class, following Kresch--Tamvakis~\cite{KT03}; this is the normalization-sensitive step in the shifted argument.

Appendix~\ref{appA:main} gives the dual type~$D$ identity
\begin{equation}\label{eq:intro-D-coeff}
 [x^\alpha]Q_{\lambda/\mu}
 =\int_{R^D_{\lambda/\mu}}
   \prod_i c_{\alpha_i}(\mathcal Q_D),
\end{equation}
where the integral spinor specialization sends $P_\nu$ to the Schubert class $\tau_\nu$.  Thus the Lagrangian and spinor coefficient extractions have the same formal shape, but the $Q_\nu\mapsto\sigma_\nu$ and $P_\nu\mapsto\tau_\nu$ normalizations produce the dual skew outputs $P_{\lambda/\mu}$ and $Q_{\lambda/\mu}$.

After choosing projective embeddings of the Richardson varieties, Formulas~\eqref{eq:intro-A-total} and~\eqref{eq:intro-C-total} also give direct Lorentzian proofs through Cid-Ruiz's homogenized theorem and specialization of its homogenizing variable.  To obtain the stronger realizable-volume statement, we lift the higher Chern classes to powers of tautological divisors on iterated projective bundles.  In type~$A$, if $d=|\lambda|-|\mu|$ and $D_A=d+n(r-1)$, then
\begin{equation}\label{eq:intro-A-volume}
 \frac1{D_A!}\int_{Y^A_{\lambda/\mu,n}}
 (x_1\xi_1+\cdots+x_n\xi_n)^{D_A}
 =\cN\!\left((x_1\cdots x_n)^{r-1}s_{\lambda/\mu}(x)\right).
\end{equation}
The type~$C$ construction has the identical form with rank $N$, dimension $D_C=d+n(N-1)$, and $P_{\lambda/\mu}$ in place of $s_{\lambda/\mu}$.  The spinor quotient bundle in Appendix~\ref{appA:main} has rank $N+1$, so its direct $Q_{\lambda/\mu}$ model has shift $N$.  In characteristic two, the exact sequence
\[
 0\longrightarrow L^\vee\otimes\mathcal O
 \longrightarrow\mathcal Q_D
 \longrightarrow\beta^*\mathcal E
 \longrightarrow0
\]
realizes the type~$C$ projective bundle as a tautological codimension-one divisor in each type~$D$ factor and accounts for this one-step difference.

The monomial differential operators
\[
 \prod_i\partial_i^{r-1}
 \qquad\text{and}\qquad
 \prod_i\partial_i^{N-1}
\]
remove the type~$A$ and type~$C$ shifts, respectively.  The operator theorem of Grund--Huh--Micha\l ek--S\"uss--Wang then proves Theorem~\ref{thm:intro-volume}.  Together, these steps prove the volume-realization theorem.  The comparison with earlier work and the precise novelty claims are collected in Subsection~\ref{subsec:recent-work-novelty}.

\subsection{Coefficient inequalities and discrete convexity}

Among the following consequences, the reverse Khovanskii--Teissier inequalities use the stronger realizable-volume structure, whereas the Hodge--Riemann, root-line, support, and dominance statements already follow from Lorentzianity and symmetry.  The following statement collects the principal general consequences; its detailed forms are Theorems~\ref{thm:master-RKT},~\ref{thm:master-HR}, and~\ref{thm:dominance-monotonicity}.

\begin{theorem}[Coefficient inequalities]\label{thm:intro-coefficient-package}
Let
\[
 F(x)=\sum_{|\alpha|=d}c_\alpha x^\alpha,
 \qquad
 f=\cN(F),
\]
and suppose that $f$ is a realizable volume polynomial.
\begin{enumerate}[label=\textup{(\roman*)}]
\item If $\beta\in\N^n$, $m=d-|\beta|\geq0$, $0\leq e\leq m$, and $i,j,k\in[n]$, then
\begin{equation}\label{eq:intro-RKT}
 \binom me
 c_{\beta+(m-e)e_i+e\,e_j}
 c_{\beta+e\,e_i+(m-e)e_k}
 \geq
 c_{\beta+m e_i}
 c_{\beta+e\,e_j+(m-e)e_k}.
\end{equation}
\item If $d\geq2$ and $|\beta|=d-2$, the matrix
\[
 H^\beta=(c_{\beta+e_i+e_j})_{i,j=1}^n
\]
has at most one positive eigenvalue, and exactly one when it is nonzero.  For every nonempty $I\subseteq[n]$,
\[
 (-1)^{|I|-1}\det H_I^\beta\geq0.
\]
In particular, if $|\alpha|=d$, $i\neq j$, and $\alpha_i,\alpha_j\geq1$, then
\[
 c_\alpha^2\geq
 c_{\alpha+e_i-e_j}c_{\alpha-e_i+e_j}.
\]
The nonzero coefficients on each affine root line form a log-concave unimodal interval.
\item If $F$ is symmetric and $\rho\unlhd\tau$ are partitions of $d$ with at most $n$ parts, then
\[
 c_\rho\geq c_\tau.
\]
Consequently, the balanced partition maximizes the partition-indexed coefficients, while the dominance-maximal support partition minimizes the positive partition-indexed coefficients.
\end{enumerate}
\end{theorem}

Part~\textup{(i)} is a reverse Khovanskii--Teissier inequality for the derivatives of $f$.  Parts~\textup{(ii)} and~\textup{(iii)} come from Lorentzian Hodge--Riemann theory and symmetry.  The dominance statement is also a general theorem of Chin--Qin~\cite[Theorem~3.16]{ChinQin25}; we retain a short root-line proof to make its tableau meaning transparent.  Corollaries~\ref{cor:ordinary-tableau-package},~\ref{cor:shifted-tableau-package}, and~\ref{cor:tableau-dominance} specialize this package to ordinary and shifted tableau multiplicities.

\subsection{Two-row Littlewood--Richardson consequences}

The root-direction inequalities have particularly concrete two-variable consequences for ordinary and shifted Littlewood--Richardson coefficients.  They translate the general root-line inequalities into log-concavity statements for cumulative two-row multiplicities.

\begin{theorem}[Two-row Littlewood--Richardson consequences]\label{thm:intro-two-row}
\begin{enumerate}[label=\textup{(\roman*)}]
\item Write
\[
 s_{\lambda/\mu}=\sum_{\nu\vdash d}c_{\mu\nu}^{\lambda}s_\nu,
 \qquad
 A_r:=\sum_{j=0}^r c_{\mu,(d-j,j)}^{\lambda}
 \quad(0\leq r\leq\lfloor d/2\rfloor).
\]
Then
\[
 A_r=[x^{d-r}y^r]s_{\lambda/\mu}(x,y),
 \qquad
 A_r^2\geq A_{r-1}A_{r+1}
\]
for $1\leq r<\lfloor d/2\rfloor$.
\item Define the shifted Littlewood--Richardson coefficients by
\[
 Q_\mu Q_\nu=\sum_\lambda g_{\mu\nu}^{\lambda}Q_\lambda.
\]
For $\nu_0=(d)$ and $\nu_j=(d-j,j)$, put $g_j=g_{\mu,\nu_j}^{\lambda}$ and
\[
 B_r=[x^{d-r}y^r]P_{\lambda/\mu}(x,y).
\]
Then $B_r=B_{d-r}$,
\[
 B_0=g_0,
 \qquad
 B_r=2\sum_{j=0}^{r-1}g_j+g_r
 \quad\left(1\leq r\leq\left\lfloor\frac{d-1}{2}\right\rfloor\right),
\]
with the corresponding central formula when $d$ is even, and
\[
 B_r^2\geq B_{r-1}B_{r+1}
 \qquad(0<r<d).
\]
Thus the weighted cumulative two-row shifted Littlewood--Richardson sequence has no internal zeros and is log-concave.
\end{enumerate}
\end{theorem}

These are Corollary~\ref{cor:ordinary-two-row-LR} and Theorem~\ref{thm:shifted-LR-logconcavity}.  The factor two in the shifted cumulative sum reflects the interior coefficients of a two-variable Schur $P$-function.

The volume models also organize several other consequences, proved in
Sections~\ref{sec:applications}--\ref{sec:concrete-C} and not used in the
proofs of the main volume-realization theorems.  Realizable-covolume operators
produce Schur- and Schubert-differential descendants, iterated factorial
normalizations, and full multiaffine polarizations whose supports are bases of
algebraic matroids over every field of characteristic zero
(Theorem~\ref{thm:operator-descendants}; compare
\cite[Theorem~1.3, Corollary~2.4, and Propositions~2.10, 4.1, 4.2,
and~5.4]{GHMSW25}).  For every nonzero symmetric transform, $M$-convexity
identifies the support with all lattice points of an exact permutahedron,
gives a support polymatroid linearly representable over every infinite field,
and makes the dominant support partition additive under products, with the
mixed integer-decomposition property and multiplicative vertex coefficients
(Theorem~\ref{thm:symmetric-permutahedron}).  In type~$A$ this specializes to
the explicit column-length permutahedron and unit vertex coefficients of
Theorem~\ref{thm:ordinary-skew-support} and
Corollary~\ref{cor:ordinary-maximal-constituent}; these ordinary support facts
recover results of McNamara and Monical--Tokcan--Yong
\cite[Proposition~3.1]{McNamara08}\cite[Corollary~5.10]{MTY19}.
For straight Schur $P/Q$-functions, Corollary~\ref{cor:straight-P-support}
recovers~\cite[Proposition~3.5]{MTY19}, whereas
Corollary~\ref{cor:shifted-tableau-package} gives an implicit exact
permutahedron for every nonzero skew $P/Q$ specialization.  Nonnegative
weighted aggregations remain Lorentzian, two-block aggregations are
ultra-log-concave, and the factorially tilted content law satisfies
$\Cov_z(\alpha)\preceq_{\mathrm{psd}}\diag(\mathbb E_z\alpha)$
(Theorem~\ref{thm:block-covariance}).  Finally, mixed products of ordinary
skew Schur and skew $P/Q$-functions again have realizable factorial
normalizations (Corollary~\ref{cor:mixed-products}).

\subsection{Relation to recent work and novelty}\label{subsec:recent-work-novelty}

\subsubsection{General Lorentzian and volume-polynomial mechanisms.}
Cid-Ruiz's theorem gives a direct Lorentzian implication from either Richardson total-Chern identity after a projective embedding and specialization of the homogenizing variable~\cite[Proposition~6.2]{CidRuiz26}.  We do not claim a new abstract total-Chern-to-Lorentzian principle.  The operator theorem of Grund--Huh--Micha\l ek--S\"uss--Wang upgrades the explicit shifted projective-bundle models to realizations of the unshifted factorial normalizations~\cite[Definition~1.2 and Theorem~1.3]{GHMSW25}; formal descendants such as truncations, polarizations, and iterated normalizations come from that theory.  Chin--Qin prove, in the stable range $n\geq d$, that the partition-indexed support of a degree-$d$ symmetric Lorentzian polynomial is a dominance interval and that its Newton polytope is the corresponding permutahedron; they also establish dominance monotonicity of normalized monomial coefficients~\cite[Theorems~3.7, 3.9, and~3.16]{ChinQin25}.  Our direct finite-variable arguments additionally identify the polymatroid rank function, give a linear representation over every infinite field, and connect these abstract statements to the ordinary and shifted tableau arrays at hand.

\subsubsection{An independent proof of skew-Schur Lorentzianity.}
After the first version of this paper appeared, Zhang independently proved that $\cN(s_{\lambda/\mu})$ is Lorentzian by realizing finite-variable skew Schur polynomials as specializations of Schubert polynomials, using dually Lorentzian closure, and applying rectangular complementation~\cite{Zhang26}.  That method and the Richardson--total-Chern method here are independent.  Zhang's theorem establishes the Lorentzian conclusion and root-direction log-concavity for skew Kostka numbers; the present construction gives the stronger realizable-volume statement, the type~$A$ cycle transform, and the shifted type~$C$ and type~$D$ $P/Q$ models.

\subsubsection{Richardson and quasisymmetric geometry.}
An--Tung--Zhang prove that Postnikov--Stanley polynomials are degree polynomials of Richardson varieties in arbitrary Weyl type and deduce their Lorentzianity~\cite[Proposition~3.3 and Theorem~1.2]{ATZ24}.  Their variables record nef divisor classes already living on the Richardson variety.  Here the coefficients are products of higher Chern classes, and the passage to divisors takes place on fiber products of projective bundles.  Nadeau--Spink--Tewari construct smooth toric Richardson varieties as iterated $\mathbb P^1$-bundles, relate their degree maps to quasisymmetric coinvariants and generalized Littlewood--Richardson coefficients, and geometrically interpret Hall pairings with ribbon skew Schur functions~\cite[Theorem~2.2 and Corollary~10.3]{NST24}.  Bergeron--Gagnon--Spink--Tewari subsequently construct the quasisymmetric Grassmannian, whose irreducible components are translates of ribbon-type toric Richardson varieties~\cite{BGST26}.  These works are predominantly type~$A$ and divisor-theoretic; the present Lagrangian construction instead uses the stable Schur-$Q$ homomorphism and total Chern classes to recover skew $P/Q$ coefficient arrays.  The contributed Appendix~\ref{appA:main} supplies the dual spinor specialization, based on the classical integral $\widetilde P$-Giambelli normalization, and compares the two isotropic geometries through the very special type~$B/C$ isogenies in characteristic two; this comparison is independent of the type~$C$ proof in the main text.  Thus the divisor-theoretic and total-Chern constructions use genuinely different geometric inputs.

\subsubsection{Coefficient inequalities versus shape inequalities.}
Lam--Postnikov--Pylyavskyy's skew Schur log-concavity inequalities~\cite{LPP07} and the skew Ahlswede--Daykin--Schur inequalities of Chan--Chen--Pak--Soskin~\cite[Theorem~1.7]{CCPS26} compare different skew shapes in the Schur-positivity order.  Very recently, Le--Nguyen introduced skew hive and skew skep models and proved a skew Schur log-concavity theorem extending the Lam--Postnikov--Pylyavskyy conjecture from straight to skew shapes; the straight-shape conjecture was proved by Speyer using skeps~\cite{Speyer26,LeNguyen26}.  Our inequalities constrain the monomial coefficient array of one fixed symmetric function.  No implication between the two kinds of statements is used.  Indeed, Schur positivity alone does not force root-direction log-concavity: in two variables,
\[
 s_{(4)}+s_{(3,1)}+8s_{(2,2)}
 =x^4+2x^3y+10x^2y^2+2xy^3+y^4,
\]
and $2^2<1\cdot10$.

Accordingly, the principal new contributions of the main text are the type~$A$ Richardson--Pieri identity, the type~$C$ stable-$Q$ Richardson identity, the two cycle transforms, realizable-volume polynomiality for arbitrary skew Schur $P/Q$-functions, and the resulting reverse Khovanskii--Teissier and shifted-tableau inequalities.  Appendix~\ref{appA:main}, contributed by Zhenpeng Wang, adds the type~$D$ cycle transform and direct spinor Richardson model for skew $Q$-functions, together with the characteristic-two Richardson-degree and projective-bundle comparison.  The general total-Chern, differential-operator, Lorentzian, and symmetric-support mechanisms are attributed to the cited general theories, while the ordinary skew-Schur and straight Schur $P/Q$ support statements are presented as geometric recoveries of known polyhedral results.

\subsection*{Organization.}
Section~\ref{sec:preliminaries} fixes conventions, and Section~\ref{sec:total-chern} proves the general total-Chern volume theorem.  Sections~\ref{sec:typeA} and~\ref{sec:typeC} establish the type~$A$ and type~$C$ cycle transforms and their skew specializations.  Section~\ref{sec:applications} develops permanence, reverse Khovanskii--Teissier inequalities, the Lorentzian Hodge--Riemann package, dominance monotonicity, symmetric supports, and probabilistic consequences.  Sections~\ref{sec:concrete-A} and~\ref{sec:concrete-C} specialize the general results to ordinary and shifted tableaux, including the two-row Littlewood--Richardson applications and mixed products.  Appendix~\ref{appA:main}, contributed by Zhenpeng Wang, develops the type~$D$ spinor companion and its characteristic-two bridge to the type~$C$ model.

\subsection*{Acknowledgements.}
We thank Dang Tuan Hiep for useful discussions.  We are grateful to Zhenpeng Wang for contributing Appendix~\ref{appA:main} and for discussions concerning the type~$D$ spinor companion.  We appreciate the support of the Morningside Center of Mathematics, Chinese Academy of Sciences.

\section{Preliminaries}\label{sec:preliminaries}

Throughout, $\N=\Z_{\geq0}$ and $[n]=\{1,\ldots,n\}$.  For a multi-index $\alpha\in\N^n$, set
\[
 |\alpha|=\sum_i\alpha_i,
 \qquad
 \alpha!=\prod_i\alpha_i!,
 \qquad
 x^{[\alpha]}:=\frac{x^\alpha}{\alpha!},
 \qquad
 \alpha(A):=\sum_{i\in A}\alpha_i\quad(A\subseteq[n]).
\]
For $F(x)=\sum_\alpha c_\alpha x^\alpha$, define
\begin{equation}\label{eq:normalization}
 \cN(F):=\sum_\alpha c_\alpha x^{[\alpha]},
 \qquad
 \Supp(F):=\{\alpha:c_\alpha\neq0\},
 \qquad
 \Newt(F):=\conv(\Supp(F)).
\end{equation}
The zero polynomial is included in every closed cone considered below.

For partitions $\nu$ and $\kappa$ of the same integer, padded by zeros to a common length, write $\nu\unlhd\kappa$ if
\[
 \sum_{i=1}^p\nu_i\leq\sum_{i=1}^p\kappa_i
 \qquad\text{for every }p;
\]
this is the dominance order.  For a partition $\kappa$ of length at most $n$, its permutahedron is
\[
 \mathcal P_\kappa:=\conv\{\pi\kappa:\pi\in\mathfrak S_n\}
 \subseteq\R^n,
\]
where $\mathfrak S_n$ permutes coordinates.

\subsection{Lorentzian and volume polynomials}

A finite set $B\subseteq\N^n$ whose elements have one coordinate sum is \emph{$M$-convex} if, whenever $\alpha,\beta\in B$ and $\alpha_i>\beta_i$, there is $j$ with $\alpha_j<\beta_j$ such that
\[
 \alpha-e_i+e_j\in B,
 \qquad
 \beta+e_i-e_j\in B.
\]
This is the symmetric exchange axiom of discrete convex analysis; see Murota~\cite[Chapter~4]{Murota03}.

For degrees zero and one, every homogeneous polynomial with nonnegative coefficients is Lorentzian.  Let $d\geq2$.  A homogeneous polynomial $f\in\R[x_1,\ldots,x_n]$ of degree $d$ with nonnegative coefficients is \emph{Lorentzian} if its support is $M$-convex and every partial derivative of total order $d-2$ has Hessian with at most one positive eigenvalue.  This is equivalent to the original closure definition of Br\"and\'en--Huh~\cite[Theorem~2.25]{BH20}.  We call $F$ \emph{denormalized Lorentzian} if $\cN(F)$ is Lorentzian.

\begin{definition}\label{def:rv}
A homogeneous polynomial $f\in\Q[x_1,\ldots,x_n]$ of degree $d$ is a \emph{realizable volume polynomial over $\C$} if
\begin{equation}\label{eq:rv-definition}
 f(x)=\frac{q}{d!}\int_X(x_1D_1+\cdots+x_nD_n)^d
\end{equation}
for some $q\in\Q_{\geq0}$, an irreducible projective $d$-dimensional variety $X/\C$, and semiample Cartier divisors $D_1,\ldots,D_n$ on $X$.  The case $q=0$ includes the zero polynomial.
\end{definition}
This is the convention of Grund--Huh--Micha\l ek--S\"uss--Wang~\cite[Definition~1.1]{GHMSW25} and Huh~\cite[Definition~3.1]{HuhVolume26}.  Every realizable volume polynomial is Lorentzian because semiample divisors are nef~\cite[Theorem~4.6]{BH20}.

We work in operational Chow cohomology on possibly singular varieties.  If $X$ is projective of pure dimension $d$ and $\theta\in A^d(X)$, then
\[
 \int_X\theta:=\deg(\theta\cap[X]).
\]
For a vector bundle $\cF$, write
\[
 c_z(\cF):=\sum_{m\geq0}c_m(\cF)z^m,
 \qquad
 \sum_{m\geq0}h_m(\cF)z^m:=c_{-z}(\cF)^{-1}.
\]

\subsection{Ordinary and shifted symmetric functions}

A partition is padded by zeros whenever convenient.  For $\mu\subseteq\lambda$, an ordinary semistandard tableau of shape $\lambda/\mu$ is weakly increasing along rows and strictly increasing down columns.  Its generating polynomial in the alphabet $[n]$ is the skew Schur polynomial $s_{\lambda/\mu}(x_1,\ldots,x_n)$.  We write
\begin{equation}\label{eq:ordinary-skew-kostka}
 K_{\lambda/\mu,\alpha}:=[x^\alpha]s_{\lambda/\mu}(x).
\end{equation}
Repeated Pieri multiplication gives
\begin{equation}\label{eq:ordinary-pieri-coeff}
 K_{\lambda/\mu,\alpha}
 =[s_\lambda]\,s_\mu h_{\alpha_1}\cdots h_{\alpha_n}.
\end{equation}

A strict partition has distinct positive parts, and $\ell(\lambda)$ denotes its number of positive parts.  We use Macdonald's marked shifted-tableau convention for Schur $P$- and $Q$-functions~\cite[Chapter~III, Section~8]{Macdonald}.  Their Hopf-algebraic skew versions will be recalled in Section~\ref{sec:typeC}.  All finite-variable specializations are understood, and a specialization that has no admissible tableaux is zero.

\section{A general total-Chern volume theorem}\label{sec:total-chern}

The following theorem is the common geometric mechanism behind the type~A and type~C constructions.  Its realizable-volume assertion follows from the explicit projective-bundle construction and the operator theorem of Grund--Huh--Micha\l ek--S\"uss--Wang~\cite[Theorem~1.3]{GHMSW25}, and its Lorentzian assertion then follows formally.  Cid-Ruiz~\cite[Proposition~6.2]{CidRuiz26} also gives an independent direct Lorentzian proof after projective homogenization.  We record the full argument because the exact volume identity is used repeatedly.

\begin{theorem}[Total-Chern volume theorem]\label{thm:total-chern-volume}
Let $X$ be an irreducible projective $d$-fold over $\C$, let $n\geq1$, and let $\cF_1,\ldots,\cF_n$ be globally generated vector bundles on $X$.  Define
\begin{equation}\label{eq:total-chern-polynomial}
 T_{X,\cF_\bullet}(x)
 :=\sum_{|\alpha|=d}
 \left(\int_X\prod_{i=1}^n c_{\alpha_i}(\cF_i)\right)x^\alpha.
\end{equation}
Then:
\begin{enumerate}[label=\textup{(\roman*)}]
\item $T_{X,\cF_\bullet}$ is denormalized Lorentzian.
\item $\cN(T_{X,\cF_\bullet})$ is a realizable volume polynomial.
\item More explicitly, choose generating sequences
\begin{equation}\label{eq:generating-sequences}
 0\longrightarrow\cK_i\longrightarrow W_i\otimes\cO_X
 \longrightarrow\cF_i\longrightarrow0
\end{equation}
with $s_i:=\rank\cK_i\geq1$, and put $\cG_i:=\cK_i^\vee$ and $\eta_i=s_i-1$.  On
\[
 Y:=\mathbb P_X(\cG_1)\times_X\cdots\times_X\mathbb P_X(\cG_n),
 \qquad
 D=d+\sum_i\eta_i,
\]
let $L_i$ be the pulled-back tautological quotient line bundle and $\xi_i=c_1(L_i)$.  Then
\begin{equation}\label{eq:general-shifted-volume}
 \frac1{D!}\int_Y\left(\sum_{i=1}^nx_i\xi_i\right)^D
 =\cN\!\left(x^\eta T_{X,\cF_\bullet}(x)\right).
\end{equation}
Each $L_i$ is globally generated and hence semiample.
\end{enumerate}
If the polynomial is nonzero, its realization in part~\textup{(ii)} may be chosen smooth and irreducible.
\end{theorem}

\begin{proof}
We first prove the explicit identity in~\textup{(iii)} and the realizable-volume assertion in~\textup{(ii)}.  Enlarge each finite generating space if necessary so that $s_i\geq1$.  Since the quotient in~\eqref{eq:generating-sequences} is locally free, $\cK_i$ is a vector bundle; dualizing the sequence shows that $\cG_i=\cK_i^\vee$ is globally generated.  The Whitney formula gives
\begin{equation}\label{eq:chern-to-complete}
 c_z(\cF_i)=c_z(\cK_i)^{-1}=c_{-z}(\cG_i)^{-1}
 =\sum_{m\geq0}h_m(\cG_i)z^m.
\end{equation}

We use Grothendieck's quotient convention
$\mathbb P_X(\cG)=\operatorname{Proj}_X(\operatorname{Sym}^\bullet\cG)$.
Projective-bundle morphisms are projective, and projectivity is stable under base change and composition; see the Stacks Project, Tags~\texttt{01W9}, \texttt{02V6}, and \texttt{0C4P}~\cite{Stacks}.  Since $X$ is integral, successive projective bundles over $X$ are integral.  Thus $Y$ is an irreducible projective variety of dimension $D$.

For a rank-$s$ vector bundle $\cG$ and $p:\mathbb P_X(\cG)\to X$, the quotient-convention projective-bundle formula is
\begin{equation}\label{eq:general-pushforward}
 p_*\bigl(\xi^{s-1+m}\cap p^*\delta\bigr)
 =h_m(\cG)\cap\delta\quad(m\geq0),
\end{equation}
with vanishing for exponents below $s-1$; see Fulton~\cite[Theorem~3.3]{FultonIT}.  Iterating~\eqref{eq:general-pushforward} yields
\[
 \int_Y\prod_i\xi_i^{\eta_i+\alpha_i}
 =\int_X\prod_i h_{\alpha_i}(\cG_i)
 =\int_X\prod_i c_{\alpha_i}(\cF_i),
\]
where the last equality is~\eqref{eq:chern-to-complete}.  Expanding the left side of~\eqref{eq:general-shifted-volume} in divided powers proves that identity exactly, including all factorials.

Because $\cG_i$ is globally generated, its tautological quotient $\cO(1)$ is globally generated; therefore each $L_i$ is globally generated and semiample.  Hence the left side of~\eqref{eq:general-shifted-volume} is a realizable volume polynomial.

It remains to remove the shift.  Put
\[
 g(\partial):=\partial^\eta=\prod_i\partial_i^{\eta_i}.
\]
In the notation of~\cite[Definition~1.2]{GHMSW25}, take the upper bound $\eta$.  Then
\[
 g(\partial)\,x^{[\eta]}=1,
\]
and the constant $1$ is the degree-zero realizable volume polynomial of a point.  Thus $g$ is a realizable covolume polynomial.  By~\cite[Theorem~1.3]{GHMSW25}, $g(\partial)$ preserves realizable volume polynomials.  Finally,
\[
 \partial^\eta\cN\!\left(x^\eta T_{X,\cF_\bullet}(x)\right)
 =\cN\!\left(T_{X,\cF_\bullet}(x)\right)
\]
term by term, so applying $g(\partial)$ to~\eqref{eq:general-shifted-volume} proves~\textup{(ii)}.  Since every realizable volume polynomial is Lorentzian~\cite[Theorem~4.6]{BH20}, part~\textup{(i)} follows.

For completeness, Cid-Ruiz gives an independent direct proof of~\textup{(i)}.  Choose a closed immersion $q:X\hookrightarrow\mathbb P^m$ with $m\geq d$, write $H$ for the hyperplane class, and define integers $a_\alpha$ by
\[
 q_*\!\left(\prod_{i=1}^n c_{t_i}(\cF_i)\cap[X]\right)
 =\sum_{|\alpha|\leq d}a_\alpha t^\alpha
   H^{m-d+|\alpha|}\cap[\mathbb P^m].
\]
Proposition~6.2 of~\cite{CidRuiz26} states that
\[
 \widetilde T(t_0,t_1,\ldots,t_n)
 :=\sum_{|\alpha|\leq d}a_\alpha
   t_0^{d-|\alpha|}t^\alpha
\]
is denormalized Lorentzian.  For $|\alpha|=d$, the projection formula gives
$a_\alpha=\int_X\prod_i c_{\alpha_i}(\cF_i)$, and hence
$\widetilde T(0,t_1,\ldots,t_n)=T_{X,\cF_\bullet}(t)$.  Lorentzian polynomials are preserved by nonnegative linear substitutions, including $t_0=0$~\cite[Theorem~2.10]{BH20}; this again proves~\textup{(i)}.

For the smooth clause, let $Z$ be an irreducible projective realization of $\cN(T_{X,\cF_\bullet})$ supplied by part~\textup{(ii)}, and choose a projective resolution $\rho:\widetilde Z\to Z$; see Hironaka~\cite{Hironaka64}.  Pullbacks of semiample divisors are semiample, and the projection formula together with $\rho_*[\widetilde Z]=[Z]$ preserves every top intersection number.  Thus $\widetilde Z$ gives a smooth irreducible realization in characteristic zero.  The realizing variety $Z$ need not be the explicit projective bundle $Y$ in~\eqref{eq:general-shifted-volume}.
\end{proof}

\section{Type A: ordinary Grassmannians and skew Schur polynomials}\label{sec:typeA}

Fix positive integers $r,c$ and let
\[
 G_A:=\Gr(r,r+c),
 \qquad
 0\longrightarrow\cS\longrightarrow\C^{r+c}\otimes\cO_{G_A}
 \longrightarrow\cQ\longrightarrow0
\]
be the universal sequence.  Put $E:=\cS^\vee$.  Both $E$ and $\cQ$ are globally generated.  For a partition $\nu\subseteq c^r$, write $\sigma_\nu=s_\nu(E)$.

For an irreducible projective subvariety $X\subseteq G_A$ of dimension $d$, define
\begin{equation}\label{eq:typeA-transform}
 d_\nu^A(X):=\int_Xs_\nu(E|_X),
 \qquad
 \Theta^A_{X,n}(x):=
 \sum_{\substack{\nu\subseteq c^r\\|\nu|=d}}
 d_\nu^A(X)s_\nu(x_1,\ldots,x_n).
\end{equation}
The coefficients $d_\nu^A(X)$ are nonnegative.  Indeed, $\sigma_\nu$ is represented by an effective Schubert cycle of codimension $d$, and a general translate of that cycle meets $X$ properly in an effective zero-cycle by Kleiman transversality; see Fulton--Pragacz~\cite[Chapter~4]{FultonPragacz}.

\begin{theorem}[Type A cycle transform]\label{thm:typeA-transform}
For every irreducible projective $X\subseteq G_A$ and every $n\geq1$,
\begin{equation}\label{eq:typeA-transform-coeff}
 [x^\alpha]\Theta^A_{X,n}
 =\int_X\prod_{i=1}^nh_{\alpha_i}(E|_X)
 =\int_X\prod_{i=1}^nc_{\alpha_i}(\cQ|_X)
 \qquad(|\alpha|=d).
\end{equation}
Equivalently,
\begin{equation}\label{eq:typeA-total-chern}
 \Theta^A_{X,n}(x)=\int_X\prod_{i=1}^nc_{x_i}(\cQ|_X).
\end{equation}
Consequently, $\cN(\Theta^A_{X,n})$ is a realizable volume polynomial.
\end{theorem}

\begin{proof}
The ordinary Cauchy--Pieri identity gives
\[
 \prod_i h_{\alpha_i}
 =\sum_\nu[x^\alpha]s_\nu(x)\,s_\nu.
\]
Under the Grassmannian Schubert homomorphism, $s_\nu$ maps to $s_\nu(E)$ for $\nu\subseteq c^r$ and to zero otherwise.  Restriction to $X$ and integration prove the first equality in~\eqref{eq:typeA-transform-coeff}.  Restricting the dual universal sequence gives
\[
 0\longrightarrow\cQ^\vee\longrightarrow(\C^{r+c})^\vee\otimes\cO_{G_A}
 \longrightarrow E\longrightarrow0,
\]
so
\[
 c_z(\cQ)=c_{-z}(E)^{-1}=\sum_{m\geq0}h_m(E)z^m.
\]
This proves the second equality and~\eqref{eq:typeA-total-chern}.  Theorem~\ref{thm:total-chern-volume}, applied to $n$ copies of $\cQ|_X$, proves the final assertion.
\end{proof}

Let $\mu\subseteq\lambda\subseteq c^r$.  Pad to length $r$ and put
\[
 \lambda^\vee=(c-\lambda_r,\ldots,c-\lambda_1).
\]
For opposite complete flags, define the Grassmannian Richardson variety in codimension notation by
\begin{equation}\label{eq:typeA-Richardson}
 R^A_{\lambda/\mu}
 :=\Omega_{\lambda^\vee}(F_\bullet)
 \cap\Omega_\mu(F_\bullet^{\mathrm{opp}}).
\end{equation}
It is integral and projective, and
\begin{equation}\label{eq:typeA-Richardson-class}
 \dim R^A_{\lambda/\mu}=|\lambda|-|\mu|,
 \qquad
 i_*[R^A_{\lambda/\mu}]
 =\sigma_{\lambda^\vee}\sigma_\mu\cap[G_A];
\end{equation}
see Brion~\cite[Section~1.3]{Brion05}.

\begin{theorem}[Ordinary skew-Schur volume theorem]\label{thm:main-A}
For every ordinary skew shape $\lambda/\mu$ and every $n\geq1$,
\begin{equation}\label{eq:typeA-main}
 \cN(s_{\lambda/\mu}(x_1,\ldots,x_n))
 \text{ is a realizable volume polynomial over }\C.
\end{equation}
More precisely, choose any rectangle $c^r$ containing $\lambda$, form
$R=R^A_{\lambda/\mu}\subseteq\Gr(r,r+c)$ as above, and put
$E_R=E|_R$.  Then
\begin{align}
 [x^\alpha]s_{\lambda/\mu}(x)
 &=\int_R\prod_{i=1}^nh_{\alpha_i}(E_R),
 \label{eq:typeA-Richardson-coeff}\\
 s_{\lambda/\mu}(x)
 &=\int_R\prod_{i=1}^nc_{x_i}(\cQ|_R).
 \label{eq:typeA-Richardson-total}
\end{align}
If $d=|\lambda|-|\mu|$ and
\[
 Y^A_{\lambda/\mu,n}
 :=\underbrace{\mathbb P_R(E_R)\times_R\cdots\times_R\mathbb P_R(E_R)}_{n\text{ factors}},
 \qquad D=d+n(r-1),
\]
then, for the tautological quotient divisor classes $\xi_i$,
\begin{equation}\label{eq:typeA-explicit}
 \frac1{D!}\int_{Y^A_{\lambda/\mu,n}}
 (x_1\xi_1+\cdots+x_n\xi_n)^D
 =\cN\!\left((x_1\cdots x_n)^{r-1}s_{\lambda/\mu}(x)\right).
\end{equation}
\end{theorem}

\begin{proof}
For $|\alpha|=|\lambda|-|\mu|$, repeated Pieri multiplication and complementary Schubert duality give
\begin{align*}
 \int_R\prod_i h_{\alpha_i}(E_R)
 &=\int_{G_A}\sigma_{\lambda^\vee}\sigma_\mu
   \prod_i h_{\alpha_i}(E)\\
 &=[s_\lambda]\,s_\mu h_{\alpha_1}\cdots h_{\alpha_n}
 =K_{\lambda/\mu,\alpha},
\end{align*}
where the final equality is~\eqref{eq:ordinary-pieri-coeff}.  Every Schur term in the Pieri product has size $|\lambda|$, so complementary duality selects exactly the $s_\lambda$ term.  This proves~\eqref{eq:typeA-Richardson-coeff}.  Equation~\eqref{eq:typeA-Richardson-total} follows from $h_m(E)=c_m(\cQ)$, and Theorem~\ref{thm:total-chern-volume} gives~\eqref{eq:typeA-main}.

For the explicit model, the restricted universal sequence
\[
 0\longrightarrow E_R^\vee\longrightarrow\C^{r+c}\otimes\cO_R
 \longrightarrow\cQ|_R\longrightarrow0
\]
has kernel dual $E_R$.  Therefore part~(iii) of Theorem~\ref{thm:total-chern-volume} gives~\eqref{eq:typeA-explicit} with shift $r-1$.
\end{proof}

Thus Conjecture~19 of Huh--Matherne--M\'esz\'aros--St.~Dizier~\cite[Conjecture~19]{HMMSD22} holds.

\begin{corollary}\label{cor:conj19}
The factorial normalization of every finite-variable skew Schur polynomial is Lorentzian.
\end{corollary}

\begin{remark}
The direct implication from~\eqref{eq:typeA-Richardson-total} to Lorentzianity is precisely Cid-Ruiz's total-Chern theorem.  The new skew-specific step is the Richardson--Pieri identity~\eqref{eq:typeA-Richardson-coeff}; the explicit model~\eqref{eq:typeA-explicit} and the upgrade to realizable volume use the more recent operator theory.  For $\mu=\varnothing$, projective volume realizability was already present in the geometric proof of Huh--Matherne--M\'esz\'aros--St.~Dizier.
\end{remark}

\begin{example}[The Richardson--Pieri identity for $(2,1)/(1)$]
\label{ex:typeA-21-over-1}
Take $r=c=2$, $\lambda=(2,1)$, and $\mu=(1)$, so that
$\lambda^\vee=(1)$ and
\[
 G_A=\Gr(2,4),
 \qquad
 R:=R^A_{(2,1)/(1)}
 =\Omega_{(1)}(F_\bullet)\cap
  \Omega_{(1)}(F_\bullet^{\mathrm{opp}}).
\]
By~\eqref{eq:typeA-Richardson-class}, $R$ is a Richardson surface with
$i_*[R]=\sigma_{(1)}^2\cap[G_A]$.  The two boxes of the skew diagram lie in
different rows and different columns, so their entries are unconstrained.
Consequently,
\begin{equation}\label{eq:example-A-polynomial}
 s_{(2,1)/(1)}=s_{(2)}+s_{(1,1)}
 =\sum_{a=1}^n x_a^2+2\sum_{1\leq a<b\leq n}x_ax_b
 =(x_1+\cdots+x_n)^2.
\end{equation}
The two content types are also visible directly in the Chow ring.  By the
Pieri rule and complementary Schubert duality on $\Gr(2,4)$
\cite[Chapter~9]{FultonYT},
\[
 \sigma_{(1)}^2=\sigma_{(2)}+\sigma_{(1,1)},
 \qquad
 \int_{G_A}\sigma_{(2)}^2
 =\int_{G_A}\sigma_{(1,1)}^2=1,
 \qquad
 \int_{G_A}\sigma_{(2)}\sigma_{(1,1)}=0.
\]
Since $h_m(E)=\sigma_{(m)}$, it follows that
\begin{align*}
 [x_a^2]s_{(2,1)/(1)}
 &=\int_R h_2(E|_R)
   =\int_{G_A}\sigma_{(1)}^2\sigma_{(2)}=1,\\
 [x_ax_b]s_{(2,1)/(1)}
 &=\int_R h_1(E|_R)^2
   =\int_{G_A}\sigma_{(1)}^4=2
   \qquad(a<b).
\end{align*}
Thus~\eqref{eq:typeA-Richardson-coeff} recovers, coefficient by coefficient,
the factor~$2$ for two distinct entries and the factor~$1$ for a repeated
entry in~\eqref{eq:example-A-polynomial}.
\end{example}

\section{Type C: Lagrangian Grassmannians and skew Schur \texorpdfstring{$P/Q$}{P/Q}-functions}\label{sec:typeC}

Let $\mathcal{SP}$ be the set of strict partitions.  Put
\begin{equation}\label{eq:q-generating}
 \prod_{j\geq1}\frac{1+y_jt}{1-y_jt}
 =\sum_{r\geq0}q_r(y)t^r,
 \qquad q_0=1.
\end{equation}
Let $\Omega_{\Q}\subseteq\Lambda_{\Q}$ be the self-dual rational Hopf algebra generated by the odd power sums.  Let
$\Gamma_{\Z}\subseteq\Omega_{\Q}$ denote the integral Schur-$Q$ subring generated by the one-row functions $q_r=Q_{(r)}$; equivalently,
$\Gamma_{\Z}=\bigoplus_{\nu\in\mathcal{SP}}\Z Q_\nu$.  The families
$\{P_\nu\}_{\nu\in\mathcal{SP}}$ and $\{Q_\nu\}_{\nu\in\mathcal{SP}}$ are dual bases of $\Omega_{\Q}$ for the canonical Hopf pairing:
\begin{equation}\label{eq:PQ-pairing}
 \angles{P_\lambda,Q_\nu}=\delta_{\lambda\nu},
 \qquad
 P_\lambda=2^{-\ell(\lambda)}Q_\lambda.
\end{equation}
This is the standard self-dual Schur-$P/Q$ Hopf algebra; see Macdonald~\cite[Chapter~III, Section~8]{Macdonald} and Lam--Lauve--Sottile~\cite[Section~4]{LLS11}.  We define skew elements by the coproduct identities
\begin{equation}\label{eq:skew-coproduct}
 \Delta(P_\lambda)=\sum_{\mu\in\mathcal{SP}}P_{\lambda/\mu}\otimes P_\mu,
 \qquad
 \Delta(Q_\lambda)=\sum_{\mu\in\mathcal{SP}}Q_{\lambda/\mu}\otimes Q_\mu.
\end{equation}
The conventions~\eqref{eq:PQ-pairing}--\eqref{eq:skew-coproduct} agree with the structure equations in Lam--Lauve--Sottile~\cite[equations~(17)--(18)]{LLS11}.  After finite-variable specialization, these skew functions agree with the standard marked shifted-tableau generating functions; see Macdonald~\cite[Chapter~III, Section~8]{Macdonald}.  Jing--Liu~\cite[Section~2]{JingLiu26} place the same construction in a general theory of skew functions for Hopf-dual pairs.

\begin{lemma}\label{lem:Cauchy-coeff}
Let $\alpha=(\alpha_1,\ldots,\alpha_n)\in\N^n$.  Then
\begin{equation}\label{eq:Cauchy-straight-coeff}
 \prod_{i=1}^nq_{\alpha_i}
 =\sum_{\nu\in\mathcal{SP}}
   [x^\alpha]P_\nu(x_1,\ldots,x_n)\,Q_\nu.
\end{equation}
For strict partitions $\mu\subseteq\lambda$,
\begin{equation}\label{eq:Cauchy-skew-coeff}
 [x^\alpha]P_{\lambda/\mu}(x)
 =[Q_\lambda]\left(Q_\mu\prod_{i=1}^nq_{\alpha_i}\right).
\end{equation}
Here $[Q_\lambda]$ means coefficient in the $Q$-basis.
\end{lemma}

\begin{proof}
The Schur $P/Q$ Cauchy identity is
\begin{equation}\label{eq:PQ-Cauchy}
 \sum_{\nu\in\mathcal{SP}}P_\nu(x)Q_\nu(y)
 =\prod_{i,j}\frac{1+x_iy_j}{1-x_iy_j}
 =\prod_{i=1}^n\left(\sum_{r\geq0}q_r(y)x_i^r\right).
\end{equation}
Taking the coefficient of $x^\alpha$ proves~\eqref{eq:Cauchy-straight-coeff}.

Expand
\[
 P_{\lambda/\mu}=\sum_{\nu\in\mathcal{SP}}a_{\mu\nu}^{\lambda}P_\nu.
\]
By Hopf duality and~\eqref{eq:skew-coproduct},
\[
 a_{\mu\nu}^{\lambda}
 =\angles{P_{\lambda/\mu},Q_\nu}
 =\angles{\Delta(P_\lambda),Q_\nu\otimes Q_\mu}
 =\angles{P_\lambda,Q_\nu Q_\mu}
 =[Q_\lambda](Q_\mu Q_\nu).
\]
Therefore, using~\eqref{eq:Cauchy-straight-coeff},
\begin{align*}
 [x^\alpha]P_{\lambda/\mu}
 &=\sum_\nu [Q_\lambda](Q_\mu Q_\nu)
     [x^\alpha]P_\nu\\
 &=[Q_\lambda]\left(
 Q_\mu\sum_\nu[x^\alpha]P_\nu Q_\nu\right)
 =[Q_\lambda]\left(Q_\mu\prod_i q_{\alpha_i}\right).
\end{align*}
This is~\eqref{eq:Cauchy-skew-coeff}.
\end{proof}

The following is straightforward.

\begin{lemma}\label{lem:skew-PQ-scaling}
For strict $\mu\subseteq\lambda$,
\[
 Q_{\lambda/\mu}=2^{\ell(\lambda)-\ell(\mu)}P_{\lambda/\mu}.
\]
Here $\ell(\nu)$ denotes the number of positive parts of a strict partition $\nu$.
\end{lemma}

\begin{proof}
Substitute $Q_\lambda=2^{\ell(\lambda)}P_\lambda$ into the second coproduct identity in~\eqref{eq:skew-coproduct}:
\[
 2^{\ell(\lambda)}\Delta(P_\lambda)
 =\sum_\mu Q_{\lambda/\mu}\otimes 2^{\ell(\mu)}P_\mu.
\]
Comparison with the first identity in~\eqref{eq:skew-coproduct}, using linear independence of the $P_\mu$, gives the claim.
\end{proof}

\subsection{The stable Schubert homomorphism}

Fix $N\geq1$, put
\[
 G_C:=\LG(N,2N),
 \qquad
 \rho_N:=(N,N-1,\ldots,1),
\]
and let
\begin{equation}\label{eq:LG-universal}
 0\longrightarrow\cU\longrightarrow V\otimes\cO_{G_C}
 \longrightarrow\cE\longrightarrow0
\end{equation}
be the universal sequence.  The symplectic form identifies $\cE\simeq\cU^\vee$, and $\cE$ is globally generated.

Recall the universal Pfaffian definition.  Put $q_0=1$ and $q_r=0$ for $r<0$, and for integers $a\geq b\geq0$ set
\begin{equation}\label{eq:Q-two-row}
 Q_{a,b}:=q_aq_b+2\sum_{j=1}^{b}(-1)^j q_{a+j}q_{b-j}.
\end{equation}
If a strict partition $\nu$ has odd length, append a zero part.  Then
\begin{equation}\label{eq:Q-Pfaffian}
 Q_\nu=\operatorname{Pf}\bigl(Q_{\nu_i,\nu_j}\bigr)_{i<j};
\end{equation}
see Macdonald~\cite[Chapter~III, Section~8]{Macdonald}.  Kresch--Tamvakis define modified polynomials $\widetilde Q_\nu$ by the same formulas after replacing the generators $q_r$ by the elementary symmetric functions $e_r$~\cite[Section~2, equations~(1)--(2)]{KT03}.

\begin{proposition}[Stable Schubert homomorphism]\label{prop:stable-map}
There is a surjective graded ring homomorphism
\begin{equation}\label{eq:stable-map}
 \phi_N:\Gamma_{\Z}\longrightarrow A^*(G_C)
\end{equation}
with
\begin{equation}\label{eq:stable-map-values}
\begin{gathered}
 \phi_N(q_0)=1,\qquad
 \phi_N(q_r)=c_r(\cE)=\sigma_{(r)}\quad(1\leq r\leq N),\qquad
 \phi_N(q_r)=0\quad(r>N),\\
 \phi_N(Q_\nu)=\sigma_\nu\quad(\nu\subseteq\rho_N).
\end{gathered}
\end{equation}
If $\nu_1>N$, then $\phi_N(Q_\nu)=0$.
\end{proposition}

\begin{proof}
By Macdonald's presentation, $\Gamma_{\Z}$ is generated by $q_1,q_2,\ldots$, subject to the coefficients of
\begin{equation}\label{eq:q-relations}
 q(t)q(-t)=1,
 \qquad q(t):=\sum_{r\geq0}q_rt^r;
\end{equation}
see~\cite[Chapter~III, equations~(8.2)--(8.5)]{Macdonald}.  On $G_C$, the universal sequence and $\cE\simeq\cU^\vee$ give
\[
 c_t(\cE)c_{-t}(\cE)=1.
\]
Consequently the assignment $q_r\mapsto c_r(\cE)$ respects the defining relations and determines a graded ring homomorphism
\[
 \phi_N:\Gamma_{\Z}\longrightarrow A^*(G_C).
\]

Under the substitution $q_r\mapsto e_r$, the two-row formula~\eqref{eq:Q-two-row} and the Pfaffian~\eqref{eq:Q-Pfaffian} become the modified polynomials $\widetilde Q_{a,b}$ and $\widetilde Q_\nu$ of Kresch--Tamvakis.  Evaluation of the elementary symmetric functions on the Chern roots of the tautological quotient bundle sends
\[
 \widetilde Q_\nu(\cE)=\sigma_\nu
 \qquad(\nu\subseteq\rho_N)
\]
and $e_r(\cE)=c_r(\cE)$ to the special Schubert class $\sigma_{(r)}$ for $1\leq r\leq N$; for $r>N$, both classes vanish.  See Kresch--Tamvakis~\cite[Section~2, equations~(1)--(2), and the paragraph preceding Theorem~1]{KT03}.  Thus $\phi_N(Q_\nu)=\sigma_\nu$ for $\nu\subseteq\rho_N$.  The special classes generate the Chow ring, so $\phi_N$ is surjective.  The same map is the $k=0$ specialization of the stable theta-polynomial Schubert homomorphism of Buch--Kresch--Tamvakis~\cite[Theorem~2]{BKT17}.

Finally, suppose $\nu_1>N$.  Since $\cE$ has rank $N$, one has $c_a(\cE)=0$ for every $a>N$.  Formula~\eqref{eq:Q-two-row} shows that every entry $Q_{\nu_1,\nu_j}$ in the Pfaffian row indexed by $\nu_1$ maps to zero: each of its summands contains a factor $c_{\nu_1+k}(\cE)$ with $k\geq0$.  Expanding~\eqref{eq:Q-Pfaffian} along that row gives $\phi_N(Q_\nu)=0$.

Kresch--Tamvakis formulate the classical Schubert presentation in integral cohomology.  Since $G_C$ has an affine Schubert-cell decomposition, the cycle-class map $A^*(G_C)\to H^{2*}(G_C,\Z)$ is an isomorphism; see Fulton~\cite[Example~19.1.11]{FultonIT}.  Hence the preceding identities hold in Chow.
\end{proof}

We need the following observation on the symplectic complete-Chern identity.

\begin{lemma}\label{lem:c-equals-h}
In $A^*(G_C)$ and after restriction to any subvariety,
\begin{equation}\label{eq:c-equals-h}
 c_m(\cE)=h_m(\cE)\qquad(m\geq0).
\end{equation}
\end{lemma}

\begin{proof}
From~\eqref{eq:LG-universal} and $\cU\simeq\cE^\vee$,
\[
 c_t(\cE)c_t(\cE^\vee)=c_t(\cE)c_{-t}(\cE)=1.
\]
Therefore $c_t(\cE)=c_{-t}(\cE)^{-1}=\sum_{m\geq0} h_m(\cE)t^m$.
\end{proof}

\begin{remark}
The coefficient of $t^2$ in $c_t(\cE)c_{-t}(\cE)=1$ is
\[
 c_1(\cE)^2=2c_2(\cE).
\]
This matches $Q_1^2=2Q_2$ and
$P_{(2)}(x,y)=x^2+2xy+y^2$, confirming that no power of $2$ is missing in the coefficient-intersection formula.
\end{remark}

For $\nu\subseteq\rho_N$, let $\nu^\vee=\rho_N\setminus\nu$, the strict partition whose parts are the elements of $\{1,\ldots,N\}$ not occurring in $\nu$.  Then
\begin{equation}\label{eq:LG-duality}
 \int_{G_C}\sigma_\eta\sigma_\nu=\delta_{\eta,\nu^\vee};
\end{equation}
see Kresch--Tamvakis~\cite[Section~3.1, equation~(36)]{KT03}.

For strict $\mu\subseteq\lambda\subseteq\rho_N$ and opposite isotropic flags, set
\begin{equation}\label{eq:typeC-Richardson}
 R^C_{\lambda/\mu}
 :=X_{\lambda^\vee}(F_\bullet)
 \cap X^\mu(F_\bullet^{\mathrm{opp}}).
\end{equation}
Standard Richardson theory for $G/P$ gives
\begin{equation}\label{eq:typeC-Richardson-class}
 \dim R^C_{\lambda/\mu}=|\lambda|-|\mu|,
 \qquad
 i_*[R^C_{\lambda/\mu}]
 =\sigma_{\lambda^\vee}\sigma_\mu\cap[G_C],
\end{equation}
and $R^C_{\lambda/\mu}$ is integral and projective~\cite[Section~1.3]{Brion05}.

\begin{theorem}[Skew Schur $P/Q$ volume theorem]\label{thm:main-PQ}
Let $\mu\subseteq\lambda$ be strict partitions and $n\geq1$.  The zero specialization is included by the convention in Section~\ref{sec:preliminaries}.  Then
\begin{equation}\label{eq:main-PQ}
 \cN(P_{\lambda/\mu}(x_1,\ldots,x_n))
 \text{ and }
 \cN(Q_{\lambda/\mu}(x_1,\ldots,x_n))
 \text{ are realizable volume polynomials over }\C.
\end{equation}
Choose $N\geq\max\{1,\lambda_1\}$, put $R=R^C_{\lambda/\mu}$ and $\cE_R=\cE|_R$.  If $d=|\lambda|-|\mu|$, then
\begin{align}
 [x^\alpha]P_{\lambda/\mu}(x)
 &=\int_R\prod_{i=1}^nc_{\alpha_i}(\cE_R)
 =\int_R\prod_{i=1}^nh_{\alpha_i}(\cE_R),
 \label{eq:typeC-coeff}\\
 P_{\lambda/\mu}(x)
 &=\int_R\prod_{i=1}^nc_{x_i}(\cE_R).
 \label{eq:typeC-total-Chern}
\end{align}
Furthermore, on
\[
 Y^C_{\lambda/\mu,n}
 :=\underbrace{\mathbb P_R(\cE_R)\times_R\cdots\times_R\mathbb P_R(\cE_R)}_{n\text{ factors}},
 \qquad D=d+n(N-1),
\]
one has
\begin{equation}\label{eq:typeC-explicit}
 \frac1{D!}\int_{Y^C_{\lambda/\mu,n}}
 (x_1\xi_1+\cdots+x_n\xi_n)^D
 =\cN\!\left((x_1\cdots x_n)^{N-1}P_{\lambda/\mu}(x)\right).
\end{equation}
\end{theorem}

\begin{proof}
Let $\alpha\in\N^n$ with $|\alpha|=d$.  Lemma~\ref{lem:Cauchy-coeff} gives
\[
 [x^\alpha]P_{\lambda/\mu}
 =[Q_\lambda]\left(Q_\mu\prod_i q_{\alpha_i}\right).
\]
Expand $Q_\mu\prod_i q_{\alpha_i}$ in the $Q$-basis and apply $\phi_N$.  Proposition~\ref{prop:stable-map} kills every term indexed by a strict partition with first part greater than $N$.  Since $\lambda_1\leq N$, the $Q_\lambda$ term survives, and pairing with $\sigma_{\lambda^\vee}$ extracts precisely its coefficient.  Therefore
\begin{align*}
 [x^\alpha]P_{\lambda/\mu}
 &=\int_{G_C}\sigma_{\lambda^\vee}\sigma_\mu
   \prod_i c_{\alpha_i}(\cE)\\
 &=\int_R\prod_i c_{\alpha_i}(\cE_R).
\end{align*}
Lemma~\ref{lem:c-equals-h} gives the second equality in~\eqref{eq:typeC-coeff}, and summation gives~\eqref{eq:typeC-total-Chern}.

Theorem~\ref{thm:total-chern-volume}, applied to $n$ copies of the globally generated bundle $\cE_R$, proves that $\cN(P_{\lambda/\mu})$ is realizable volume.  In the universal sequence the kernel is $\cU_R\simeq\cE_R^\vee$, so the dual kernel in Theorem~\ref{thm:total-chern-volume} is again $\cE_R$, of rank $N$.  Its explicit identity is exactly~\eqref{eq:typeC-explicit}.

Finally, Lemma~\ref{lem:skew-PQ-scaling} gives
$Q_{\lambda/\mu}=2^{\ell(\lambda)-\ell(\mu)}P_{\lambda/\mu}$, and positive rational scaling preserves realizable volume polynomiality.
\end{proof}

Thus Conjecture~20 of Huh--Matherne--M\'esz\'aros--St.~Dizier~\cite[Conjecture~20]{HMMSD22} holds.

\begin{corollary}\label{cor:P-conjecture}
For every strict partition $\lambda$, $\cN(P_\lambda)$ is Lorentzian.  More generally, the normalized skew $P$- and $Q$-polynomials are Lorentzian.
\end{corollary}

\subsection{Cycle transforms in type C}

For an irreducible projective subvariety $X\subseteq G_C$ of dimension $d$ and an integer $n\geq1$, put
\begin{equation}\label{eq:typeC-transform}
 d_\nu^C(X):=\int_X\sigma_\nu|_X,
 \qquad
 \Theta^C_{X,n}(x):=
 \sum_{\substack{\nu\subseteq\rho_N\\\nu\text{ strict},\ |\nu|=d}}
 d_\nu^C(X)P_\nu(x_1,\ldots,x_n).
\end{equation}
The integers $d_\nu^C(X)$ are nonnegative: $\sigma_\nu$ is represented by an effective Schubert cycle of codimension $d$, and a general translate meets $X$ properly in an effective zero-cycle by Kleiman transversality.

\begin{theorem}[Type C cycle transform]\label{thm:typeC-transform}
For every irreducible projective $X\subseteq G_C$ and every $n\geq1$,
\begin{equation}\label{eq:typeC-transform-coeff}
 [x^\alpha]\Theta^C_{X,n}
 =\int_X\prod_i c_{\alpha_i}(\cE|_X)
 =\int_X\prod_i h_{\alpha_i}(\cE|_X)
 \qquad(|\alpha|=d).
\end{equation}
Consequently, $\cN(\Theta^C_{X,n})$ is a realizable volume polynomial.  The Schubert variety $X_{\lambda^\vee}$ gives $P_\lambda$.  If $N\geq\lambda_1$, then the Richardson variety $R^C_{\lambda/\mu}$ gives $P_{\lambda/\mu}$.
\end{theorem}

\begin{proof}
Apply $\phi_N$ to~\eqref{eq:Cauchy-straight-coeff}; in degree $d$,
\[
 \prod_i c_{\alpha_i}(\cE)
 =\sum_{\substack{\nu\subseteq\rho_N\text{ strict}\\|\nu|=d}}
 [x^\alpha]P_\nu\,\sigma_\nu.
\]
Restriction to $X$ and integration prove the first equality in~\eqref{eq:typeC-transform-coeff}; Lemma~\ref{lem:c-equals-h} gives the second.  Theorem~\ref{thm:total-chern-volume} gives realizable volume.

For $X=X_{\lambda^\vee}$, Poincar\'e duality gives $d_\nu^C(X)=\delta_{\lambda\nu}$.  Suppose now that $N\geq\lambda_1$ and $X=R^C_{\lambda/\mu}$.  By~\eqref{eq:typeC-transform-coeff} and the Richardson coefficient identity~\eqref{eq:typeC-coeff},
\[
 [x^\alpha]\Theta^C_{X,n}
 =\int_X\prod_i c_{\alpha_i}(\cE|_X)
 =[x^\alpha]P_{\lambda/\mu}
 \qquad(|\alpha|=|\lambda|-|\mu|).
\]
Hence $\Theta^C_{X,n}=P_{\lambda/\mu}$ in every finite alphabet, with no stronger stable-range assumption.
\end{proof}
\section{Consequences and applications}\label{sec:applications}

\subsection{Permanence and formal operator descendants}

Throughout this section, let
\begin{equation}\label{eq:generic-F}
 F(x)=\sum_{\abs\alpha=d}c_\alpha x^\alpha
 \qquad(c_\alpha\geq0)
\end{equation}
be a nonzero homogeneous polynomial such that
\[
 f:=\cN(F)=\sum_{\abs\alpha=d}c_\alpha x^{[\alpha]}
\]
is a realizable volume polynomial.  Every nonzero type~A or type~C cycle transform constructed above, every nonzero ordinary skew Schur polynomial, and every nonzero skew Schur $P$- or $Q$-polynomial satisfies this hypothesis.

\begin{proposition}\label{prop:permanence}
The following polynomials are realizable volume polynomials.
\begin{enumerate}[label=\textup{(\roman*)}]
\item $\partial^\beta f$ for every $\beta\in\N^n$.
\item $f(Ay)$ for every matrix $A$ with nonnegative rational entries.
\item The lower and upper coordinatewise truncations
\[
 f_{\geq\gamma}=\sum_{\alpha\geq\gamma}c_\alpha x^{[\alpha]},
 \qquad
 f_{\leq\gamma}=\sum_{\alpha\leq\gamma}c_\alpha x^{[\alpha]},
\]
where $\alpha\leq\gamma$ means $\alpha_i\leq\gamma_i$ for every $i$.
\item If $F_1,\ldots,F_m$ have realizable factorial normalizations, then both
\[
 \prod_{a=1}^m\cN(F_a)
 \quad\text{and}\quad
 \cN\left(\prod_{a=1}^mF_a\right)
\]
are realizable volume polynomials.
\end{enumerate}
\end{proposition}

\begin{proof}
Part~(i) is Theorem~1.3 of Grund--Huh--Micha{\l}ek--S\"uss--Wang applied to the realizable-covolume monomial differential operator $\partial^\beta$.  Part~(ii) is~\cite[Proposition~2.8]{GHMSW25}.  Part~(iii) is~\cite[Corollary~3.3]{GHMSW25}.  The first assertion of part~(iv) follows from the product construction on the product of the realizing varieties~\cite[Corollary~2.9]{GHMSW25}; the second is~\cite[Corollary~2.6]{GHMSW25}.
\end{proof}

A homogeneous polynomial $g(\partial)$ is called a \emph{realizable covolume
polynomial} if, for some $\mu\in\N^n$ with
$g\in\Q[\partial]_{\leq\mu}$, the polynomial
$g(\partial)x^{[\mu]}$ is realizable volume.  This condition is independent of
$\mu$; see \cite[Definition~1.2]{GHMSW25}.

\begin{theorem}
\label{thm:operator-descendants}
Let $g(\partial)\in\Q[\partial_1,\ldots,\partial_n]$ be a realizable
covolume polynomial over $\C$.  Then
\begin{equation}\label{eq:covolume-descendant}
 g(\partial)f\text{ is a realizable volume polynomial}.
\end{equation}
In particular, the following assertions hold.
\begin{enumerate}[label=\textup{(\roman*)}]
\item For every partition $\tau$ with $\ell(\tau)\leq n$ and every permutation $w\in\mathfrak S_n$, the polynomials
\[
 s_\tau(\partial_1,\ldots,\partial_n)f,
 \qquad
 \mathfrak S_w(\partial_1,\ldots,\partial_n)f
\]
are realizable volume polynomials, where $\mathfrak S_w$ is the Schubert polynomial of $w$.
\item For every integer $q\geq1$, the $q$-fold factorial normalization
\begin{equation}\label{eq:iterated-factorial-normalization}
 \cN_q(F):=\sum_{|\alpha|=d}c_\alpha
       \frac{x^\alpha}{(\alpha!)^q}
\end{equation}
is a realizable volume polynomial.  Thus the same coefficient array gives an
infinite hierarchy of volume polynomials.
\item Put $\mu_i:=\max\{\alpha_i:c_\alpha\neq0\}$ and let
$Y_i=\{y_{i1},\ldots,y_{i\mu_i}\}$.  The full polarization
\begin{equation}\label{eq:full-polarization}
 \operatorname{Pol}_\mu(f)
 :=\sum_{|\alpha|=d}\frac{c_\alpha}{\alpha!}
   \prod_{i=1}^n
   \frac{e_{\alpha_i}(Y_i)}{\binom{\mu_i}{\alpha_i}}
\end{equation}
(with the evident convention when $\mu_i=0$) is a multiaffine realizable
volume polynomial.  If $B=\Supp(F)$ and
$E_i=\{(i,1),\ldots,(i,\mu_i)\}$, then its support is
\begin{equation}\label{eq:polarized-support}
 \widetilde B
 =\left\{\mathbf 1_S:
   \bigl(|S\cap E_1|,\ldots,|S\cap E_n|\bigr)\in B\right\}.
\end{equation}
Consequently, $\widetilde B$ is the set of bases of an algebraic matroid over
every field of characteristic zero, and the block-cardinality projection
$\widetilde B\to B$ is surjective.
\item For $i\neq j$ and $t\in\Q_{\geq0}$,
$(1+t x_i\partial_j)f$ is realizable volume.  If $h$ is any multiaffine
realizable volume polynomial, then for rational $0\leq t\leq1$ the symmetric
exclusion operator
\[
 \Phi_{ij}(t)h
 :=t h+(1-t)h\circ(i\ j)
\]
is realizable volume.
\end{enumerate}
Every nonzero polynomial in this theorem therefore satisfies all conclusions
proved below whose stated hypotheses it meets.
\end{theorem}

\begin{proof}
The main assertion is the characterization of realizable covolume
polynomials as differential operators preserving realizable volume
polynomials~\cite[Theorem~1.3]{GHMSW25}.  Schubert polynomials are realizable
covolume polynomials by~\cite[Corollary~2.4]{GHMSW25}.  A Schur polynomial
$s_\tau(x_1,\ldots,x_n)$ with $\ell(\tau)\leq n$ is a Grassmannian Schubert
polynomial after adjoining finitely many dummy variables.  We may likewise view
$f$ as independent of these additional variables; applying the operator theorem
and then discarding the unused variables proves~(i).

For~(ii), the case $q=1$ is the hypothesis on $f$.  The normalization operator
sends a realizable volume polynomial to a realizable volume polynomial
by~\cite[Proposition~4.2]{GHMSW25}; iteration gives
\eqref{eq:iterated-factorial-normalization}.

The polynomial in~\eqref{eq:full-polarization} is the usual full polarization:
it is symmetric in each block $Y_i$, multiaffine, and its diagonal
specialization $y_{i1}=\cdots=y_{i\mu_i}=x_i$ is $f$.  Successive applications
of the polarization theorem~\cite[Proposition~4.1]{GHMSW25} prove that it is
realizable volume.  Each elementary symmetric polynomial in
\eqref{eq:full-polarization} has positive coefficients on all subsets of the
prescribed cardinality, which proves~\eqref{eq:polarized-support}.  A
multiaffine $M$-convex support is a matroid.  Since the polarized polynomial is
realizable volume, the support-realization theorem
\cite[Proposition~5.4]{GHMSW25} makes this matroid algebraic; characteristic
independence follows from~\cite[Proposition~2.10]{GHMSW25}.

Finally,~(iv) is~\cite[Propositions~4.3 and~4.4]{GHMSW25}.
\end{proof}

\subsection{Reverse Khovanskii--Teissier inequalities}

For $v\in\R_{\geq0}^n$, write $\partial_v=\sum_i v_i\partial_i$, and define the complete mixed form
\begin{equation}\label{eq:complete-mixed-form}
 M_f(v_1,\ldots,v_d)
 :=\partial_{v_1}\cdots\partial_{v_d}f.
\end{equation}
It is symmetric and multilinear.

\begin{theorem}[Reverse Khovanskii--Teissier inequalities]\label{thm:master-RKT}
For $u,v,w\in\R_{\geq0}^n$ and $0\leq e\leq d$,
\begin{equation}\label{eq:master-RKT}
 \binom de
 M_f(u^{d-e},v^e)M_f(u^e,w^{d-e})
 \geq
 M_f(u^d)M_f(v^e,w^{d-e}).
\end{equation}
More generally, let $\beta\in\N^n$, put $m=d-\abs\beta$, and assume $m\geq0$.  Then, for $0\leq e\leq m$ and $i,j,k\in[n]$,
\begin{equation}\label{eq:master-coeff-RKT}
 \binom me
 c_{\beta+(m-e)e_i+e\,e_j}
 c_{\beta+e\,e_i+(m-e)e_k}
 \geq
 c_{\beta+m e_i}
 c_{\beta+e\,e_j+(m-e)e_k}.
\end{equation}
Coefficients with an index outside $\N^n$ are interpreted as zero.
\end{theorem}

\begin{proof}
For $e=0$ and $e=d$, inequality~\eqref{eq:master-RKT} is an identity.  Assume $1\leq e\leq d-1$, and choose any projective realization~\eqref{eq:rv-definition} with $q>0$, as supplied by the definition of a nonzero realizable volume polynomial.  No smoothness is needed: the algebraic reverse Khovanskii--Teissier inequality of Jiang--Li applies to arbitrary projective varieties over an algebraically closed field.  For $a\in\R_{\geq0}^n$, let $D(a)=\sum_i a_iD_i$.  Then
\[
 M_f(v_1,\ldots,v_d)=q\int_X D(v_1)\cdots D(v_d).
\]
The classes $D(u),D(v),D(w)$ are nef $\R$-divisor classes.  The reverse Khovanskii--Teissier inequality of Lehmann--Xiao~\cite[Theorem~5.7]{LehmannXiao17}, with its algebraic proof over arbitrary algebraically closed fields by Jiang--Li~\cite[Theorem~1.1]{JiangLi23}, states
\[
 \left( D(u)^{d-e}D(v)^e\right)
 \left( D(u)^eD(w)^{d-e}\right)
 \geq
 \frac{e!(d-e)!}{d!}
 \left(D(u)^d\right)
 \left(D(v)^eD(w)^{d-e}\right).
\]
Multiplication by $q^2\binom de$ and the mixed-intersection identity prove~\eqref{eq:master-RKT}.  Irrational coefficients follow by continuity from rational approximations.

Put $g=\partial^\beta f$.  If $g=0$, then every coefficient $c_{\beta+\eta}$ with $|\eta|=m$ vanishes, so both sides of~\eqref{eq:master-coeff-RKT} are zero and the assertion is immediate.  Assume henceforth that $g\neq0$.  By Proposition~\ref{prop:permanence}(i), $g$ is a realizable volume polynomial of degree $m$, and
\[
 g(x)=\sum_{\abs\eta=m}c_{\beta+\eta}x^{[\eta]}.
\]
Apply~\eqref{eq:master-RKT} to $g$ with $u=e_i$, $v=e_j$, and $w=e_k$.  Complete differentiation of divided powers gives
\[
 M_g(e_i^{m-e},e_j^e)=c_{\beta+(m-e)e_i+e\,e_j},
\]
and similarly for the other three terms, proving~\eqref{eq:master-coeff-RKT}.
\end{proof}

\subsection{Hodge--Riemann inequalities and dominance monotonicity}

The following is fundamental.

\begin{theorem}[Lorentzian Hodge--Riemann package]\label{thm:master-HR}
The polynomial $f$ is Lorentzian, completely log-concave, and strongly log-concave.  Its support is $M$-convex and hence saturated in its Newton polytope.  If $d\geq2$, the following additional statements hold.
\begin{enumerate}[label=\textup{(\roman*)}]
\item For every $\beta\in\N^n$ with $\abs\beta=d-2$, the symmetric matrix
\begin{equation}\label{eq:H-beta}
 H^\beta:=\bigl(c_{\beta+e_i+e_j}\bigr)_{1\leq i,j\leq n}
\end{equation}
has at most one positive eigenvalue.  If $H^\beta\neq0$, it has exactly one positive eigenvalue.
\item For every $\beta\in\N^n$ with $\abs\beta=d-2$ and every nonempty $I\subseteq[n]$,
\begin{equation}\label{eq:principal-sign}
 (-1)^{\abs I-1}\det H_I^\beta\geq0,
\end{equation}
where $H_I^\beta$ is the principal submatrix on $I$.
\item If $\alpha\in\N^n$ satisfies $\abs\alpha=d$, $i\neq j$, and $\alpha_i,\alpha_j\geq1$, then
\begin{equation}\label{eq:root-logconcavity}
 c_\alpha^2\geq
 c_{\alpha+e_i-e_j}c_{\alpha-e_i+e_j}.
\end{equation}
On every affine root line $\alpha+\Z(e_i-e_j)$, the nonzero coefficient sequence is supported on an integer interval and is log-concave and unimodal.
\item For $u,v,w_3,\ldots,w_d\in\R_{\geq0}^n$,
\begin{equation}\label{eq:AF-master}
 M_f(u,v,w_3,\ldots,w_d)^2
 \geq
 M_f(u,u,w_3,\ldots,w_d)
 M_f(v,v,w_3,\ldots,w_d).
\end{equation}
\end{enumerate}
\end{theorem}

\begin{proof}
Semiample divisors are nef, so the Br\"and\'en--Huh volume theorem implies that $f$ is Lorentzian~\cite[Theorem~4.6]{BH20}.  Lorentzianity is equivalent to complete and strong log-concavity for homogeneous polynomials with nonnegative coefficients, and its support is $M$-convex~\cite[Theorems~2.25 and~2.30]{BH20}.  A finite $M$-convex set is the set of lattice points of its integral base polytope, proving saturation.

For $\abs\beta=d-2$,
\[
 \partial^\beta f
 =\sum_i c_{\beta+2e_i}\frac{x_i^2}{2}
   +\sum_{i<j}c_{\beta+e_i+e_j}x_ix_j.
\]
Thus its Hessian is exactly $H^\beta$.  Lorentzianity gives at most one positive eigenvalue.  If $H^\beta\neq0$, then it is entrywise nonnegative and
\[
 \one^{\mathsf T}H^\beta\one>0,
\]
so it has at least one positive eigenvalue.  This proves the first assertion.  Every principal submatrix also has at most one positive eigenvalue; if nonzero, it has exactly one by the same argument.  Its remaining eigenvalues are nonpositive, which gives~\eqref{eq:principal-sign}.

Set $\beta=\alpha-e_i-e_j$ and apply~\eqref{eq:principal-sign} to the $2\times2$ principal minor on $\{i,j\}$.  This is~\eqref{eq:root-logconcavity}.  To prove the interval assertion, let two support points lie on the same root line.  In the symmetric exchange axiom, the only coordinates in which they differ are $i$ and $j$, so repeated exchanges fill every lattice point between them.  Log-concavity on this interval follows from~\eqref{eq:root-logconcavity}, and a nonnegative log-concave sequence without internal zeros is unimodal.

Finally,~\eqref{eq:AF-master} is the Alexandrov--Fenchel inequality for the nef divisor classes
\[
 D(u),\ D(v),\ D(w_3),\ldots,D(w_d),
\]
or equivalently Br\"and\'en--Huh~\cite[Proposition~4.5]{BH20}.
\end{proof}

\subsection{Symmetric supports and polymatroid realizations}

We record the observation that ymmetric support is an exact permutahedron.

\begin{theorem}\label{thm:symmetric-permutahedron}
Assume in addition that $F$ is nonzero and symmetric.  Then there is a unique partition
\[
 \kappa(F)=(\kappa_1\geq\cdots\geq\kappa_n\geq0),
 \qquad \sum_i\kappa_i=d,
\]
such that
\begin{equation}\label{eq:symmetric-support}
 \Supp(F)=\mathcal P_{\kappa(F)}\cap\Z^n,
 \qquad
 \Newt(F)=\mathcal P_{\kappa(F)},
\end{equation}
where $\mathcal P_\kappa=\conv\{\pi\kappa:\pi\in\mathfrak S_n\}$.  The exponent $\kappa(F)$ belongs to $\Supp(F)$ and is the unique dominance-maximal partition occurring in the support.  The associated polymatroid rank function is
\begin{equation}\label{eq:symmetric-rank-function}
 r_\kappa(A)=\sum_{p=1}^{|A|}\kappa_p,
\end{equation}
and this polymatroid is linearly representable over every infinite field.

If $F_1,\ldots,F_m$ are symmetric polynomials satisfying the same hypotheses, then
\begin{equation}\label{eq:product-kappa}
 \kappa\left(\prod_{a=1}^mF_a\right)=\sum_{a=1}^m\kappa(F_a),
\end{equation}
and
\begin{equation}\label{eq:product-support-general}
 \Supp\left(\prod_aF_a\right)
 =\mathcal P_{\sum_a\kappa(F_a)}\cap\Z^n.
\end{equation}
Equivalently, the lattice points satisfy the mixed integer-decomposition identity
\begin{equation}\label{eq:mixed-IDP}
 \left(\mathcal P_{\kappa(F_1)}\cap\Z^n\right)+\cdots+
 \left(\mathcal P_{\kappa(F_m)}\cap\Z^n\right)
 =\mathcal P_{\sum_a\kappa(F_a)}\cap\Z^n.
\end{equation}
In particular, every $\mathcal P_{\kappa(F)}$ has the integer-decomposition
property.  If $v_a=[x^{\kappa(F_a)}]F_a$, then the coefficient of every vertex monomial of the product is $\prod_a v_a$.
\end{theorem}

\begin{proof}
Let $B=\Supp(F)$.  By Theorem~\ref{thm:master-HR}, $B$ is $M$-convex, hence the base set of an integral polymatroid.  Its rank function is
\[
 r(A):=\max_{\alpha\in B}\sum_{i\in A}\alpha_i.
\]
Symmetry of $F$ implies $r(A)=R(\abs A)$ for a function $R:\{0,\ldots,n\}\to\N$.  Submodularity applied to two subsets of size $p$ whose intersection has size $p-1$ shows
\[
 R(p+1)-R(p)\leq R(p)-R(p-1).
\]
Define $\kappa_p=R(p)-R(p-1)$.  Monotonicity gives $\kappa_p\geq0$, the displayed inequality gives $\kappa_1\geq\cdots\geq\kappa_n$, and $R(n)=d$ gives $\sum_p\kappa_p=d$.

The base polytope of $r$ is
\[
 \left\{z\in\R_{\geq0}^n:
 z([n])=d,\ z(A)\leq R(\abs A)=\sum_{p=1}^{\abs A}\kappa_p
 \text{ for all }A\subseteq[n]\right\}.
\]
By Rado's theorem~\cite{Rado52}, this is precisely $\mathcal P_\kappa$.  Since an $M$-convex set consists of all lattice points of its base polytope,~\eqref{eq:symmetric-support} follows.  The vector $\kappa$ is a vertex and therefore lies in the support.  The same inequalities show that every partition $\nu\in B$ satisfies $\nu\unlhd\kappa$, proving uniqueness of the dominance maximum.

Set $\kappa_{n+1}=0$.  The rank function admits the uniform-matroid decomposition
\begin{equation}\label{eq:uniform-rank-decomposition}
 r_\kappa(A)
 =\sum_{q=1}^n(\kappa_q-\kappa_{q+1})\min\{|A|,q\}.
\end{equation}
Let $k$ be an infinite field.  For every $q$, choose $n$ vectors
$v_{q,1},\ldots,v_{q,n}$ in a $q$-dimensional $k$-vector space $W_q$ so that
every at most $q$ of them are linearly independent.  Take
$\kappa_q-\kappa_{q+1}$ mutually independent copies of this configuration and,
for each ground-set element $i$, let $V_i$ be the direct sum of the lines
spanned by all vectors carrying the label $i$.  Then
\[
 \dim_k\sum_{i\in A}V_i
 =\sum_{q=1}^n(\kappa_q-\kappa_{q+1})\min\{|A|,q\}
 =r_\kappa(A).
\]
This proves linear representability over every infinite field.

For products, coefficients are nonnegative, so the Newton polytope is the Minkowski sum of the individual Newton polytopes.  Permutahedra satisfy
\[
 \mathcal P_{\kappa^{(1)}}+\cdots+\mathcal P_{\kappa^{(m)}}
 =\mathcal P_{\kappa^{(1)}+\cdots+\kappa^{(m)}}.
\]
By Proposition~\ref{prop:permanence}(iv), the factorial normalization of the product is realizable volume; hence its support is $M$-convex and contains every lattice point of its Newton polytope.  This proves~\eqref{eq:product-kappa}--\eqref{eq:product-support-general}.  Since the support of a product with nonnegative coefficients is the set-theoretic Minkowski sum of the supports of the factors, comparison with~\eqref{eq:product-support-general} gives~\eqref{eq:mixed-IDP}.  Applying this identity to $m$ identical copies of $F$ gives
\[
 \underbrace{\bigl(\mathcal P_{\kappa(F)}\cap\mathbb Z^n\bigr)+\cdots+
 \bigl(\mathcal P_{\kappa(F)}\cap\mathbb Z^n\bigr)}_{m\text{ summands}}
 =\bigl(m\mathcal P_{\kappa(F)}\bigr)\cap\mathbb Z^n,
\]
which is precisely the integer-decomposition property of $\mathcal P_{\kappa(F)}$.

Choose a strictly decreasing linear functional exposing the vertex $\sum_a\kappa(F_a)$.  It exposes $\kappa(F_a)$ uniquely in each summand.  Thus the only support decomposition of the product vertex is the tuple of individual vertices, and its coefficient is the product of their coefficients.  Symmetry gives the same conclusion at every permuted vertex.
\end{proof}

\begin{remark}[Relation to symmetric Lorentzian support theory]\label{rem:chin-qin-support}
Chin--Qin prove, in the stable range $n\geq d$, that the partition-indexed support of a degree-$d$ symmetric Lorentzian polynomial is a dominance interval and that its Newton polytope is the corresponding permutahedron~\cite[Theorems~3.7 and~3.9]{ChinQin25}.  Theorem~\ref{thm:symmetric-permutahedron} gives a finite-variable formulation directly from the $M$-convex rank function and additionally records the uniform-matroid representation, the mixed integer-decomposition identity, and the behavior of vertex coefficients under products.
\end{remark}

\begin{theorem}\label{thm:dominance-monotonicity}
Assume that $F$ is symmetric.  For partitions $\mu,\lambda$ of $d$, padded by zeros to length $n$, if
\[
 \mu\unlhd\lambda,
\]
then
\begin{equation}\label{eq:dominance-monotonicity}
 c_\mu\geq c_\lambda.
\end{equation}
Consequently, if
\[
 d=qn+r,
 \qquad 0\leq r<n,
 \qquad
 \operatorname{bal}_n(d):=(\underbrace{q+1,\ldots,q+1}_{r},
 \underbrace{q,\ldots,q}_{n-r}),
\]
then $c_{\operatorname{bal}_n(d)}$ is maximal among all partition-indexed coefficients of $F$.  If $F\neq0$, the dominance-maximal support partition $\kappa(F)$ from Theorem~\ref{thm:symmetric-permutahedron} has minimal positive coefficient among the partition-indexed support.
\end{theorem}

\begin{proof}
It suffices to prove monotonicity under one balancing transfer.  Let
$\alpha\in\N^n$ and suppose that $\alpha_i\geq\alpha_j+2$.  Put
$L=\alpha_i-\alpha_j$ and
\[
 a_t:=c_{\alpha-t e_i+t e_j}
 \qquad(0\leq t\leq L).
\]
If $a_0=0$, the desired inequality $a_1\geq a_0$ is trivial.  Assume $a_0>0$.  Symmetry gives $a_L=a_0>0$, since the exponent at $t=L$ is obtained from $\alpha$ by interchanging the $i$th and $j$th coordinates.  The $M$-convexity of the support fills the entire root segment between these two exponents, so every $a_t$ is positive.  The root-direction inequalities in Theorem~\ref{thm:master-HR}~\textup{(iii)} imply
\[
 a_t^2\geq a_{t-1}a_{t+1}
 \qquad(1\leq t\leq L-1).
\]
Thus the ratios $a_t/a_{t-1}$ are nonincreasing.  If $a_1<a_0$, all these ratios would be strictly smaller than one, forcing $a_L<a_0$, a contradiction.  Hence
\[
 c_{\alpha-e_i+e_j}=a_1\geq a_0=c_\alpha.
\]
The transfer characterization of majorization expresses $\mu\unlhd\lambda$ as a finite sequence of such unit transfers, up to permutations of coordinates; see Marshall--Olkin--Arnold~\cite[Chapter~2, Section~B]{MarshallOlkinArnold11}.  Iteration and symmetry prove~\eqref{eq:dominance-monotonicity}.

The balanced partition is dominance-minimal among partitions of $d$ with at most $n$ parts, which proves the maximality assertion.  Every partition in the support is dominated by $\kappa(F)$, so the final assertion follows from~\eqref{eq:dominance-monotonicity}.
\end{proof}

\begin{remark}
The coefficient inequality~\eqref{eq:dominance-monotonicity} is a general theorem for symmetric Lorentzian polynomials; see also Chin--Qin~\cite[Theorem~3.16]{ChinQin25}.  We include the short root-line proof because it makes the direction of monotonicity transparent and applies immediately to the ordinary and shifted tableau arrays identified in this paper.
\end{remark}

The following states linear and algebraic support polymatroids.

\begin{corollary}\label{cor:algebraic-polymatroid}
The support polymatroid of every nonzero type~A or type~C cycle transform, and hence of every nonzero skew Schur, skew Schur $P$, or skew Schur $Q$ polynomial, is linearly representable over every infinite field.  In the present symmetric setting this is already a formal consequence of $M$-convexity and symmetry, through the uniform-matroid decomposition in Theorem~\ref{thm:symmetric-permutahedron}; it is not claimed as an additional volume-theoretic novelty.  The support polymatroids of all nonzero derivatives and all nonzero coordinate truncations are algebraic over every field of characteristic zero.
\end{corollary}

\begin{proof}
The original cycle transforms are symmetric, so their linear representability is Theorem~\ref{thm:symmetric-permutahedron}.  For the descendants, Grund--Huh--Micha{\l}ek--S\"uss--Wang prove that an integral polymatroid is algebraic over a field if and only if its base set is the support of a realizable volume polynomial over that field~\cite[Proposition~5.4]{GHMSW25}; see also~\cite[Proposition~5.5]{HuhVolume26}.  Factorial normalization does not alter support.  Our descendants are realizable over $\C$ by Proposition~\ref{prop:permanence}, and realizability depends only on the characteristic of the ground field~\cite[Proposition~2.10]{GHMSW25}.  The characterization gives the result.
\end{proof}

\subsection{Weighted aggregations and factorially tilted content laws}

We have the following result on weighted aggregations and factorially tilted content laws.

\begin{theorem}\label{thm:block-covariance}
For every matrix $L$ with nonnegative real entries, the polynomial $f(Ly)$ is
Lorentzian; if the entries of $L$ are rational, it is a realizable volume
polynomial.  Thus every weighted aggregation of the coefficient array into
finitely many blocks remains Lorentzian.

Let $I\sqcup J=[n]$, choose nonnegative real weights $a_i$ for $i\in I$ and $b_j$ for $j\in J$, and set
\[
 x_i=a_i s\ (i\in I),
 \qquad
 x_j=b_j t\ (j\in J).
\]
Write
\begin{equation}\label{eq:block-g}
 f(x(s,t))=\sum_{r=0}^d C_r s^r t^{d-r}.
\end{equation}
If this polynomial is nonzero, then $C_0,\ldots,C_d$ has no internal zeros and
\begin{equation}\label{eq:block-ULC}
 \left(\frac{C_r}{\binom dr}\right)^2
 \geq
 \frac{C_{r-1}}{\binom d{r-1}}
 \frac{C_{r+1}}{\binom d{r+1}}
 \quad(0<r<d).
\end{equation}

For $z\in\R_{>0}^n$, define
\begin{equation}\label{eq:factorial-probability}
 \mathbb P_z(\alpha)
 :=\frac{c_\alpha z^\alpha/\alpha!}{f(z)}.
\end{equation}
Then, with $Z=\diag(z_1,\ldots,z_n)$,
\begin{equation}\label{eq:cov-identity-general}
 \Cov_z(\alpha)
 =\diag(\mathbb E_z\alpha)+Z\nabla^2\log f(z)Z,
\end{equation}
and hence
\begin{equation}\label{eq:cov-bound-general}
 \Cov_z(\alpha)\psd\diag(\mathbb E_z\alpha).
\end{equation}
Here $A\preceq_{\mathrm{psd}}B$ means that $B-A$ is positive semidefinite.
Equivalently,
\[
 \Var_z(u\cdot\alpha)
 \leq\sum_i u_i^2\mathbb E_z[\alpha_i]
 \qquad(u\in\R^n).
\]
\end{theorem}

\begin{proof}
Nonnegative real changes of variables preserve Lorentzianity
\cite[Theorem~2.10]{BH20}, while the rational realizable-volume assertion is
Proposition~\ref{prop:permanence}(ii).  The substitution in~\eqref{eq:block-g}
is such a nonnegative linear change of variables, so the resulting bivariate
polynomial is Lorentzian.  The bivariate characterization of Lorentzian forms is precisely~\eqref{eq:block-ULC}, together with absence of internal zeros; see Br\"and\'en--Huh~\cite[Example~2.26]{BH20}.

Logarithmic differentiation of the partition function $f(z)$ gives
\[
 \mathbb E_z[\alpha_i]=z_i\partial_i\log f(z).
\]
A second logarithmic derivative gives
\[
 \Cov_z(\alpha_i,\alpha_j)=z_iz_j\partial_{ij}\log f(z)
 \quad(i\neq j)
\]
and
\[
 \Var_z(\alpha_i)
 =z_i^2\partial_{ii}\log f(z)+\mathbb E_z[\alpha_i].
\]
These identities assemble into~\eqref{eq:cov-identity-general}.  Since a nonzero Lorentzian polynomial is log-concave on the positive orthant, $\nabla^2\log f(z)$ is negative semidefinite.  This proves~\eqref{eq:cov-bound-general}; the variance inequality follows by taking quadratic forms.
\end{proof}

\begin{remark}
The law~\eqref{eq:factorial-probability} is the factorially tilted content law, not the uniform distribution on tableaux.  The Loewner-order estimate~\eqref{eq:cov-bound-general} does not by itself imply negative association or real stability.
\end{remark}

\begin{corollary}\label{cor:ordinary-tableau-package}
Let $K_{\lambda/\mu,\alpha}=[x^\alpha]s_{\lambda/\mu}$ and $d=|\lambda|-|\mu|$, and set $K_{\lambda/\mu,\gamma}=0$ when $\gamma\notin\N^n$ or $|\gamma|\neq d$.  Then the following hold.
\begin{enumerate}[label=\textup{(\roman*)}]
\item For $|\beta|\leq d$, $m=d-|\beta|$, $0\leq e\leq m$, and $i,j,k\in[n]$,
\[
 \binom me
 K_{\lambda/\mu,\beta+(m-e)e_i+e\,e_j}
 K_{\lambda/\mu,\beta+e\,e_i+(m-e)e_k}
 \geq
 K_{\lambda/\mu,\beta+m e_i}
 K_{\lambda/\mu,\beta+e\,e_j+(m-e)e_k}.
\]
\item If $|\beta|=d-2$, the matrix
$\bigl(K_{\lambda/\mu,\beta+e_i+e_j}\bigr)_{i,j}$ has at most one positive eigenvalue, and exactly one when nonzero; its principal minors have the signs in~\eqref{eq:principal-sign}.
\item Whenever $\alpha\in\N^n$, $|\alpha|=d$, $i\neq j$, and $\alpha_i,\alpha_j\geq1$,
\[
 K_{\lambda/\mu,\alpha}^2
 \geq K_{\lambda/\mu,\alpha+e_i-e_j}
          K_{\lambda/\mu,\alpha-e_i+e_j}.
\]
The nonzero coefficients in each coordinate root direction form a log-concave unimodal interval.
\end{enumerate}
\end{corollary}

\begin{proof}
If $s_{\lambda/\mu}(x_1,\ldots,x_n)=0$, all assertions are immediate from the stated zero convention.  Otherwise apply Theorems~\ref{thm:master-RKT} and~\ref{thm:master-HR} to $F=s_{\lambda/\mu}$, using Theorem~\ref{thm:main-A}.
\end{proof}

\section{Concrete type A support and extremal coefficients}\label{sec:concrete-A}

Let $\theta=\lambda/\mu$ be an ordinary skew shape.  For each column $j$, put
\begin{equation}\label{eq:ordinary-column-heights}
 c_j(\theta):=\lambda'_j-\mu'_j,
 \qquad
 \kappa_p(\theta):=\#\{j:c_j(\theta)\geq p\}.
\end{equation}
Thus $\kappa(\theta)$ is the partition conjugate to the weakly decreasing
rearrangement of the column lengths.  For the empty shape we use
$\max\varnothing=0$ and, after fixing $n$, set
$\kappa(\varnothing)=(0,\ldots,0)\in\mathbb Z^n$; thus the statement below
includes $s_\varnothing=1$.

\begin{theorem}[Exact ordinary skew-Schur support]\label{thm:ordinary-skew-support}
Let $n\geq1$.  If $\max_j c_j(\theta)>n$, then
$s_\theta(x_1,\ldots,x_n)=0$.  Otherwise, pad $\kappa(\theta)$ by zeros to
length $n$.  Then
\begin{equation}\label{eq:ordinary-skew-support}
 \Supp(s_\theta)=\mathcal P_{\kappa(\theta)}\cap\Z^n,
 \qquad
 \Newt(s_\theta)=\mathcal P_{\kappa(\theta)}.
\end{equation}
For $A\subseteq[n]$, the support-polymatroid rank is
\begin{equation}\label{eq:ordinary-skew-rank}
 r_\theta(A)
 =\sum_j\min\{|A|,c_j(\theta)\}
 =\sum_{p=1}^{|A|}\kappa_p(\theta).
\end{equation}
Every vertex monomial has coefficient one, and the support polymatroid is
linearly representable over every infinite field.
\end{theorem}

\begin{proof}
A semistandard tableau with entries in $[n]$ is strictly increasing down each
column, so a column of length greater than $n$ is impossible.  Assume henceforth
that every $c_j(\theta)\leq n$.

Let $T$ be a semistandard tableau of shape $\theta$ and content $\alpha$.  If
$A\subseteq[n]$ has cardinality $p$, then a strictly increasing column of
length $c_j$ contains at most $\min\{p,c_j\}$ entries from $A$.  Therefore
\begin{equation}\label{eq:ordinary-subset-upper}
 \alpha(A)\leq\sum_j\min\{p,c_j(\theta)\}
 =\sum_{q=1}^{p}\kappa_q(\theta).
\end{equation}

There is a canonical tableau attaining all prefix bounds.  The boxes in column
$j$ occupy rows $\mu'_j+1,\ldots,\lambda'_j$; define
\begin{equation}\label{eq:ordinary-canonical-tableau}
 T_0(i,j):=i-\mu'_j.
\end{equation}
Each column is $1,2,\ldots,c_j$ from top to bottom.  If both $(i,j)$ and
$(i,j+1)$ are boxes, then $\mu'_j\geq\mu'_{j+1}$, and hence
$T_0(i,j)\leq T_0(i,j+1)$.  Thus $T_0$ is semistandard and has content
$\kappa(\theta)$.  In particular, equality in~\eqref{eq:ordinary-subset-upper}
is attained for $A=[p]$ for every $p$.  Since $s_\theta$ is symmetric, the
maximum depends only on $|A|$, and~\eqref{eq:ordinary-skew-rank} follows.

The type~A volume theorem makes $\Supp(s_\theta)$ an $M$-convex set.  Applying
Theorem~\ref{thm:symmetric-permutahedron} with the just-computed rank function
proves~\eqref{eq:ordinary-skew-support} and linear representability.

It remains to prove the vertex coefficient.  Suppose a tableau has content
$\kappa(\theta)$.  For every $p$, its total number of entries in $[p]$ attains
the upper bound in~\eqref{eq:ordinary-subset-upper}.  Hence every column of
length $c$ contains exactly $\min\{p,c\}$ entries in $[p]$ for every $p$.
Strict increase then forces that column to be $1,2,\ldots,c$.  Thus the tableau
is $T_0$, so $[x^{\kappa(\theta)}]s_\theta=1$.  Symmetry gives coefficient one
at every permuted vertex.
\end{proof}

\begin{remark}
The Newton-polytope and extremal-coefficient conclusions are classical in
substance.  The proof above is included because it exhibits the complete
polymatroid rank function directly and shows how it is recovered from the
Richardson-volume theorem.
\end{remark}

We obtain the following consequence.

\begin{corollary}\label{cor:ordinary-maximal-constituent}
Write
\[
 s_{\lambda/\mu}=\sum_{\nu}c_{\mu\nu}^{\lambda}s_\nu.
\]
Then $\kappa(\lambda/\mu)$ is the unique dominance-maximal partition $\nu$ with $c_{\mu\nu}^{\lambda}>0$, and
\[
 c_{\mu,\kappa(\lambda/\mu)}^{\lambda}=1.
\]
\end{corollary}

\begin{proof}
Specialize to $n\geq d=|\lambda|-|\mu|$, so no Schur polynomial indexed by a partition of $d$ vanishes.  The Kostka dominance criterion gives $\Newt(s_\nu)=\mathcal P_\nu$.  Since the Littlewood--Richardson expansion has nonnegative coefficients, Theorem~\ref{thm:ordinary-skew-support} implies
\[
 c_{\mu\nu}^{\lambda}>0\quad\Longrightarrow\quad
 \mathcal P_\nu\subseteq\mathcal P_{\kappa(\lambda/\mu)},
\]
which is equivalent by Rado's theorem~\cite{Rado52} to $\nu\unlhd\kappa(\lambda/\mu)$.  The coefficient of the dominant vertex monomial $x^{\kappa(\lambda/\mu)}$ in the skew Schur polynomial is one.  Among the Schur summands with $\nu\unlhd\kappa(\lambda/\mu)$, that monomial can occur only for $\nu=\kappa(\lambda/\mu)$.  Its coefficient is therefore the displayed Littlewood--Richardson coefficient.
\end{proof}

\begin{remark}
McNamara's dominance bounds and the occurrence of both extreme constituents already determine the Newton polytope, and the extreme constituents occur with coefficient one~\cite[Proposition~3.1 and its proof]{McNamara08}.  Monical--Tokcan--Yong supply the saturated Newton-polytope statement~\cite[Corollary~5.10]{MTY19}.  The preceding argument is retained as a direct Lorentzian reconstruction of the complete support rank function.
\end{remark}

\section{Concrete type C coefficient consequences}\label{sec:concrete-C}

For strict $\mu\subseteq\lambda$, put
\[
 d=\abs\lambda-\abs\mu,
 \qquad
 K^P_{\lambda/\mu,\alpha}:=[x^\alpha]P_{\lambda/\mu}(x).
\]

\begin{corollary}\label{cor:shifted-tableau-package}
Fix $n\geq1$ and interpret $K^P_{\lambda/\mu,\gamma}$ as zero when
$\gamma\notin\N^n$ or $|\gamma|\neq d$.  The following hold.
\begin{enumerate}[label=\textup{(\roman*)}]
\item If $\abs\beta\leq d$, $m=d-\abs\beta$, $0\leq e\leq m$, and $i,j,k\in[n]$, then
\begin{equation}\label{eq:shifted-RKT}
 \binom me
 K^P_{\lambda/\mu,\beta+(m-e)e_i+e\,e_j}
 K^P_{\lambda/\mu,\beta+e\,e_i+(m-e)e_k}
 \geq
 K^P_{\lambda/\mu,\beta+m e_i}
 K^P_{\lambda/\mu,\beta+e\,e_j+(m-e)e_k}.
\end{equation}
\item For $\abs\beta=d-2$, the matrix
\[
 \left(K^P_{\lambda/\mu,\beta+e_i+e_j}\right)_{i,j}
\]
has at most one positive eigenvalue, and exactly one when nonzero; all principal minors have alternating signs as in~\eqref{eq:principal-sign}.
\item Whenever $\alpha\in\N^n$, $|\alpha|=d$, $i\neq j$, and $\alpha_i,\alpha_j\geq1$,
\begin{equation}\label{eq:shifted-root}
 \left(K^P_{\lambda/\mu,\alpha}\right)^2
 \geq
 K^P_{\lambda/\mu,\alpha+e_i-e_j}
 K^P_{\lambda/\mu,\alpha-e_i+e_j}.
\end{equation}
The nonzero shifted-tableau multiplicities along each coordinate root direction $e_i-e_j$ form a log-concave unimodal interval.
\item If $P_{\lambda/\mu}(x_1,\ldots,x_n)\neq0$, the weighted block counts of the factorially tilted marked-shifted-tableau content distribution are ultra-log-concave, and the full content vector satisfies the covariance bound~\eqref{eq:cov-bound-general}.
\item If $P_{\lambda/\mu}(x_1,\ldots,x_n)\neq0$, there is a unique partition $\kappa^P_{\lambda/\mu,n}$ such that
\begin{equation}\label{eq:implicit-P-support}
 \Supp(P_{\lambda/\mu}(x_1,\ldots,x_n))
 =\mathcal P_{\kappa^P_{\lambda/\mu,n}}\cap\Z^n.
\end{equation}
The same support statement holds for $Q_{\lambda/\mu}$.  The associated
support polymatroid is linearly representable over every infinite field and
its permutahedron has the integer-decomposition property.
\end{enumerate}
\end{corollary}

\begin{proof}
If $P_{\lambda/\mu}(x_1,\ldots,x_n)=0$, parts~\textup{(i)}--\textup{(iii)} are immediate.  Otherwise apply Theorems~\ref{thm:master-RKT} and~\ref{thm:master-HR} to $F=P_{\lambda/\mu}$.  Under the stated nonvanishing hypothesis, Theorems~\ref{thm:symmetric-permutahedron} and~\ref{thm:block-covariance} give parts~\textup{(iv)} and~\textup{(v)}.  The $Q$-polynomial differs by the positive scalar in Lemma~\ref{lem:skew-PQ-scaling}, so it has the same support and satisfies the same homogeneous quadratic inequalities.
\end{proof}

We get the following observation.

\begin{corollary}\label{cor:tableau-dominance}
Let $n\geq1$, let $\theta$ be an ordinary skew shape of size $d$, let
$\vartheta=\lambda/\mu$ be a strict skew shape of size $d$, and write
$K_{\theta,\alpha}:=[x^\alpha]s_\theta(x_1,\ldots,x_n)$ and
$K^P_{\vartheta,\alpha}:=[x^\alpha]P_\vartheta(x_1,\ldots,x_n)$.  Let
$\rho,\tau$ be partitions of $d$ with at most $n$ parts.  If
$\rho\unlhd\tau$, then
\[
 K_{\theta,\rho}\geq K_{\theta,\tau},
 \qquad
 K^P_{\vartheta,\rho}\geq K^P_{\vartheta,\tau}.
\]
The analogous inequality holds for the coefficients of $Q_\vartheta$.
Consequently, among partition-indexed contents, every nonzero ordinary or
shifted tableau array is maximized at $\operatorname{bal}_n(d)$ and minimized
on its dominance-maximal support partition.
\end{corollary}

\begin{proof}
For a zero specialization the corresponding inequality is immediate.  For
each nonzero specialization, the factorial normalization is a realizable
volume polynomial by Theorem~\ref{thm:main-A} or~\ref{thm:main-PQ}; apply
Theorem~\ref{thm:dominance-monotonicity}.  The $Q$-coefficient array is a
positive scalar multiple of the $P$-coefficient array by
Lemma~\ref{lem:skew-PQ-scaling}.
\end{proof}

We obtain the following consequence.

\begin{corollary}\label{cor:straight-P-support}
Let $\lambda$ be a strict partition with $\ell(\lambda)\leq n$, padded by zeros to length $n$.  Then
\begin{equation}\label{eq:straight-P-support}
 \Supp(P_\lambda(x_1,\ldots,x_n))
 =\mathcal P_\lambda\cap\Z^n,
 \qquad
 \Newt(P_\lambda)=\mathcal P_\lambda.
\end{equation}
The same support and Newton-polytope identities hold for $Q_\lambda$.  Every vertex monomial of $P_\lambda$ has coefficient one, and the corresponding vertex coefficient of $Q_\lambda$ is $2^{\ell(\lambda)}$.

More generally, for strict partitions $\lambda^{(1)},\ldots,\lambda^{(m)}$ with at most $n$ parts,
\begin{equation}\label{eq:straight-P-product-support}
 \Supp\left(\prod_{a=1}^mP_{\lambda^{(a)}}\right)
 =\mathcal P_{\lambda^{(1)}+\cdots+\lambda^{(m)}}\cap\Z^n,
\end{equation}
and every vertex coefficient of the product is one.  For the analogous product of $Q$-polynomials, every vertex coefficient is
$2^{\sum_a\ell(\lambda^{(a)})}$.
\end{corollary}

\begin{proof}
The Schur expansion of a Schur $P$-function is dominance-unitriangular and Schur-positive:
\begin{equation}\label{eq:P-Schur-unitriangular}
 P_\lambda=s_\lambda+
 \sum_{\nu\lhd\lambda}a_{\lambda\nu}s_\nu,
 \qquad a_{\lambda\nu}\in\N;
\end{equation}
see Macdonald~\cite[Chapter~III, equation~(8.17)(ii)]{Macdonald} and Shaw--van Willigenburg~\cite[Section~1]{ShawVW07}.  For every surviving term, namely every $\nu$ with $\ell(\nu)\leq n$, the classical Kostka nonvanishing criterion gives
$\Supp(s_\nu(x_1,\ldots,x_n))=\mathcal P_\nu\cap\Z^n$; terms with
$\ell(\nu)>n$ vanish after specialization; see
Macdonald~\cite[Chapter~I, Section~6]{Macdonald}.  If $\nu\lhd\lambda$, then Rado's dominance criterion gives
$\mathcal P_\nu\subseteq\mathcal P_\lambda$.  Hence every monomial in the surviving terms on the right side of~\eqref{eq:P-Schur-unitriangular} lies in
$\mathcal P_\lambda$, while the summand $s_\lambda$ already contains every lattice point of $\mathcal P_\lambda$.  This proves~\eqref{eq:straight-P-support}.

At the dominant vertex $\lambda$, the coefficient in $s_\lambda$ is one: the unique semistandard tableau of shape and content $\lambda$ has every box in row $i$ filled with $i$.  Symmetry gives the same coefficient at every vertex.  No term $s_\nu$ with $\nu\lhd\lambda$ contains such a vertex, because its Newton polytope is a proper subpolytope of $\mathcal P_\lambda$.  Thus the vertex coefficient is one.  The assertion for $Q_\lambda$ follows from $Q_\lambda=2^{\ell(\lambda)}P_\lambda$.

The product statements follow from Theorem~\ref{thm:symmetric-permutahedron}: the associated maximal partitions add, and the product vertex coefficient is the product of the individual vertex coefficients.
\end{proof}

\begin{remark}
The support and Newton-polytope assertions for straight Schur $P$- and $Q$-functions were proved by Monical--Tokcan--Yong~\cite[Proposition~3.5]{MTY19}.  We retain the proof above to identify the dominant partition, vertex coefficients, and product supports within the uniform symmetric-$M$-convex framework used throughout this paper.
\end{remark}

\subsection{Two-row ordinary and weighted shifted Littlewood--Richardson consequences}

Define nonnegative integers $g_{\mu\nu}^{\lambda}$ by
\begin{equation}\label{eq:shifted-LR-def}
 Q_\mu Q_\nu=\sum_\lambda g_{\mu\nu}^{\lambda}Q_\lambda.
\end{equation}
Equivalently, by Hopf duality,
\begin{equation}\label{eq:skew-P-expansion}
 P_{\lambda/\mu}=\sum_\nu g_{\mu\nu}^{\lambda}P_\nu.
\end{equation}
These are one standard normalization of the shifted Littlewood--Richardson coefficients.  Their nonnegativity follows from the shifted Littlewood--Richardson rule of Stembridge~\cite[Theorem~8.3]{Stembridge89}.

\begin{lemma}\label{lem:P-two-variable}
For $r\geq1$,
\begin{equation}\label{eq:P-r-two-vars}
 P_{(r)}(x,y)=x^r+y^r+2\sum_{a=1}^{r-1}x^{r-a}y^a.
\end{equation}
For $a>b>0$,
\begin{equation}\label{eq:P-ab-two-vars}
 P_{(a,b)}(x,y)=(xy)^bP_{(a-b)}(x,y).
\end{equation}
\end{lemma}

\begin{proof}
Equation~\eqref{eq:P-r-two-vars} follows by taking the coefficient of $t^r$ in
\[
 \frac{1+xt}{1-xt}\frac{1+yt}{1-yt}
 =\sum_{r\geq0}Q_{(r)}(x,y)t^r
\]
and using $P_{(r)}=Q_{(r)}/2$ for $r>0$.

For a strict partition $(a,b)$ of length two, Macdonald's symmetrization formula~\cite[Chapter~III, (8.3)]{Macdonald} gives
\[
 P_{(a,b)}(x,y)
 =\frac{x+y}{x-y}\left(x^ay^b-x^by^a\right)
 =(xy)^b\frac{x+y}{x-y}\left(x^{a-b}-y^{a-b}\right).
\]
The final factor is $P_{(a-b)}(x,y)$ by~\eqref{eq:P-r-two-vars}.
\end{proof}

\begin{corollary}[Cumulative two-row Littlewood--Richardson coefficients]\label{cor:ordinary-two-row-LR}
Let $\mu\subseteq\lambda$ be ordinary partitions, put $d=|\lambda|-|\mu|$, and write
\[
 s_{\lambda/\mu}=\sum_{\nu\vdash d}c_{\mu\nu}^{\lambda}s_\nu.
\]
For $0\leq r\leq\lfloor d/2\rfloor$, define
\[
 A_r:=\sum_{j=0}^r c_{\mu,(d-j,j)}^{\lambda}.
\]
Then
\begin{equation}\label{eq:ordinary-LR-cumulative}
 A_r=[x^{d-r}y^r]s_{\lambda/\mu}(x,y),
 \qquad
 A_r^2\geq A_{r-1}A_{r+1}
\end{equation}
for $1\leq r<\lfloor d/2\rfloor$.  Equivalently, if
$c_r=c_{\mu,(d-r,r)}^{\lambda}$, then
\begin{equation}\label{eq:ordinary-LR-increment}
 A_rc_r\geq A_{r-1}c_{r+1}.
\end{equation}
\end{corollary}

\begin{proof}
In two variables, only the Schur polynomials indexed by $(d-j,j)$ survive, and
\[
 s_{(d-j,j)}(x,y)=\sum_{r=j}^{d-j}x^{d-r}y^r.
\]
Taking the coefficient of $x^{d-r}y^r$ in the Littlewood--Richardson expansion proves the first identity in~\eqref{eq:ordinary-LR-cumulative}.  The normalized skew Schur polynomial is realizable volume by the type~A theorem, so Theorem~\ref{thm:master-HR}~\textup{(iii)} gives the log-concavity inequality.  Finally,
$A_r-A_{r-1}=c_r$ and $A_{r+1}-A_r=c_{r+1}$; rearranging
$A_r^2\geq A_{r-1}A_{r+1}$ gives~\eqref{eq:ordinary-LR-increment}.
\end{proof}

The following result is about two-row shifted Littlewood--Richardson coefficients.

\begin{theorem}\label{thm:shifted-LR-logconcavity}
Assume $d:=\abs\lambda-\abs\mu\geq1$ and
\[
 m=\left\lfloor\frac{d-1}{2}\right\rfloor,
 \qquad
 \nu_0=(d),
 \qquad
 \nu_j=(d-j,j)\quad(1\leq j\leq m).
\]
Put
\[
 g_j:=g_{\mu,\nu_j}^{\lambda},
 \qquad
 B_r:=[x^{d-r}y^r]P_{\lambda/\mu}(x,y).
\]
Thus $B_r$ is a weighted cumulative sum of the two-row shifted Littlewood--Richardson coefficients: interior contributions occur with weight two.  More precisely, $B_r=B_{d-r}$ and
\begin{align}
 B_0&=g_0,
 \label{eq:B0}\\
 B_r&=2\sum_{j=0}^{r-1}g_j+g_r
 \qquad(1\leq r\leq m),
 \label{eq:Br}\end{align}
while, if $d$ is even,
\begin{equation}\label{eq:Bcenter}
 B_{d/2}=2\sum_{j=0}^{m}g_j.
\end{equation}
The sequence $B_0,B_1,\ldots,B_d$ has no internal zeros and is log-concave:
\begin{equation}\label{eq:B-logconcave}
 B_r^2\geq B_{r-1}B_{r+1}\qquad(0<r<d).
\end{equation}
In particular, for $1\leq r<m$,
\begin{equation}\label{eq:g-cumulative-ineq}
 B_r(g_{r-1}+g_r)
 \geq
 B_{r-1}(g_r+g_{r+1}).
\end{equation}
\end{theorem}

\begin{proof}
After specialization to two variables, only strict partitions of length at most two survive in~\eqref{eq:skew-P-expansion}.  Lemma~\ref{lem:P-two-variable} shows that $P_{\nu_j}$ is supported on
\[
 x^{d-j}y^j, x^{d-j-1}y^{j+1},\ldots, x^j y^{d-j},
\]
with coefficient $1$ at the two endpoints and coefficient $2$ at every interior monomial.  Summing the contributions in~\eqref{eq:skew-P-expansion} gives~\eqref{eq:B0}--\eqref{eq:Bcenter}; symmetry is clear.

Equation~\eqref{eq:B-logconcave} is the root-direction inequality~\eqref{eq:shifted-root} in two variables.  Absence of internal zeros follows from $M$-convexity.  For $1\leq r<m$, formulas~\eqref{eq:Br} give
\[
 B_r-B_{r-1}=g_{r-1}+g_r,
 \qquad
 B_{r+1}-B_r=g_r+g_{r+1}.
\]
Rearranging $B_r^2\geq B_{r-1}B_{r+1}$ yields~\eqref{eq:g-cumulative-ineq}.
\end{proof}

\begin{corollary}[Mixed products]\label{cor:mixed-products}
Let $\theta_a$ be ordinary skew shapes and let $\eta_b,\zeta_c$ be strict skew shapes.  Then
\begin{equation}\label{eq:mixed-product}
 \cN\left(
  \prod_a s_{\theta_a}(x)
  \prod_b P_{\eta_b}(x)
  \prod_c Q_{\zeta_c}(x)
 \right)
\end{equation}
is a realizable volume polynomial.  If the product is nonzero, all conclusions of Theorems~\ref{thm:master-RKT},~\ref{thm:master-HR},~\ref{thm:symmetric-permutahedron}, and~\ref{thm:block-covariance} apply; in particular, its support is exactly the set of lattice points of its Newton polytope.  If it is zero, the homogeneous coefficient inequalities are trivial.
\end{corollary}

\begin{proof}
Each factor has realizable factorial normalization by the type~A theorem and Theorem~\ref{thm:main-PQ}.  Apply Proposition~\ref{prop:permanence}(iv), followed by the master theorems.
\end{proof}

\appendix

\section{The type \texorpdfstring{$D$}{D} spinor companion}
\label{appA:main}

\begin{center}
{\large\textsc{Zhenpeng Wang}\footnote{Department of Mathematics, The University of Hong Kong, u3011717@connect.hku.hk}}
\end{center}
\smallskip

\noindent This appendix is contributed by Zhenpeng Wang.  We construct the spinor counterpart of the Lagrangian model in Section~\ref{sec:typeC}.
The first two subsections give the type~$D$ cycle transform, the Richardson
realization of skew Schur $Q$-functions, and the associated
projective-bundle formula.  We then compare the Schubert geometry of the
Lagrangian and spinor constructions.  In characteristic two this comparison
is induced by the exceptional isogenies between the symplectic and odd
orthogonal groups, and the resulting morphisms explain both the power of two
relating the skew $P$- and $Q$-normalizations and the one-step difference
between the two projective-bundle shifts.

We use the notation and conventions of Sections~\ref{sec:preliminaries},~\ref{sec:total-chern}, and~\ref{sec:typeC}.  Thus
$\Omega_{\mathbb Q}$ is the rational Schur $P/Q$ Hopf algebra,
$\Gamma_{\mathbb Z}=\bigoplus_{\nu\in\mathcal{SP}}\mathbb ZQ_\nu$, and
$\langle P_\lambda,Q_\nu\rangle=\delta_{\lambda\nu}$.  Projective bundles
are taken in Grothendieck's quotient convention.  Integration over a
possibly singular variety means operational-Chow degree; Chow groups have
integral coefficients unless a subscript $\mathbb Q$ is displayed.

\subsection{The spinor Schubert specialization}
\label{appA:spinor-specialization}

Let $W$ be a complex vector space of dimension $2N+2$ endowed with a
nondegenerate quadratic form, and fix one component
\[
  G_D:=\operatorname{OG}^{+}(N+1,2N+2)
\]
of the Grassmannian of maximal isotropic $(N+1)$-planes.  Thus
$G_D\simeq \operatorname{Spin}(W)/P_{\mathrm{spin}}$, where $P_{\mathrm{spin}}$ is the
spinor maximal parabolic corresponding to the chosen component.  Its universal sequence
is
\begin{equation}
\label{appA:eq:universal-D}
  0\longrightarrow\mathcal S\longrightarrow
  W\otimes\mathcal O_{G_D}\longrightarrow\mathcal Q_D\longrightarrow0.
\end{equation}
The quadratic form identifies $\mathcal Q_D\simeq\mathcal S^\vee$, and hence
\begin{equation}
\label{appA:eq:Chern-relation-D}
  c_t(\mathcal Q_D)c_{-t}(\mathcal Q_D)=1.
\end{equation}

Put
\[
  \rho_N=(N,N-1,\ldots,1),
  \qquad
  M=|\rho_N|=\frac{N(N+1)}2.
\]
Strict partitions $\nu\subseteq\rho_N$ index Schubert classes
$\tau_\nu\in A^{|\nu|}(G_D)$.  For such a partition, let
\[
  \nu^\vee:=\rho_N\setminus\nu,
\]
the strict partition whose parts are the elements of $\{1,\ldots,N\}$ not
occurring in $\nu$.  Then
\[
  |\nu^\vee|=M-|\nu|,
  \qquad
  \ell(\nu^\vee)=N-\ell(\nu),
\]
and
\begin{equation}
\label{appA:eq:D-pairing}
  \int_{G_D}\tau_{\nu^\vee}\tau_\eta=\delta_{\nu\eta}.
\end{equation}
The special classes satisfy
\begin{equation}
\label{appA:eq:D-special-classes}
  c_r(\mathcal Q_D)=2\tau_{(r)}\quad(1\le r\le N),
  \qquad
  c_{N+1}(\mathcal Q_D)=0.
\end{equation}
These are the standard integral normalizations of spinor Schubert calculus;
see \cite{KT04,PR97}.  Since $G_D$ is cellular, its integral Chow ring
is identified with its even integral cohomology.

Let
\[
  \Gamma^P_{\mathbb Z}
  :=\bigoplus_{\nu\in\mathcal{SP}}\mathbb ZP_\nu
  \subset\Omega_{\mathbb Q}.
\]
This is the integral Schur-$P$ subring; as a lattice, it is dual to
$\Gamma_{\mathbb Z}$ under the pairing above
\cite[Chapter~III, Section~8]{Macdonald}.  Since
$q_r=Q_{(r)}=2P_{(r)}$, every generator $q_r$ belongs to
$\Gamma^P_{\mathbb Z}$.

\begin{proposition}[Spinor Schubert specialization]
\label{appA:prop:spinor-specialization}
There is a graded $\mathbb Q$-algebra homomorphism
\[
  \Phi^D_{N,\mathbb Q}:\Omega_{\mathbb Q}
  \longrightarrow A^*(G_D)_{\mathbb Q},
  \qquad
  q_r\longmapsto c_r(\mathcal Q_D),
\]
which restricts to a surjective graded ring homomorphism
\[
  \varphi_N^D:\Gamma^P_{\mathbb Z}\longrightarrow A^*(G_D).
\]
It satisfies $\varphi_N^D(q_r)=c_r(\mathcal Q_D)$ for every $r\geq0$.
For every strict partition $\nu$,
\begin{equation}
\label{appA:eq:D-specialization}
  \varphi_N^D(P_\nu)=
  \begin{cases}
    \tau_\nu,&\nu\subseteq\rho_N,\\
    0,&\nu\not\subseteq\rho_N.
  \end{cases}
\end{equation}
Consequently,
\begin{equation}
\label{appA:eq:D-specialization-Q}
  \Phi^D_{N,\mathbb Q}(Q_\nu)
  =2^{\ell(\nu)}\tau_\nu
  \qquad(\nu\subseteq\rho_N).
\end{equation}
\end{proposition}

\begin{proof}
Equation~\eqref{appA:eq:Chern-relation-D} is precisely the relation
$q(t)q(-t)=1$ in the standard presentation of $\Omega_{\mathbb Q}$; it
therefore defines $\Phi^D_{N,\mathbb Q}$.  Since
$P_{(r)}=q_r/2$, equations~\eqref{appA:eq:D-special-classes} give
$\Phi^D_{N,\mathbb Q}(P_{(r)})=\tau_{(r)}$ for $1\le r\le N$.
Under the substitution $q_r\mapsto c_r(\mathcal Q_D)$, the classical
two-row and Pfaffian formulas for $P_\nu$ become the modified
$\widetilde P_\nu$ Giambelli polynomials.  Here and below,
$\widetilde P_\nu(c_1(\mathcal Q_D),c_2(\mathcal Q_D),\ldots)$ means the
universal integral spinor Giambelli polynomial in the special classes
$\tau_{(r)}=c_r(\mathcal Q_D)/2$; it is not the evaluation of a rational
polynomial with denominators on an arbitrary sequence of Chern classes.
The integral spinor Giambelli formula therefore gives
\[
  \Phi^D_{N,\mathbb Q}(P_\nu)
  =\widetilde P_\nu\bigl(c_1(\mathcal Q_D),c_2(\mathcal Q_D),\ldots\bigr)
  =\tau_\nu
  \qquad(\nu\subseteq\rho_N);
\]
see \cite{PR97} and \cite[\S3.1]{KT04}.  These images are integral.  If
$\nu_1>N$, expand the Pfaffian along the row containing $\nu_1$.  In each
corresponding two-row entry, every summand contains a one-row factor of index
at least $N+1$.  Its image is zero by
\eqref{appA:eq:D-special-classes} and the rank of $\mathcal Q_D$.  This proves
\eqref{appA:eq:D-specialization} and shows that
$\Phi^D_{N,\mathbb Q}$ restricts to $\Gamma^P_{\mathbb Z}$.  The special
Schubert classes generate $A^*(G_D)$, so the restriction is surjective.
Finally, \eqref{appA:eq:D-specialization-Q} follows from
$Q_\nu=2^{\ell(\nu)}P_\nu$.
\end{proof}

For an irreducible projective subvariety $X\subseteq G_D$ of dimension $d$
and an integer $n\ge1$, define
\begin{equation}
\label{appA:eq:D-cycle-transform}
  d_\nu^D(X):=\int_X\tau_\nu|_X,
  \qquad
  \Theta^D_{X,n}(x):=
  \sum_{\substack{\nu\subseteq\rho_N\text{ strict}\\|\nu|=d}}
  d_\nu^D(X)Q_\nu(x_1,\ldots,x_n).
\end{equation}
The integers $d_\nu^D(X)$ are nonnegative: by Kleiman transversality, a
general translate of the Schubert cycle representing $\tau_\nu$ meets $X$
properly in an effective zero-cycle.

\begin{theorem}[Type $D$ cycle transform]
\label{appA:thm:D-cycle-transform}
For every $\alpha\in\mathbb N^n$ with $|\alpha|=d$,
\begin{equation}
\label{appA:eq:D-cycle-coeff}
  [x^\alpha]\Theta^D_{X,n}(x)
  =\int_X\prod_{i=1}^n c_{\alpha_i}(\mathcal Q_D|_X).
\end{equation}
Equivalently,
\begin{equation}
\label{appA:eq:D-cycle-Chern}
  \Theta^D_{X,n}(x)
  =\int_X\prod_{i=1}^n c_{x_i}(\mathcal Q_D|_X).
\end{equation}
Consequently, $\cN(\Theta^D_{X,n})$ is a realizable volume polynomial
over $\mathbb C$.
\end{theorem}

\begin{proof}
The second form of the Schur $P/Q$ Cauchy identity gives
\[
  \prod_{i=1}^n q_{\alpha_i}
  =\sum_{\nu\in\mathcal{SP}}[x^\alpha]Q_\nu(x)P_\nu.
\]
Apply Proposition~\ref{appA:prop:spinor-specialization}, restrict to $X$, and
take degree.  This proves \eqref{appA:eq:D-cycle-coeff}; summing over
$|\alpha|=d$ gives \eqref{appA:eq:D-cycle-Chern}.  The last assertion follows from Theorem~\ref{thm:total-chern-volume},
applied to $n$ copies of the globally generated bundle $\mathcal Q_D|_X$.
\end{proof}

\subsection{Spinor Richardson varieties}
\label{appA:spinor-Richardson}

Fix strict partitions $\mu\subseteq\lambda$ and work in the stable range
$N\ge\max\{1,\lambda_1\}$.  For a pair of opposite isotropic flags, define the spinor Richardson variety by
\begin{equation}
\label{appA:eq:D-Richardson}
  R^D_{\lambda/\mu}
  :=X^D_{\lambda^\vee}(F_\bullet)
    \cap (X^D)^\mu(F_\bullet^{\mathrm{opp}}).
\end{equation}
Standard Richardson theory gives that $R^D_{\lambda/\mu}$ is integral and
projective, of dimension $d=|\lambda|-|\mu|$ \cite[Section~1.3]{Brion05}, and
\begin{equation}
\label{appA:eq:D-Richardson-class}
  i_*[R^D_{\lambda/\mu}]
  =\tau_{\lambda^\vee}\tau_\mu\cap[G_D].
\end{equation}

\begin{theorem}[Spinor Richardson formula]
\label{appA:thm:D-Richardson}
Let $\mathcal Q_R=\mathcal Q_D|_{R^D_{\lambda/\mu}}$.  For every
$\alpha\in\mathbb N^n$ with $|\alpha|=d$,
\begin{equation}
\label{appA:eq:D-Richardson-coeff}
  [x^\alpha]Q_{\lambda/\mu}(x_1,\ldots,x_n)
  =\int_{R^D_{\lambda/\mu}}
    \prod_{i=1}^n c_{\alpha_i}(\mathcal Q_R).
\end{equation}
Equivalently,
\begin{equation}
\label{appA:eq:D-Richardson-poly}
  Q_{\lambda/\mu}(x_1,\ldots,x_n)
  =\int_{R^D_{\lambda/\mu}}
    \prod_{i=1}^n c_{x_i}(\mathcal Q_R).
\end{equation}
\end{theorem}

\begin{proof}
The dual skew-Cauchy identity gives
\[
  [x^\alpha]Q_{\lambda/\mu}
  =[P_\lambda]\left(P_\mu\prod_{i=1}^nq_{\alpha_i}\right).
\]
Expand the expression in the $P$-basis and apply $\varphi_N^D$.
Terms indexed by strict partitions with first part greater than $N$ vanish by
Proposition~\ref{appA:prop:spinor-specialization}; among the surviving terms,
pairing with $\tau_{\lambda^\vee}$ extracts precisely the coefficient of
$P_\lambda$, because $\lambda\subseteq\rho_N$.  Equation~\eqref{appA:eq:D-Richardson-class} and the projection formula identify the
resulting ambient intersection number with the right-hand side of
\eqref{appA:eq:D-Richardson-coeff}.  Summing over $|\alpha|=d$ proves
\eqref{appA:eq:D-Richardson-poly}.
\end{proof}

Define
\begin{equation}
\label{appA:eq:D-projective-bundle}
  Y^D_{\lambda/\mu,n}
  :=\bigl(\mathbb P_{R^D_{\lambda/\mu}}(\mathcal Q_R)\bigr)
    ^{\times_{R^D_{\lambda/\mu}}n},
\end{equation}
and let $\xi_i$ be the tautological quotient class on the $i$th factor.

\begin{corollary}[Spinor projective-bundle model]
\label{appA:cor:D-projective-bundle}
Put $D_D=d+nN$.  Then
\begin{equation}
\label{appA:eq:D-volume-model}
  \frac{1}{D_D!}
  \int_{Y^D_{\lambda/\mu,n}}
  \left(\sum_{i=1}^n x_i\xi_i\right)^{D_D}
  =\cN\!\left((x_1\cdots x_n)^NQ_{\lambda/\mu}(x)\right).
\end{equation}
\end{corollary}

\begin{proof}
In the universal sequence~\eqref{appA:eq:universal-D}, the evaluation kernel
is $\mathcal S$ and its dual is $\mathcal Q_D$.  Apply the explicit projective-bundle identity in
Theorem~\ref{thm:total-chern-volume} to $n$ copies of $\mathcal Q_R$.
Since $\operatorname{rk}\mathcal Q_R=N+1$, the shift is $N$; the coefficient
identification is Theorem~\ref{appA:thm:D-Richardson}.
\end{proof}

\subsection{The common integral Schubert geometry}
\label{appA:common-Schubert-geometry}

Assume first that $N\geq2$.  Let $\mathbf G_B$, $\mathbf G_C$, and
$\mathbf G_D$ be the split simply connected Chevalley group schemes over
$\mathbb Z$ of types $B_N$, $C_N$, and $D_{N+1}$, respectively.  Choose
split pinnings
\[
 (\mathbf T_B,\mathbf B_B,\{x^B_\gamma\}_{\gamma\in\Phi_B}),
 \qquad
 (\mathbf T_C,\mathbf B_C,\{x^C_\delta\}_{\delta\in\Phi_C}),
\]
and choose a split pinning of $\mathbf G_D$.  Let
$\mathbf P_B\supseteq\mathbf B_B$ and
$\mathbf P_C\supseteq\mathbf B_C$ be the maximal parabolics corresponding
to the terminal nodes, and let $\mathbf P_D\subseteq\mathbf G_D$ be the
spinor maximal parabolic corresponding to the chosen component.  Put
\[
 \mathbf X_B:=\mathbf G_B/\mathbf P_B,
 \qquad
 \mathbf X_C:=\mathbf G_C/\mathbf P_C,
 \qquad
 \mathbf X_D:=\mathbf G_D/\mathbf P_D.
\]
Over $\mathbb C$, the last two homogeneous schemes recover the varieties
used above:
\[
 \mathbf X_{C,\mathbb C}\simeq\operatorname{LG}(N,2N)=G_C,
 \qquad
 \mathbf X_{D,\mathbb C}\simeq
 \operatorname{OG}^{+}(N+1,2N+2)=G_D.
\]
Moreover, the standard odd/even spinor isomorphism
\[
 \mathbf X_{B,\mathbb C}\simeq\operatorname{OG}(N,2N+1)
 \xrightarrow{\sim}G_D
\]
identifies the strict-partition Schubert stratifications; see
\cite[Introduction and \S3.1]{KT04}.  The Weyl groups of types $B_N$ and
$C_N$ are both the signed permutation group, and their terminal-parabolic
quotients have the same Bruhat poset, indexed by strict partitions
$\nu\subseteq\rho_N$.  Via the odd/even spinor identification, the same
indexing is used on $G_D$.  Thus $R^C_{\lambda/\mu}$ and
$R^D_{\lambda/\mu}$ realize the same interval $[\mu,\lambda]$ in the two
Bruhat posets.

\begin{lemma}[Integral Schubert calculus and base change]
\label{appA:lem:integral-base-change}
Let $K$ be an algebraically closed field.  The Chow rings of
$\mathbf X_{C,K}$ and $\mathbf X_{D,K}$ are free abelian on their
Schubert bases, and base change carries each Schubert class to the class
with the same Weyl-group index.  If $\mathcal E$ and $\mathcal Q_D$ denote
the corresponding tautological quotient bundles, then, in integral Chow,
\begin{align}
 c_r(\mathcal E)&=\sigma_{(r)} &&(1\leq r\leq N),
 &\sigma_\nu&=\widetilde Q_\nu
   \bigl(c_1(\mathcal E),c_2(\mathcal E),\ldots\bigr),
 \label{appA:eq:integral-Giambelli-C}\\
 c_r(\mathcal Q_D)&=2\tau_{(r)} &&(1\leq r\leq N),
 &c_{N+1}(\mathcal Q_D)&=0,
 \label{appA:eq:integral-special-D}\\
 &&&
 &\tau_\nu&=\widetilde P_\nu
   \bigl(c_1(\mathcal Q_D),c_2(\mathcal Q_D),\ldots\bigr)
 \label{appA:eq:integral-Giambelli-D}
\end{align}
for every strict $\nu\subseteq\rho_N$.  Here
$\widetilde P_\nu(\mathcal Q_D)$ denotes the universal integral spinor
Giambelli polynomial in the special classes
$\tau_{(r)}=c_r(\mathcal Q_D)/2$; it is not an a priori evaluation of a
polynomial with denominators in an arbitrary sequence of Chern classes.
In particular, these identities remain valid in characteristic two.
\end{lemma}

\begin{proof}
We distinguish the integral models from their fibers.  The schemes
$\mathbf X_C$ and $\mathbf X_D$ are smooth and projective over
$\operatorname{Spec}\mathbb Z$.  The relative Bruhat decomposition for a
split reductive group scheme and a standard parabolic gives each
$\mathbf X_\bullet$ a finite filtration by closed subschemes whose strata
are affine spaces $\mathbb A^{\ell(w)}_{\mathbb Z}$ indexed by
$W_\bullet^{P_\bullet}$, the minimal-length representatives for
$W_\bullet/W_{P_\bullet}$; see
\cite[Expos\'e~XXVI, \S5]{SGA3}.  By homotopy invariance,
$A^*(\mathbb A^m_{\mathbb Z})\simeq
A^*(\operatorname{Spec}\mathbb Z)=\mathbb Z$ in the relevant shifted
grading, so each relative cell contributes one generator.  Successive
application of the localization sequence shows that the relative Schubert
closures generate $A^*(\mathbf X_\bullet)$ as an abelian group: at each
stage the closure of the newly adjoined cell restricts to the fundamental
class of that affine-space stratum.  To prove linear independence, suppose
that
\[
  \sum_{w\in W_\bullet^{P_\bullet}}
  a_w[\overline{C_w}]=0.
\]
Flat base change to $\mathbb C$ gives the corresponding relation among the
ordinary Schubert classes on $\mathbf X_{\bullet,\mathbb C}$.  Those
classes form a $\mathbb Z$-basis by the standard cellular decomposition
over a field, so every $a_w$ is zero.  Hence
\begin{equation}
\label{appA:eq:relative-cellular-Chow}
 A^*(\mathbf X_\bullet)
 =\bigoplus_{w\in W_\bullet^{P_\bullet}}\mathbb Z[\overline{C_w}],
\end{equation}
as a graded abelian group, with each relative Schubert closure in its
natural codimension.  This is the relative form of the standard cellular
Chow argument~\cite[Sections~1.8--1.9, especially Example~1.9.1]{FultonIT}.
Here we use smoothness to identify operational Chow cohomology with
homological Chow groups by cap product with the fundamental class.

Let $j_\eta:\mathbf X_{\bullet,\mathbb Q}\to\mathbf X_\bullet$ be the
generic-fiber morphism.  The restriction map to the generic fiber, obtained
as the filtered colimit of the flat pullbacks over nonempty open subschemes
of $\operatorname{Spec}\mathbb Z$, sends every relative Schubert class in
\eqref{appA:eq:relative-cellular-Chow} to the Schubert class with the same
Weyl index on the generic fiber.  The latter classes form a
$\mathbb Z$-basis by the identical cellular argument over $\mathbb Q$;
hence
\[
 j_\eta^*:A^*(\mathbf X_\bullet)
 \longrightarrow A^*(\mathbf X_{\bullet,\mathbb Q})
\]
is injective.  For every algebraically closed field $K$, base change of the
relative cells gives the usual Bruhat cells on $\mathbf X_{\bullet,K}$, and
the same argument gives
\[
 A^*(\mathbf X_{\bullet,K})
 =\bigoplus_{w\in W_\bullet^{P_\bullet}}\mathbb Z[\overline{C_{w,K}}].
\]
We spell out the compatibility of the classes with a geometric fiber.  In
characteristic zero the base change is flat.  In characteristic $p>0$, the
morphism $\mathbf X_{\bullet,K}\to\mathbf X_\bullet$ factors as the flat
map $\mathbf X_{\bullet,K}\to\mathbf X_{\bullet,\mathbb F_p}$ followed by
the regular closed immersion
$\mathbf X_{\bullet,\mathbb F_p}\hookrightarrow\mathbf X_\bullet$.
Under the smooth-ambient identification of operational and homological Chow
groups, operational pullback is refined Gysin pullback along the latter map,
followed by flat pullback along the former.  Each relative Schubert scheme
$\overline{C_w}$ is, by construction, the scheme-theoretic closure of its
generic fiber in $\mathbf X_\bullet$.  On every affine open its coordinate
ring therefore injects into its localization at the nonzero integers; it is
$\mathbb Z$-torsion-free and hence flat over $\mathbb Z$.

We also record explicitly the multiplicity of the top-dimensional component
of its special fiber.  The relative Bruhat closure relations show that
$C_w$ is an open subscheme of $\overline{C_w}$ and that its complement is
stratified by the cells $C_v$ with $v<w$.  After base change to
$\mathbb F_p$, the open cell
\[
  C_{w,\mathbb F_p}
  \simeq \mathbb A_{\mathbb F_p}^{\ell(w)}
\]
is reduced, whereas every stratum in its complement has dimension
$\ell(v)<\ell(w)$.  Hence the reduced support of
$(\overline{C_w})_{\mathbb F_p}$ is the usual Schubert variety
$\overline{C_{w,\mathbb F_p}}$, this is its unique top-dimensional
irreducible component, and the special fiber is reduced at the generic point
of that component.  Consequently, the component occurs with multiplicity
one in the fundamental cycle:
\[
  \bigl[(\overline{C_w})_{\mathbb F_p}\bigr]
  =
  \bigl[\overline{C_{w,\mathbb F_p}}\bigr].
\]

Since $\overline{C_w}$ is flat over $\mathbb Z$, refined Gysin pullback along
the regular closed immersion
$\mathbf X_{\bullet,\mathbb F_p}\hookrightarrow\mathbf X_\bullet$
sends $[\overline{C_w}]$ to
$\bigl[(\overline{C_w})_{\mathbb F_p}\bigr]$, and hence, by the preceding
equality, to the usual Schubert class
$\bigl[\overline{C_{w,\mathbb F_p}}\bigr]$.  The subsequent flat pullback
along
$\mathbf X_{\bullet,K}\to\mathbf X_{\bullet,\mathbb F_p}$
gives
$\bigl[\overline{C_{w,K}}\bigr]$.  Thus every relative Schubert class pulls
back to the class with the same Weyl-group index, and every fiber Chow group
is torsion-free.

We next establish the special classes before invoking the full Giambelli
formulas.  The tautological bundles and the relative Schubert cycles are
defined on the integral models.  Over $\mathbb C$, the standard special-class
identities are
\[
 c_r(\mathcal E)=\sigma_{(r)},\qquad
 c_r(\mathcal Q_D)=2\tau_{(r)},\qquad
 c_{N+1}(\mathcal Q_D)=0;
\]
see~\cite[Section~2]{KT03} and~\cite[\S3.1]{KT04}.  Each difference is an
integral class on the corresponding $\mathbb Z$-model.  Its restriction to
the generic fiber becomes zero after extension of scalars from $\mathbb Q$
to $\mathbb C$.  That scalar-extension map is injective because it carries
the Schubert basis over $\mathbb Q$ to the Schubert basis over $\mathbb C$.
The difference therefore vanishes on the generic fiber, and the injectivity
of $j_\eta^*$ forces it to vanish on the integral model.  Pullback gives the
special-class identities on every geometric fiber, including characteristic
two.

Only now do we spread the full Giambelli identities.  The modified
$\widetilde Q_\nu$ formulas are integral polynomials in the classes
$c_r(\mathcal E)$.  On the spinor side, the identities just proved show that
$c_r(\mathcal Q_D)$ is divisible by $2$ in integral Chow and identify
$c_r(\mathcal Q_D)/2$ with the integral special class $\tau_{(r)}$.
Consequently the standard Pfaffian expression
$\widetilde P_\nu(c_1(\mathcal Q_D),c_2(\mathcal Q_D),\ldots)$ is an
integral class: equivalently, it is the universal spinor Giambelli polynomial
in the classes $\tau_{(r)}$.  Over $\mathbb C$ these evaluations equal
$\sigma_\nu$ and $\tau_\nu$, respectively, by
\cite[Section~2, equations~(1)--(2)]{KT03},
\cite[\S3.1]{KT04}, and~\cite{PR97}.  Their differences from the corresponding
relative Schubert classes are therefore integral classes that vanish on the
generic fiber by the same scalar-extension argument.  Injectivity of
$j_\eta^*$ and then pullback to each geometric fiber prove
\eqref{appA:eq:integral-Giambelli-C}--
\eqref{appA:eq:integral-Giambelli-D}.  In particular, every displayed power
of $2$ is an ordinary integer coefficient in a torsion-free Chow group and
does not disappear in characteristic two.
\end{proof}

\begin{corollary}[Fiberwise Richardson and projective-bundle formulas]
\label{appA:cor:fiberwise-formulas}
Let $N\geq1$, let $K$ be an algebraically closed field, let $n\geq1$, and
let $\mu\subseteq\lambda\subseteq\rho_N$ be strict partitions.  For $N=1$,
use the standard identifications
$\mathbf X_{C,K}\simeq\mathbf X_{D,K}\simeq\mathbb P^1_K$; for $N\geq2$,
retain the integral models above.  Put $d=|\lambda|-|\mu|$, and let
$R^C_{\lambda/\mu,K}$ and $R^D_{\lambda/\mu,K}$ denote the Richardson
varieties in the corresponding geometric fibers, defined by base-changed
opposite flags.  Write $\mathcal E_K$ and $\mathcal Q_{D,K}$ for the
corresponding tautological quotient bundles.  For every
$\alpha\in\mathbb N^n$ with $|\alpha|=d$, one has
\begin{align}
 [x^\alpha]P_{\lambda/\mu}
 &=\int_{R^C_{\lambda/\mu,K}}
   \prod_{i=1}^n c_{\alpha_i}(\mathcal E_K),
 \label{appA:eq:fiberwise-C-coeff}\\
 [x^\alpha]Q_{\lambda/\mu}
 &=\int_{R^D_{\lambda/\mu,K}}
   \prod_{i=1}^n c_{\alpha_i}(\mathcal Q_{D,K}).
 \label{appA:eq:fiberwise-D-coeff}
\end{align}
Here the bundles are restricted to the indicated Richardson varieties.
Consequently, the type~$C$ and type~$D$ projective-bundle identities remain
valid over $K$.  Let $\eta_i$ denote the tautological quotient class on the
$i$th type~$C$ factor, and let $\xi_i$ denote the tautological quotient
class on the $i$th type~$D$ factor.  The following are identities in
$\mathbb Q[x_1,\ldots,x_n]$, obtained from the integer-valued Chow
intersection numbers by factorial normalization:
\begin{align}
 \frac{1}{D_C!}\int_{Y^C_{\lambda/\mu,n,K}}
 \left(\sum_{i=1}^n x_i\eta_i\right)^{D_C}
 &=\cN\!\left((x_1\cdots x_n)^{N-1}P_{\lambda/\mu}(x)\right),
 &D_C&=d+n(N-1),
 \label{appA:eq:fiberwise-C-projective}\\
 \frac{1}{D_D!}\int_{Y^D_{\lambda/\mu,n,K}}
 \left(\sum_{i=1}^n x_i\xi_i\right)^{D_D}
 &=\cN\!\left((x_1\cdots x_n)^NQ_{\lambda/\mu}(x)\right),
 &D_D&=d+nN.
 \label{appA:eq:fiberwise-D-projective}
\end{align}
\end{corollary}

\begin{proof}
For $N=1$, one has
$\mathbf X_{C,K}\simeq\mathbf X_{D,K}\simeq\mathbb P^1_K$,
$\mathcal E_K\simeq\mathcal O_{\mathbb P^1}(1)$, and
$\mathcal Q_{D,K}\simeq\mathcal O_{\mathbb P^1}(1)^{\oplus2}$.  Since
$\rho_1=(1)$, apart from the degree-zero identities the only nonconstant
coefficient identities are
\[
 P_{(1)}(x)=\sum_i x_i,\qquad Q_{(1)}(x)=2\sum_i x_i,
 \qquad
 \int_{\mathbb P^1}c_1(\mathcal E_K)=1,\qquad
 \int_{\mathbb P^1}c_1(\mathcal Q_{D,K})=2.
\]
The projective-bundle identities follow from the pushforward formula.  Thus
the corollary holds for $N=1$ over every algebraically closed field.  Assume
henceforth that $N\geq2$.
The ordinary Schur $P/Q$ Cauchy identities and their skew coefficient
extractions are identities over $\mathbb Z$.  Repeat the type~$C$ and
spinor Cauchy--Schubert arguments using
Lemma~\ref{appA:lem:integral-base-change}, the Schubert duality pairings and
Richardson class formulas in the geometric fibers, and the projection
formula.  This gives
\eqref{appA:eq:fiberwise-C-coeff} and
\eqref{appA:eq:fiberwise-D-coeff}.  The projective-bundle pushforward
formula is valid over every field and depends only on the ranks of the
bundles.  Applying it coefficientwise gives
\eqref{appA:eq:fiberwise-C-projective} and
\eqref{appA:eq:fiberwise-D-projective}.
\end{proof}

With respect to the standard positive systems, the positive roots in the
nilradicals of the two terminal maximal parabolics are
\begin{equation}
\label{appA:eq:nilradical-roots}
\begin{aligned}
  \Phi(\mathfrak u_C)
  &=\{2e_i:1\leq i\leq N\}
    \cup\{e_i+e_j:1\leq i<j\leq N\},\\
  \Phi(\mathfrak u_B)
  &=\{e_i:1\leq i\leq N\}
    \cup\{e_i+e_j:1\leq i<j\leq N\}.
\end{aligned}
\end{equation}
The tangent weights at the standard fixed points are the negatives of these
roots.  In the shifted-staircase parametrization, the diagonal boxes
correspond to $2e_i$ on the Lagrangian side and to $e_i$ on the spinor side;
the off-diagonal box $(i,j)$ corresponds to $e_i+e_j$ in both models.  This
is the root-theoretic origin of the distinction between the diagonal and
off-diagonal directions used below.

The two coefficient formulas differ only in their integral Schubert
normalizations.  The stable type~$C$ homomorphism of
Section~\ref{sec:typeC} sends $Q_\nu$ to $\sigma_\nu$, whereas
Proposition~\ref{appA:prop:spinor-specialization} sends $P_\nu$ to
$\tau_\nu$.  Since the $P$- and $Q$-bases are dual under the Schur $P/Q$
pairing, the same Cauchy--Schubert--Richardson coefficient extraction gives
\begin{equation}
\label{appA:eq:dual-skew-output}
  Q_\nu\longmapsto\sigma_\nu
  \quad\Longrightarrow\quad P_{\lambda/\mu},
  \qquad
  P_\nu\longmapsto\tau_\nu
  \quad\Longrightarrow\quad Q_{\lambda/\mu}.
\end{equation}
Thus the two Richardson proofs have the same formal structure, with the
homogeneous space and the integral Schubert normalization changed together.

\subsection{Characteristic two and Richardson degrees}
\label{appA:characteristic-two}

Let $k$ be an algebraically closed field of characteristic two and retain
$N\geq2$.  Write
\[
 G_B:=\mathbf X_{B,k},\qquad
 G_C:=\mathbf X_{C,k},\qquad
 G_D:=\mathbf X_{D,k}.
\]
The hyperplane construction of van Geemen--Marrani sends a maximal
isotropic $(N+1)$-plane $U$ in the even quadratic space to $U\cap H$ and
gives an isomorphism
\begin{equation}
\label{appA:eq:odd-even-char2}
 h:G_D\xrightarrow{\sim}G_B,
 \qquad
 \iota:=h^{-1}:G_B\xrightarrow{\sim}G_D;
\end{equation}
see~\cite[\S5.2]{VGM19}.  We spell out the compatibility with the
stratifications.  Choose split complete isotropic flags
$F_\bullet^D$ and $(F_\bullet^D)^{\mathrm{opp}}$ in the even quadratic space,
adapted to $H$ so that their members of dimensions at most $N$ lie in $H$,
and let $F_\bullet^B$ and $(F_\bullet^B)^{\mathrm{opp}}$ be the resulting
odd isotropic flags in $H$.  For $1\leq i\leq N$ and $U\in G_D$,
\[
 \dim(U\cap F_i^D)
 =\dim\bigl((U\cap H)\cap F_i^B\bigr),
\]
and the same equality holds for the opposite flags.  The strict-partition
Schubert and opposite-Schubert rank conditions are therefore identical
under $h$.  Thus $h$ and $\iota$ identify the Bruhat and opposite-Bruhat
cells with the same strict-partition indices.  Equivalently, the
isomorphism is equivariant for the split type~$B$ action, sends the chosen
base point to the chosen base point, and transports the selected Borel and
opposite-Borel orbits to their counterparts.  We use the Borels, opposite
Borels, maximal parabolics, and Weyl-group representatives determined by
these compatible pinnings.  Since all split models, flags, and pinnings are
defined over $\mathbb F_2$, we identify their Frobenius twists with the
original models and write $F_X$ for relative Frobenius.

The common Weyl group identifies the two root systems by the root-length-exchanging
bijection
\begin{equation}
\label{appA:eq:root-bijection}
 e_i\longleftrightarrow 2e_i,
 \qquad
 e_i\pm e_j\longleftrightarrow e_i\pm e_j.
\end{equation}
Choose the pinning-compatible very special isogenies
\[
 \pi_B:\mathbf G_{B,k}\longrightarrow\mathbf G_{C,k},
 \qquad
 \pi_C:\mathbf G_{C,k}\longrightarrow\mathbf G_{B,k}.
\]
For every root $\gamma$ and its partner $\bar\gamma$ under
\eqref{appA:eq:root-bijection}, the root-subgroup description of an isogeny
provides nonzero constants $c^B_\gamma,c^C_\gamma\in k^\times$ such that
\begin{align}
 \pi_B\bigl(x^B_\gamma(t)\bigr)
 &=x^C_{\bar\gamma}\bigl(c^B_\gamma
   t^{\epsilon_B(\gamma)}\bigr),
 &\epsilon_B(\gamma)&=
 \begin{cases}2,&\gamma\text{ short},\\1,&\gamma\text{ long},\end{cases}
 \label{appA:eq:root-map-B}\\
 \pi_C\bigl(x^C_\gamma(t)\bigr)
 &=x^B_{\bar\gamma}\bigl(c^C_\gamma
   t^{\epsilon_C(\gamma)}\bigr),
 &\epsilon_C(\gamma)&=
 \begin{cases}2,&\gamma\text{ short},\\1,&\gamma\text{ long}.
 \end{cases}
 \label{appA:eq:root-map-C}
\end{align}
Moreover,
\begin{equation}
\label{appA:eq:group-Frobenius-factorizations}
 \pi_B\circ\pi_C=F_{\mathbf G_C},
 \qquad
 \pi_C\circ\pi_B=F_{\mathbf G_B}.
\end{equation}
These are the root-subgroup description and Frobenius factorization of the
very special isogeny; see Maccan~\cite[equation~(1.2) and
Proposition~1.10]{Maccan26}.

The root bijection sends the terminal simple root to the terminal simple
root.  Hence $\pi_B$ and $\pi_C$ carry the chosen Borels and maximal
parabolics to their counterparts and induce homogeneous-space morphisms
\[
 \overline\beta:G_B\longrightarrow G_C,
 \qquad
 \overline\alpha:G_C\longrightarrow G_B.
\]
Define
\begin{equation}
\label{appA:eq:alpha-beta}
  \alpha:=\iota\circ\overline\alpha:G_C\longrightarrow G_D,
  \qquad
  \beta:=\overline\beta\circ\iota^{-1}:G_D\longrightarrow G_C.
\end{equation}
By Proposition~5.1 and Remark~5.2 of~\cite{VGM19}, the
homogeneous-space morphism induced by $\pi_B$ agrees, after transport by
$h$, with their hyperplane-projection morphism on the standard
alternating-matrix big cell.  Since the source is integral and the target is separated, agreement on that
dense cell implies global equality; this is the map $\beta$ in
\eqref{appA:eq:alpha-beta}.  The map $\alpha$ is the homogeneous-space
morphism induced by the reverse very special isogeny $\pi_C$.  On the
standard symmetric-matrix big cell it agrees with the
principal-minor/Pfaffian morphism of
van Geemen--Marrani~\cite[Theorem~2.3]{VGM19}; the global equivariant
closed-orbit interpretation is given in their discussion following
\cite[Remark~5.2]{VGM19}.  Since the source is integral and the target is separated, agreement on the
dense big cell identifies the two morphisms globally.

\begin{proposition}[The two special morphisms]
\label{appA:prop:special-morphisms}
The morphisms $\alpha$ and $\beta$ are finite flat and radicial and satisfy
\begin{equation}
\label{appA:eq:Frobenius-factorizations}
  \beta\circ\alpha=F_{G_C},
  \qquad
  \alpha\circ\beta=F_{G_D}.
\end{equation}
There is a one-dimensional $k$-vector space $L$ and an exact sequence
\begin{equation}
\label{appA:eq:bundle-comparison}
  0\longrightarrow L^\vee\otimes\mathcal O_{G_D}
  \longrightarrow\mathcal Q_D
  \longrightarrow\beta^*\mathcal E
  \longrightarrow0.
\end{equation}
Moreover, for every $\nu\subseteq\rho_N$,
\begin{equation}
\label{appA:eq:Schubert-pullbacks}
  \beta^*\sigma_\nu=2^{\ell(\nu)}\tau_\nu,
  \qquad
  \alpha^*\tau_\nu=2^{|\nu|-\ell(\nu)}\sigma_\nu.
\end{equation}
In particular,
\begin{equation}
\label{appA:eq:ambient-degrees}
  \deg\beta=2^N,
  \qquad
  \deg\alpha=2^{M-N}.
\end{equation}
\end{proposition}

\begin{proof}
The group factorizations~\eqref{appA:eq:group-Frobenius-factorizations}
induce the corresponding factorizations for $\overline\alpha$ and
$\overline\beta$.  The isomorphism $\iota$ is defined over $\mathbb F_2$
and therefore commutes with relative Frobenius, giving
\eqref{appA:eq:Frobenius-factorizations}.  These factorizations imply that both maps are surjective and universally
injective.  Thus they are radicial and quasi-finite.  They are projective, hence
proper, and a proper quasi-finite morphism is finite.  Let $f$ denote either map, let $x$ be a point of its source, put
$y=f(x)$, and write
\[
 A=\mathcal O_{\operatorname{target}(f),y},\qquad
 B=\mathcal O_{\operatorname{source}(f),x}.
\]
Both homogeneous spaces are smooth and equidimensional of dimension
$M=N(N+1)/2$, so $A$ and $B$ are regular local rings and in particular $B$
is Cohen--Macaulay.  The finite surjective radicial map $f$ is a universal
homeomorphism; hence $\dim A=\dim B$.  Thus $B$ is a finite
Cohen--Macaulay $A$-module of full dimension over the regular local ring
$A$.  Miracle flatness~\cite[Tag~00R4]{Stacks} shows that $B$ is flat over
$A$.  Since this holds at every point and $f$ is finite, both morphisms are
finite flat.

The hyperplane construction for $\beta$ supplies the bundle comparison.
Let $H$ be the distinguished hyperplane in the split even-dimensional
quadratic space and put $L=W/H$.  The intersection of the universal spinor
subbundle with $H$ is $\beta^*\mathcal U$, and projection to $L$ gives
\[
  0\longrightarrow\beta^*\mathcal U
  \longrightarrow\mathcal S
  \longrightarrow L\otimes\mathcal O_{G_D}
  \longrightarrow0.
\]
Dualizing and using $\mathcal E\simeq\mathcal U^\vee$ gives
\eqref{appA:eq:bundle-comparison}; see
\cite[\S\S5.1--5.3]{VGM19}.  Since its left-hand term is constant,
$c(\mathcal Q_D)=\beta^*c(\mathcal E)$.

Lemma~\ref{appA:lem:integral-base-change} and the modified Giambelli formulas
now give
\[
  \beta^*\sigma_\nu
  =\widetilde Q_\nu(\beta^*\mathcal E)
  =\widetilde Q_\nu(\mathcal Q_D)
  =2^{\ell(\nu)}\widetilde P_\nu(\mathcal Q_D)
  =2^{\ell(\nu)}\tau_\nu.
\]
Applying $\alpha^*$ and using
$\alpha^*\beta^*=F_{G_C}^*$ yields
\[
  2^{\ell(\nu)}\alpha^*\tau_\nu
  =2^{|\nu|}\sigma_\nu.
\]
Indeed, relative Frobenius is the identity on the underlying topological
space, so its pullback of a codimension-$j$ Schubert class defined over
$\mathbb F_2$ is a positive multiple of the same irreducible Schubert class.
The ambient Frobenius has degree $2^M$, its restriction to that Schubert
variety has degree $2^{M-j}$, and finite flat pull-push is multiplication by
$2^M$; hence the multiple is $2^j$.  The Chow ring is
torsion-free by Lemma~\ref{appA:lem:integral-base-change}, so the second
identity in~\eqref{appA:eq:Schubert-pullbacks} follows.  Finally, apply the
two pullback identities to the point class $\nu=\rho_N$.  Since
$\ell(\rho_N)=N$ and $|\rho_N|=M$, the projection formula gives
\eqref{appA:eq:ambient-degrees}.
\end{proof}

\begin{lemma}[Bruhat compatibility]
\label{appA:lem:Bruhat-compatibility}
The morphisms $\alpha$ and $\beta$ map every Bruhat cell and every opposite
Bruhat cell finitely, surjectively, and radicially onto the cell with the same
index in the common Weyl group.  Consequently, for every strict
$\mu\subseteq\lambda\subseteq\rho_N$,
\begin{equation}
\label{appA:eq:reduced-Richardson-preimages}
  (\beta^{-1}R^C_{\lambda/\mu})_{\mathrm{red}}
  =R^D_{\lambda/\mu},
  \qquad
  (\alpha^{-1}R^D_{\lambda/\mu})_{\mathrm{red}}
  =R^C_{\lambda/\mu}.
\end{equation}
\end{lemma}

\begin{proof}
Via the isomorphism~\eqref{appA:eq:odd-even-char2}, it is enough to work on
the type~$B$ and type~$C$ homogeneous spaces.  We give the argument for
$\overline\beta$; the proof for $\overline\alpha$ is the same with the two
root lengths interchanged.

Let $W^P$ denote the common set of minimal-length representatives for the
right cosets $W/W_P$; with this convention the Bruhat cell indexed by
$w\in W^P$ is
\[
 C_w^\bullet=\mathbf B_\bullet\dot w\mathbf P_\bullet/
              \mathbf P_\bullet.
\]
For $\bullet\in\{B,C\}$, put
\[
 \Psi_\bullet(w)
 :=\{\gamma\in\Phi_\bullet^+:w^{-1}\gamma\in\Phi_\bullet^-\}.
\]
For $w\in W^P$, the set $\Psi_\bullet(w)$ is the inversion set of
$w^{-1}$, and hence $|\Psi_\bullet(w)|=\ell(w)$.  The
root-length-exchanging bijection~\eqref{appA:eq:root-bijection} preserves
the chosen positive systems and intertwines the common Weyl-group action.
It therefore carries $\Psi_B(w)$ bijectively onto $\Psi_C(w)$.

Fix a reduced expression $w=s_{i_1}\cdots s_{i_\ell}$ in the common Weyl
group.  The associated ordered inversion roots are
\[
 \gamma_j=s_{i_1}\cdots s_{i_{j-1}}(\alpha_{i_j})
 \qquad(1\leq j\leq\ell).
\]
Taking the corresponding roots in types $B$ and $C$, multiplication gives
the standard root-subgroup coordinate isomorphisms
\begin{align}
 \mathbb A_k^\ell
 &\xrightarrow{\sim}C_w^B,
 &(t_1,\ldots,t_\ell)&\longmapsto
   x^B_{\gamma_1}(t_1)\cdots x^B_{\gamma_\ell}(t_\ell)
   \dot w\mathbf P_B,\label{appA:eq:cell-coordinates-B}\\
 \mathbb A_k^\ell
 &\xrightarrow{\sim}C_w^C,
 &(u_1,\ldots,u_\ell)&\longmapsto
   x^C_{\bar\gamma_1}(u_1)\cdots x^C_{\bar\gamma_\ell}(u_\ell)
   \dot w\mathbf P_C.
 \label{appA:eq:cell-coordinates-C}
\end{align}
These are the usual factorizations of
$U\cap\dot wU^-\dot w^{-1}$ into its root subgroups; see
Springer~\cite[\S8.3]{Springer98}.

The element $\pi_B(\dot w)$ represents the same Weyl-group element as the
chosen target representative $\dot w$ and therefore differs from it by a
torus factor.  Since the torus lies in $\mathbf P_C$, moving that factor
through the ordered root product only rescales the root coordinates by
nonzero constants.  Combining this observation with
\eqref{appA:eq:root-map-B}, the restriction of $\overline\beta$ to the cell
indexed by $w$ becomes, in the coordinates above,
\begin{equation}
\label{appA:eq:coordinatewise-cell-map}
 (t_j)_{j=1}^\ell
 \longmapsto
 \bigl(a_jt_j^{\epsilon_B(\gamma_j)}\bigr)_{j=1}^\ell,
 \qquad a_j\in k^\times,
 \quad\epsilon_B(\gamma_j)\in\{1,2\}.
\end{equation}
Every coordinate is a nonzero scalar multiple of either the identity of
$\mathbb A^1_k$ or relative Frobenius.  Hence
\eqref{appA:eq:coordinatewise-cell-map} is finite, surjective, and radicial.
Applying the same construction to the opposite pinning
$(\mathbf T_\bullet,\mathbf B_\bullet^-,
  \{x^\bullet_{-\gamma}\})$
proves the identical assertion for every opposite Bruhat cell.  Transporting
the type~$B$ cells through $\iota$ proves the stated cellwise assertions for
$\beta$; the argument using~\eqref{appA:eq:root-map-C} proves them for
$\alpha$.

The Schubert and opposite Schubert varieties are the closures of their
corresponding cells, equivalently the unions of cells with the appropriate
Bruhat inequalities.  Since $\alpha$ and $\beta$ are universal
homeomorphisms and preserve every cell index, the underlying topological
inverse image of each Schubert or opposite Schubert variety is its
counterpart with the same index.  Therefore its reduced scheme-theoretic
inverse image is exactly that counterpart.  Applying this to the two factors
defining a Richardson variety gives
\[
 \bigl(\beta^{-1}X^C_{\lambda^\vee}\bigr)_{\mathrm{red}}
   =X^D_{\lambda^\vee},
 \qquad
 \bigl(\beta^{-1}(X^C)^\mu\bigr)_{\mathrm{red}}
   =(X^D)^\mu,
\]
and similarly for $\alpha$.  Richardson varieties for opposite Borels are reduced and irreducible in
arbitrary characteristic; see Brion~\cite[Section~1.3]{Brion05} and the
Frobenius-splitting treatment in Brion--Kumar~\cite[Chapter~2]{BrionKumar05}.
Thus the Richardson intersections on both sides are reduced and integral,
and taking the intersection and then reducing proves
\eqref{appA:eq:reduced-Richardson-preimages}.
\end{proof}

\begin{remark}[Low ranks]
\label{appA:rem:low-rank}
The spinor specialization, Richardson formula, and projective-bundle model
of Subsections~\ref{appA:spinor-specialization}--\ref{appA:spinor-Richardson}
hold for every $N\geq1$.  The pinned argument above includes $N=2$; the
accidental isomorphism of the abstract root systems $B_2$ and $C_2$ does not
alter the long/short-root formulas.  When $N=1$, both homogeneous spaces are $\mathbb P^1$.  After compatible
identifications, $\alpha$ is an isomorphism and $\beta$ is relative
Frobenius.  Moreover, $\mathcal E\simeq\mathcal O_{\mathbb P^1}(1)$ and
$\mathcal Q_D\simeq\mathcal O_{\mathbb P^1}(1)^{\oplus2}$, while the
bundle comparison is the twisted Euler sequence
\[
 0\longrightarrow\mathcal O_{\mathbb P^1}
 \longrightarrow\mathcal O_{\mathbb P^1}(1)^{\oplus2}
 \longrightarrow\mathcal O_{\mathbb P^1}(2)
 =\beta^*\mathcal O_{\mathbb P^1}(1)\longrightarrow0.
\]
Hence $\deg\alpha=1=2^{M-N}$ and $\deg\beta=2=2^N$, and all Richardson
and projective-bundle formulas below reduce directly to the corresponding
statements on $\mathbb P^1$.
\end{remark}

Let
\[
  \beta_{\lambda/\mu}:R^D_{\lambda/\mu}\longrightarrow R^C_{\lambda/\mu},
  \qquad
  \alpha_{\lambda/\mu}:R^C_{\lambda/\mu}\longrightarrow R^D_{\lambda/\mu}
\]
be the induced morphisms between the reduced Richardson varieties.  For
$N=1$ they have the meaning described in Remark~\ref{appA:rem:low-rank}.

\begin{theorem}[Richardson degrees]
\label{appA:thm:Richardson-degrees}
Put
\[
  d=|\lambda|-|\mu|,
  \qquad
  r=\ell(\lambda)-\ell(\mu).
\]
Then
\begin{equation}
\label{appA:eq:restricted-degrees}
  \deg\beta_{\lambda/\mu}=2^r,
  \qquad
  \deg\alpha_{\lambda/\mu}=2^{d-r}.
\end{equation}
For every $\gamma\in\mathbb N^n$ with $|\gamma|=d$,
\begin{equation}
\label{appA:eq:Chern-comparison}
  \int_{R^D_{\lambda/\mu}}
  \prod_i c_{\gamma_i}(\mathcal Q_D|_{R^D_{\lambda/\mu}})
  =2^r
  \int_{R^C_{\lambda/\mu}}
  \prod_i c_{\gamma_i}(\mathcal E|_{R^C_{\lambda/\mu}}).
\end{equation}
Thus the standard identity
\begin{equation}
\label{appA:eq:PQ-degree}
  Q_{\lambda/\mu}
  =2^{\ell(\lambda)-\ell(\mu)}P_{\lambda/\mu}
\end{equation}
is geometrically accounted for by the degree of the spinor-to-Lagrangian
Richardson morphism.
\end{theorem}

\begin{proof}
For $N=1$, all assertions follow directly from
Remark~\ref{appA:rem:low-rank}; hence assume $N\geq2$.  By
Lemma~\ref{appA:lem:Bruhat-compatibility}, the restrictions are finite
dominant morphisms between integral varieties of the same dimension.  The
first identity in~\eqref{appA:eq:Schubert-pullbacks} gives
\begin{equation}
\label{appA:eq:beta-pullback-Richardson}
  \beta^*(\sigma_{\lambda^\vee}\sigma_\mu)\cap[G_D]
  =2^{\ell(\lambda^\vee)+\ell(\mu)}[R^D_{\lambda/\mu}]
  =2^{N-r}[R^D_{\lambda/\mu}],
\end{equation}
because $\ell(\lambda^\vee)=N-\ell(\lambda)$.  Write
\[
  (\beta_{\lambda/\mu})_*[R^D_{\lambda/\mu}]
  =e_\beta[R^C_{\lambda/\mu}].
\]
Pushing~\eqref{appA:eq:beta-pullback-Richardson} forward and using the
projection formula and $\deg\beta=2^N$ gives
\[
  2^{N-r}e_\beta[R^C_{\lambda/\mu}]
  =2^N[R^C_{\lambda/\mu}]
\]
in $A_d(G_C)$.  The class $[R^C_{\lambda/\mu}]$ is nonzero: for example,
its intersection with a suitable power of the ample generator of
$\operatorname{Pic}(G_C)$ has positive degree.  Since $A_d(G_C)$ is free
abelian, and hence torsion-free, by
Lemma~\ref{appA:lem:integral-base-change}, the displayed equality implies
\[
  2^{N-r}e_\beta=2^N.
\]
Therefore
\[
  e_\beta=2^r.
\]

Similarly, the second identity in~\eqref{appA:eq:Schubert-pullbacks} gives
\begin{equation}
\label{appA:eq:alpha-pullback-Richardson}
  \alpha^*(\tau_{\lambda^\vee}\tau_\mu)\cap[G_C]
  =2^{M-N-d+r}[R^C_{\lambda/\mu}],
\end{equation}
where the complement identities fixed in
Subsection~\ref{appA:spinor-specialization} give
\[
  |\lambda^\vee|-\ell(\lambda^\vee)
  +|\mu|-\ell(\mu)=M-N-d+r.
\]
Writing
\[
  (\alpha_{\lambda/\mu})_*[R^C_{\lambda/\mu}]
  =e_\alpha[R^D_{\lambda/\mu}]
\]
and pushing~\eqref{appA:eq:alpha-pullback-Richardson} forward gives
\[
  2^{M-N-d+r}e_\alpha[R^D_{\lambda/\mu}]
  =2^{M-N}[R^D_{\lambda/\mu}]
\]
in $A_d(G_D)$.  The class $[R^D_{\lambda/\mu}]$ is nonzero, since its
intersection with a suitable power of the ample generator of
$\operatorname{Pic}(G_D)$ has positive degree.  Since $A_d(G_D)$ is free
abelian by Lemma~\ref{appA:lem:integral-base-change}, the displayed equality
implies
\[
  2^{M-N-d+r}e_\alpha=2^{M-N}.
\]
Consequently,
\[
  e_\alpha=2^{d-r}.
\]
In particular,
\[
  e_\alpha e_\beta=2^d,
\]
as required by Frobenius on the $d$-dimensional Richardson variety.

Restrict~\eqref{appA:eq:bundle-comparison} to $R^D_{\lambda/\mu}$.  For
$\theta=\prod_i c_{\gamma_i}(\mathcal E|_{R^C_{\lambda/\mu}})$, the
projection formula gives
\[
  \int_{R^D_{\lambda/\mu}}\beta_{\lambda/\mu}^*\theta
  =\int_{R^C_{\lambda/\mu}}
    \theta\cap(\beta_{\lambda/\mu})_*[R^D_{\lambda/\mu}]
  =2^r\int_{R^C_{\lambda/\mu}}\theta,
\]
which is~\eqref{appA:eq:Chern-comparison}.  Equation
\eqref{appA:eq:PQ-degree} is the standard skew $P/Q$ normalization from
Section~\ref{sec:typeC}; the equality of its scalar with
$\deg\beta_{\lambda/\mu}$ is the geometric content of the comparison.
\end{proof}

\begin{example}[The shifted shape $(3,1)/(1)$]
\label{appA:ex:31-over-1}
Let $N=3$, $\rho_3=(3,2,1)$, $\lambda=(3,1)$, and $\mu=(1)$.
Then
\[
 \lambda^\vee=\rho_3\setminus\lambda=(2),
 \qquad
 d=|\lambda|-|\mu|=3,
 \qquad
 r=\ell(\lambda)-\ell(\mu)=1.
\]
The shifted skew diagram has the boxes $(1,2)$, $(1,3)$, and $(2,2)$;
there is one diagonal box and two off-diagonal boxes.

We first compute the skew functions.  From~\eqref{eq:Q-two-row} and
\eqref{eq:q-relations},
\[
 Q_{(2,1)}=q_2q_1-2q_3,
 \qquad
 Q_{(3,1)}=q_3q_1-2q_4,
 \qquad
 q_1^2=2q_2,
 \qquad
 q_2^2=2q_1q_3-2q_4.
\]
Consequently,
\[
 Q_{(1)}Q_{(3)}=Q_{(3,1)}+2Q_{(4)},
 \qquad
 Q_{(1)}Q_{(2,1)}=2Q_{(3,1)}.
\]
Since $(3)$ and $(2,1)$ are the only strict partitions of~$3$,
\eqref{eq:shifted-LR-def} and~\eqref{eq:skew-P-expansion} give
\begin{equation}\label{appA:eq:example-skew-expansion}
 P_{(3,1)/(1)}=P_{(3)}+2P_{(2,1)}.
\end{equation}
Using Lemma~\ref{lem:P-two-variable}, we therefore obtain
\begin{align}
 P_{(3,1)/(1)}(x,y)
 &=x^3+y^3+4x^2y+4xy^2,
 \label{appA:eq:example-P}\\
 Q_{(3,1)/(1)}(x,y)
 &=2x^3+2y^3+8x^2y+8xy^2,
 \label{appA:eq:example-Q}
\end{align}
where the second equality also follows from
Lemma~\ref{lem:skew-PQ-scaling}.

Let $R_C=R^C_{(3,1)/(1)}$ and $R_D=R^D_{(3,1)/(1)}$.  The coefficient
formulas~\eqref{appA:eq:fiberwise-C-coeff} and
\eqref{appA:eq:fiberwise-D-coeff} make the factor~$2^r=2$ visible in
individual Chern numbers:
\begin{align*}
 \int_{R_C}c_3(\mathcal E)&=1,
 &\int_{R_C}c_2(\mathcal E)c_1(\mathcal E)&=4,\\
 \int_{R_D}c_3(\mathcal Q_D)&=2,
 &\int_{R_D}c_2(\mathcal Q_D)c_1(\mathcal Q_D)&=8.
\end{align*}
Here and below the bundles are restricted to the indicated Richardson
varieties.

Retain the algebraically closed characteristic-two ground field fixed in
Subsection~\ref{appA:characteristic-two}.  Since $\lambda^\vee=(2)$ and
$\mu=(1)$, the pullback formulas
\eqref{appA:eq:Schubert-pullbacks} give
\[
 \beta^*(\sigma_{(2)}\sigma_{(1)})\cap[G_D]=4[R_D],
 \qquad
 \alpha^*(\tau_{(2)}\tau_{(1)})\cap[G_C]=2[R_C].
\]
Moreover, $M=|\rho_3|=6$, and hence
\[
 \deg\beta=2^3=8,
 \qquad
 \deg\alpha=2^{6-3}=8
\]
by~\eqref{appA:eq:ambient-degrees}.  Pushing the preceding two cycle
identities forward gives
\[
 4\deg(\beta_{(3,1)/(1)})=8,
 \qquad
 2\deg(\alpha_{(3,1)/(1)})=8.
\]
Thus
\[
 \deg(\beta_{(3,1)/(1)})=2=2^r,
 \qquad
 \deg(\alpha_{(3,1)/(1)})=4=2^{d-r}.
\]
In this example the spinor-to-Lagrangian degree records the single diagonal
box, while the reverse degree records the two off-diagonal boxes; their
product is $2^3$, the degree of Frobenius on the three-dimensional
Richardson variety.
\end{example}

The exponent $r$ is the number of diagonal boxes of the shifted skew diagram
$\lambda/\mu$, while $d-r$ is the number of off-diagonal boxes.  The
cellwise root-subgroup description in
Lemma~\ref{appA:lem:Bruhat-compatibility} shows that $\beta$ is Frobenius
on the diagonal root directions and $\alpha$ on the off-diagonal root
directions.  Thus
\eqref{appA:eq:restricted-degrees} is the global degree form of the
long/short-root exchange in~\eqref{appA:eq:nilradical-roots}.

\subsection{The projective-bundle bridge}
\label{appA:projective-bundle-bridge}

Continue over $k$.  Write
\[
  R_C=R^C_{\lambda/\mu},
  \qquad
  R_D=R^D_{\lambda/\mu},
  \qquad
  \mathcal E_R=\mathcal E|_{R_C},
  \qquad
  \mathcal Q_R=\mathcal Q_D|_{R_D}.
\]
Let
\[
  Y^C_{\lambda/\mu,n}
  :=\bigl(\mathbb P_{R_C}(\mathcal E_R)\bigr)^{\times_{R_C}n}
\]
be the type~$C$ projective-bundle construction of Section~\ref{sec:typeC}.  Its base
change along $\beta_{\lambda/\mu}$ is
\begin{equation}
\label{appA:eq:base-changed-C-model}
  \widetilde Y_C
  :=R_D\times_{R_C}Y^C_{\lambda/\mu,n}
  =\bigl(\mathbb P_{R_D}(\beta_{\lambda/\mu}^*\mathcal E_R)
    \bigr)^{\times_{R_D}n}.
\end{equation}
The quotient in \eqref{appA:eq:bundle-comparison} induces a closed immersion
\begin{equation}
\label{appA:eq:bridge-immersion}
  j:\widetilde Y_C\hookrightarrow Y^D_{\lambda/\mu,n}.
\end{equation}

\begin{theorem}[Tautological complete-intersection bridge]
\label{appA:thm:projective-bundle-bridge}
The morphism $j$ is a regular closed immersion of codimension $n$, and
\begin{equation}
\label{appA:eq:bridge-cycle}
  j_*[\widetilde Y_C]
  =\xi_1\cdots\xi_n\cap[Y^D_{\lambda/\mu,n}].
\end{equation}
Let
\begin{align*}
  V_C(x)&:=\frac{1}{D_C!}
  \int_{Y^C_{\lambda/\mu,n}}
  \left(\sum_i x_i\eta_i\right)^{D_C},
  &D_C&=d+n(N-1),\\
  V_D(x)&:=\frac{1}{D_D!}
  \int_{Y^D_{\lambda/\mu,n}}
  \left(\sum_i x_i\xi_i\right)^{D_D},
  &D_D&=d+nN,
\end{align*}
where $\eta_i$ are the tautological classes on the type~$C$ factors.  These
are elements of $\mathbb Q[x_1,\ldots,x_n]$, defined by Chow intersection
numbers on varieties over $k$.  Then
\begin{equation}
\label{appA:eq:volume-bridge}
  \partial_{x_1}\cdots\partial_{x_n}V_D(x)=2^rV_C(x).
\end{equation}
Equivalently, as an identity of polynomials over $\mathbb Q$,
\begin{equation}
\label{appA:eq:normalized-bridge}
  \partial_{x_1}\cdots\partial_{x_n}
  \cN\!\left((x_1\cdots x_n)^NQ_{\lambda/\mu}(x)\right)
  =2^r\cN\!\left((x_1\cdots x_n)^{N-1}
  P_{\lambda/\mu}(x)\right).
\end{equation}
\end{theorem}

\begin{proof}
For $N=1$, the twisted Euler sequence in
Remark~\ref{appA:rem:low-rank} identifies
$\mathbb P_{R_D}(\beta_{\lambda/\mu}^*\mathcal E_R)$ with the zero scheme of
a tautological section on $\mathbb P_{R_D}(\mathcal Q_R)$, hence with a
regular Cartier divisor; on the $n$-fold fiber product the corresponding
sections form a regular sequence.  The same cycle,
derivative, and pushforward argument therefore proves the assertion.  We may
assume $N\geq2$ for the remainder of the proof.

On a factor $p:\mathbb P_{R_D}(\mathcal Q_R)\to R_D$, the composite
\[
  p^*(L^\vee\otimes\mathcal O_{R_D})
  \longrightarrow p^*\mathcal Q_R
  \longrightarrow\mathcal O(1)
\]
is a section of $\mathcal O(1)\otimes L$.  Its zero scheme parametrizes the
one-dimensional quotients of $\mathcal Q_R$ that annihilate the distinguished
constant line subbundle, and is therefore
$\mathbb P_{R_D}(\beta_{\lambda/\mu}^*\mathcal E_R)$.  Since $L$ is constant,
the divisor class of this section is $\xi=c_1(\mathcal O(1))$.  Zariski-locally
on $R_D$ the vector-bundle extension splits, and the section is one relative
homogeneous coordinate.  On every affine chart meeting its zero scheme this
coordinate is a polynomial variable over the possibly singular base ring,
hence a non-zero-divisor.  On the $n$-fold fiber product the corresponding
coordinates lie in distinct factors and form a regular sequence.  This proves
that $j$ is regular and gives \eqref{appA:eq:bridge-cycle}.

Since $D_D-n=D_C$, differentiation and
\eqref{appA:eq:bridge-cycle} give
\[
  \partial_{x_1}\cdots\partial_{x_n}V_D(x)
  =\frac{1}{D_C!}
   \int_{\widetilde Y_C}
   \left(\sum_i x_i j^*\xi_i\right)^{D_C}.
\]
The morphism $\widetilde Y_C\to Y^C_{\lambda/\mu,n}$ is finite dominant of
degree $2^r$, and $j^*\xi_i$ is the pullback of $\eta_i$.  The projection
formula proves \eqref{appA:eq:volume-bridge}.  The fiberwise type~$C$ and type~$D$ projective-bundle identities in
Corollary~\ref{appA:cor:fiberwise-formulas} then give
\eqref{appA:eq:normalized-bridge}.  Algebraically, that normalized identity
also follows from $Q_{\lambda/\mu}=2^rP_{\lambda/\mu}$; the content here is
that its derivative is realized by the regular complete intersection
\eqref{appA:eq:bridge-cycle}.
\end{proof}

\begin{remark}
For $N\geq2$, the common Bruhat indexing and the dual $P/Q$ Schubert
normalizations do not furnish an isomorphism between $G_C$ and $G_D$ over
$\mathbb C$.  The case $N=1$ is the exceptional identification described in
Remark~\ref{appA:rem:low-rank}.  The special comparison morphisms $\alpha$
and $\beta$ occur in characteristic two.  There the two numerical differences in the
projective-bundle constructions have direct geometric meanings: the factor
$2^r$ is the degree of $\beta_{\lambda/\mu}$, while the shift changes from
$N-1$ to $N$ because \eqref{appA:eq:bundle-comparison} contains one additional
constant line bundle.
\end{remark}


\begin{thebibliography}{99}

\bibitem{ABW82}
K.~Akin, D.~A.~Buchsbaum, and J.~Weyman,
\emph{Schur functors and Schur complexes},
Adv. Math. \textbf{44} (1982), no.~3, 207--278.

\bibitem{ATZ24}
S.~An, K.~Tung, and Y.~Zhang,
\emph{Postnikov--Stanley polynomials are Lorentzian},
arXiv:2412.02051v2, 2024.

\bibitem{BGST26}
N.~Bergeron, L.~Gagnon, H.~Spink, and V.~Tewari,
\emph{The quasisymmetric Grassmannian},
arXiv:2604.24903v1, 2026.

\bibitem{BH20}
P.~Br\"and\'en and J.~Huh,
\emph{Lorentzian polynomials},
Ann. of Math. (2) \textbf{192} (2020), no.~3, 821--891.

\bibitem{BKT17}
A.~S.~Buch, A.~Kresch, and H.~Tamvakis,
\emph{A Giambelli formula for isotropic Grassmannians},
Selecta Math. (N.S.) \textbf{23} (2017), no.~2, 869--914.

\bibitem{Brion05}
M.~Brion,
\emph{Lectures on the geometry of flag varieties},
in Topics in cohomological studies of algebraic varieties,
Trends Math., Birkh\"auser, Basel, 2005, pp.~33--85.

\bibitem{BrionKumar05}
M.~Brion and S.~Kumar,
\emph{Frobenius splitting methods in geometry and representation theory},
Progress in Mathematics, vol.~231, Birkh\"auser Boston, Boston, MA, 2005.

\bibitem{CCPS26}
S.~H.~Chan, H.~Chen, I.~Pak, and D.~Soskin,
\emph{Correlation inequalities for Schur positivity},
arXiv:2606.06688v1, 2026.

\bibitem{ChinQin25}
T.~Chin and D.~Qin,
\emph{Lorentzian Symmetric Polynomials},
arXiv:2510.07819v1, 2025.

\bibitem{CidRuiz26}
Y.~Cid-Ruiz,
\emph{Mixed Segre zeta functions and their log-concavity},
arXiv:2507.06424v2, to appear in Algebra \& Number Theory.

\bibitem{FultonIT}
W.~Fulton,
\emph{Intersection theory}, second ed.,
Ergebnisse der Mathematik und ihrer Grenzgebiete, vol.~2,
Springer-Verlag, Berlin, 1998.

\bibitem{FultonPragacz}
W.~Fulton and P.~Pragacz,
\emph{Schubert varieties and degeneracy loci},
Lecture Notes in Mathematics, vol.~1689, Springer-Verlag, Berlin, 1998.

\bibitem{FultonYT}
W.~Fulton,
\emph{Young tableaux},
London Mathematical Society Student Texts, vol.~35,
Cambridge University Press, Cambridge, 1997.

\bibitem{GHMSW25}
L.~Grund, J.~Huh, M.~Micha\l ek, H.~S\"uss, and B.~Wang,
\emph{Linear operators preserving volume polynomials},
arXiv:2506.22415v1, 2025.

\bibitem{Hironaka64}
H.~Hironaka,
\emph{Resolution of singularities of an algebraic variety over a field of characteristic zero. I, II},
Ann. of Math. (2) \textbf{79} (1964), 109--203, 205--326.

\bibitem{HMMSD22}
J.~Huh, J.~Matherne, K.~M\'esz\'aros, and A.~St.~Dizier,
\emph{Logarithmic concavity of Schur and related polynomials},
Trans. Amer. Math. Soc. \textbf{375} (2022), no.~6, 4411--4427.

\bibitem{HuhVolume26}
J.~Huh,
\emph{Volume polynomials},
arXiv:2601.13249v4, 2026.

\bibitem{JiangLi23}
C.~Jiang and Z.~Li,
\emph{Algebraic reverse Khovanskii--Teissier inequality via Okounkov bodies},
Math. Z. \textbf{305} (2023), no.~2, Paper No.~26, 14 pp.

\bibitem{JingLiu26}
N.~Jing and N.~Liu,
\emph{A skew Murnaghan--Nakayama rule for Hopf dual pairs},
arXiv:2606.15138v2, 2026.

\bibitem{KT03}
A.~Kresch and H.~Tamvakis,
\emph{Quantum cohomology of the Lagrangian Grassmannian},
J. Algebraic Geom. \textbf{12} (2003), no.~4, 777--810.

\bibitem{KT04}
A.~Kresch and H.~Tamvakis,
\emph{Quantum cohomology of orthogonal Grassmannians},
Compos. Math. \textbf{140} (2004), no.~2, 482--500.

\bibitem{LLS11}
T.~Lam, A.~Lauve, and F.~Sottile,
\emph{Skew Littlewood--Richardson rules from Hopf algebras},
Int. Math. Res. Not. IMRN 2011, no.~6, 1205--1219.

\bibitem{LPP07}
T.~Lam, A.~Postnikov, and P.~Pylyavskyy,
\emph{Schur positivity and Schur log-concavity},
Amer. J. Math. \textbf{129} (2007), no.~6, 1611--1622.

\bibitem{LeNguyen26}
T.~Le and S.~Nguyen,
\emph{Skew Hives, Skew Skeps, Skew Schur Log-Concavity},
arXiv:2608.13544v1, 2026.

\bibitem{LehmannXiao17}
B.~Lehmann and J.~Xiao,
\emph{Correspondences between convex geometry and complex geometry},
Epijournal G\'eom. Alg\'ebrique \textbf{1} (2017), Art.~6, 29 pp.

\bibitem{Macdonald}
I.~G.~Macdonald,
\emph{Symmetric functions and Hall polynomials}, second ed.,
Oxford Mathematical Monographs, Oxford University Press, New York, 1995.

\bibitem{Maccan26}
M.~Maccan,
\emph{Projective homogeneous varieties of Picard rank one in small characteristic},
Ann. Sc. Norm. Super. Pisa Cl. Sci. (5) \textbf{27} (2026), no.~2, 613--705.

\bibitem{MarshallOlkinArnold11}
A.~W.~Marshall, I.~Olkin, and B.~C.~Arnold,
\emph{Inequalities: theory of majorization and its applications}, second ed.,
Springer Series in Statistics, Springer, New York, 2011.

\bibitem{McNamara08}
P.~R.~W.~McNamara,
\emph{Necessary conditions for Schur-positivity},
J. Algebraic Combin. \textbf{28} (2008), no.~4, 495--507.

\bibitem{MTY19}
C.~Monical, N.~Tokcan, and A.~Yong,
\emph{Newton polytopes in algebraic combinatorics},
Selecta Math. (N.S.) \textbf{25} (2019), no.~5, Paper No.~66, 37 pp.

\bibitem{Murota03}
K.~Murota,
\emph{Discrete convex analysis},
SIAM Monographs on Discrete Mathematics and Applications,
Society for Industrial and Applied Mathematics, Philadelphia, 2003.

\bibitem{NST24}
P.~Nadeau, H.~Spink, and V.~Tewari,
\emph{The geometry of quasisymmetric coinvariants},
arXiv:2410.12643v2, 2024.

\bibitem{PR97}
P.~Pragacz and J.~Ratajski,
\emph{Formulas for Lagrangian and orthogonal degeneracy loci; the $\widetilde Q$-polynomial approach},
Compos. Math. \textbf{107} (1997), no.~1, 11--87.

\bibitem{Rado52}
R.~Rado,
\emph{An inequality},
J. London Math. Soc. \textbf{27} (1952), 1--6.

\bibitem{SGA3}
M.~Demazure and A.~Grothendieck,
\emph{Sch\'emas en groupes. III: Structure des sch\'emas en groupes r\'eductifs},
S\'eminaire de G\'eom\'etrie Alg\'ebrique du Bois Marie 1962--64 (SGA~3),
Documents Math\'ematiques, vol.~8, Soci\'et\'e Math\'ematique de France,
Paris, 2011.

\bibitem{ShawVW07}
S.~Shaw and S.~van Willigenburg,
\emph{Multiplicity free expansions of Schur $P$-functions},
Ann. Comb. \textbf{11} (2007), no.~1, 69--77.

\bibitem{Speyer26}
D.~E.~Speyer,
\emph{$L$-log-concavity and a proof of the conjecture of Lam, Postnikov
and Pylyavskyy},
arXiv:2601.05007v2, 2026.

\bibitem{Springer98}
T.~A.~Springer,
\emph{Linear algebraic groups}, second ed.,
Progress in Mathematics, vol.~9, Birkh\"auser Boston, Boston, MA, 1998.

\bibitem{Stacks}
The Stacks Project Authors,
\emph{The Stacks Project},
\url{https://stacks.math.columbia.edu}.

\bibitem{Stembridge89}
J.~R.~Stembridge,
\emph{Shifted tableaux and the projective representations of symmetric groups},
Adv. Math. \textbf{74} (1989), no.~1, 87--134.

\bibitem{VGM19}
B.~van Geemen and A.~Marrani,
\emph{Lagrangian Grassmannians and spinor varieties in characteristic two},
SIGMA \textbf{15} (2019), Paper No.~064, 22 pp.

\bibitem{Zhang26}
P.~B.~Zhang,
\emph{Normalized skew Schur polynomials are Lorentzian},
arXiv:2608.12266v1, 2026.

\end{thebibliography}
\end{document}